\documentclass[11pt]{article}
\usepackage{amsfonts}
\usepackage{amsfonts,latexsym}
\usepackage{amsmath}
\usepackage{amsthm}
\usepackage{amssymb,bbm}
\usepackage{mathrsfs}
\usepackage{color}
\usepackage[]{cite}
\usepackage{graphicx}
\usepackage{amsmath}
\usepackage{authblk}
\usepackage{indentfirst}
\usepackage{setspace}
\usepackage[colorlinks,linkcolor=blue,citecolor=blue]{hyperref}
\usepackage{geometry}
\allowdisplaybreaks
\newtheorem{theorem}{\indent Theorem}[section]
\newtheorem{proposition}[theorem]{\indent Proposition}

\newtheorem{condition}[theorem]{\indent Condition}
\newtheorem{lemma}[theorem]{\indent Lemma}

\begin{document}	
	\title{\bf Ergodicity for stochastic 2D Boussinesq equations with a highly degenerate pure jump L\'{e}vy noise}

	\author[a]{Jianhua Huang}
    \author[b]{Xuhui Peng}
    \author[a]{Xue Wang \thanks{Corresponding author: wangxuexue1997@163.com} }
    \author[c]{Jiangwei Zhang }
	\affil[a]{\small  \it College of Science, National University of Defense Technology, Changsha, Hunan 410073, P.R.China}
	\affil[b]{\small  \it MOE-LCSM, School of Mathematics and Statistics, Hunan Normal University, Changsha, Hunan 410081, P.R.China}
    \affil[c]{\small  \it Institute of Applied Physics and Computational Mathematics, Beijing 100088, P.R.China}

	\date{}
		
\maketitle

\begin{abstract}
This study aims to analyze the ergodicity for stochastic 2D Boussinesq equations and explore the impact of a highly degenerate pure jump L\'{e}vy noise acting only in the temperature equation, where this noise could appear on only a few Fourier modes.
By leveraging the equi-continuity of the semigroup established through Malliavin calculus and an analysis of stochastic calculus, together with the weak irreducibility of the solution process, we prove the existence and uniqueness of the invariant measure. Moreover, we overcome the main challenge of establishing time asymptotic smoothing properties of the Markovian dynamics corresponding to this system by conducting spectral analysis of the Malliavin covariance matrix.

\medskip

	\noindent \textbf{Keywords:} Ergodicity, Stochastic Boussinesq equation, Pure jump L\'{e}vy noise, Degenerate noise, Malliavin calculus.
\end{abstract}

{\hspace*{2mm} AMS Subject Classification: 60H15, 60G51, 76D06, 76M35.}

\linespread{1.2}
\section{Introduction}
\numberwithin{equation}{section}
In this paper, we consider the ergodicity of the following stochastic Boussinesq equations with a degenerate pure jump noise on a 2D torus $\mathbb{T}^2=[-\pi,\pi]^2=\mathbb{R}^2/(2\pi\mathbb{Z}^2)$:
\begin{eqnarray}\label{1.1}
\left\{
\begin{aligned}
&\mathrm{d}u+(u\cdot\nabla u)\mathrm{d}t=(-\nabla p+\nu_1\Delta u+\mathbf{g}\theta)\mathrm{d}t,\\
&\mathrm{d}\theta+(u\cdot\nabla\theta)\mathrm{d}t=\nu_2\Delta\theta \mathrm{d}t+\sigma_{\theta}\mathrm{d}W_{S_t},\\
&\nabla\cdot u=0,
\end{aligned}
\right.
\end{eqnarray}
where $u=(u_1,u_2)$ denotes the velocity field, $\theta$ represents temperature of a viscous incompressible fluid, and $p$ is the pressure; the parameters $\nu_1,\nu_2>0$ are respectively the kinematic viscosity and thermal diffusivity of the fluid and $\mathbf{g}=(0,g)^T$ with $g\neq0$ is the product of the gravitational constant and the thermal expansion coefficient. The spatial variable $x=(x_1,x_2)$ belongs to $\mathbb{T}^2$. That is, we impose periodic boundary conditions in space.
It is worth noting that the buoyancy term $\mathbf{g}\theta$ serves as the sole mechanism for transferring the influence of the stochastic forcing from the temperature equation \eqref{1.1}$_2$ to the momentum equations \eqref{1.1}$_1$. $W_{S_t}$ is a subordinated Brownian motion that will be introduced later.

From a physical perspective, the deterministic Boussinesq system (i.e. $\sigma_{\theta}=0$ in \eqref{1.1}) arises in simplified models of geophysics, including atmospheric fronts, oceanic currents, and Rayleigh-Bénard convection and so on. This emergence occurs when rotation and stratification are significant factors, highlighting its importance in the study of phenomena like cold fronts and the jet stream.
In particular, the relevant research on convective phenomena can be found in \cite{Constantin}.
In recent decades, significant progress has been achieved in the study of well-posedness, regularity and long-term behavior of solutions, particularly focusing on two-dimensional situations with fractional or fully dissipative characteristics, see e.g., \cite{Abidi-2007,Cao-2013,Foias-2001,Larios-2013,Pu-2011} and the references therein. However, there has been a growing interest in exploring the influence of stochastic effects on fluid models. This is primarily because random phenomena should be as closely aligned with real-life facts and situations as possible. Consequently, researchers have begun to gradually investigate the stochastic Boussinesq equations.

In fact, the stochastic (modified or fractional) Boussinesq equations driven by random external forces have received widespread attention in recent  years, including cases of additive noise, L\'{e}vy noise, nonlinear multiplicative noise,  as well as transport noise etc., see e.g., \cite{Duan-2009,Lin-2024,Luo-2021, Shen-2015,Shen-2014,Wu-2024,Zhang-2024}.
Specifically, Duan and Millet in \cite{Duan-2009} studied the large deviation
 principle of stochastic Boussinesq. Shen and Huang in \cite{Shen-2015,Shen-2014} researched the well-posedness of the stochastic modified Boussinesq equations driven by L\'{e}vy noise and fractional Brownian motion with $H\in(\frac{1}{2},1)$, respectively.  Recently, the first author of the paper and his collaborator in \cite{Wu-2024} investigated the well-posedness and limit behavior of the stochastic fractional Boussinesq equations (SFBE) driven by nonlinear noise, and the ergodicity under additive noise was also proved. In particular,
 Zhang and Huang
of \cite{Zhang-2024} explored the existence of weak solutions, quantitative convergence rates and central limit theorem for 3D stochastic modified fractional Boussinesq equations driven by transport noise on a 3D torus. Furthermore, the authors in \cite{Lin-2024} proved the global existence for the stochastic Boussinesq equations on $\mathbb{T}^3$ with transport noise and rough initial data, and the author in \cite{Luo-2021} considered the weak convergence of the solutions of 2D stochastic equations with transport noise to those of the deterministic equations.

However, when the noise is moderate or highly degenerate, the complex coupling relationship between the noise and physical quantities of the fluid itself (such as velocity, pressure, density, etc.) makes the study of fluid equations driven by degenerate noise extremely challenging, see e.g., \cite{Peng-Huang-2020,Peng-2020}. This leads to relatively limited research on the stochastic Boussinesq equation with a degenerate noise, see e.g., \cite{Foldes-2015}, especially regarding highly degenerate noise. Moreover, from the mathematical point of view, the study of ergodic properties for infinite dimensional systems is a field that has been intensely studied over the past two decades but is still in its infancy compared to the corresponding theory for finite-dimensional systems. And the uniqueness of invariant measure and the ergodicity of the randomly forced dissipative partial differential equations (PDEs) driven by degenerate noise have been problems of central concern for many years.  What is the motivation behind studying the ergodicity of the Boussinesq equation driven by degenerate L\'{e}vy noise?

First of all, for the stochastic partial differential equations (SPDEs) driven by highly degenerate Wiener type, its ergodicity and related problems have been intensively studied recently, see \cite{Foldes-2015,MH-2006,MH-2011,Odasso-2007,Odasso-2008}.
Hairer et al. \cite{MH-2006,MH-2008} considered stochastic 2D Navier-Stokes equations on a torus driven by degenerate additive noise. They established an exponential mixing property of the solution of the vorticity formulation for the 2D stochastic Navier-Stokes equations by using Malliavin calculus, although the noise is extremely degenerate. Subsequently, F\"{o}ldes et al. \cite{Foldes-2015} studied the Boussinesq equations in the presence of a degenerate stochastic forcing acting only in the temperature equation and established the ergodic and mixing properties. Peng et al. \cite{Peng-2020} proved the ergodicity and exponential mixing of the real Ginzburg-Landau equation with a degenerate stochastic forcing acting only in four directions, under the case of without using the vorticity transformation. Meanwhile, they \cite{Peng-Huang-2020} also established the existence, uniqueness and exponential attraction properties of an invariant measure for the Magneto-Hydrodynamic equations with degenerate stochastic forcing acting only in the magnetic equation. Recently, Huang et al. \cite{Huang-2022} studied the asymptotic properties of invariant measures for SFBE with degenerate random forcing.

Secondly, under certain conditions, pure Brownian motion perturbations have many drawbacks in capturing large motions and unpredictable events. The emergence of L\'{e}vy type perturbations is to reproduce the characteristics of these natural phenomena in some real-world models according to empirical researches, for instance, statistical physics (\cite{Novikov-1965}), climatology (\cite{Imkeller-2006}) and mathematics of finance (\cite{Kuchler-2013}), etc.
Therefore, SPDEs driven by L\'{e}vy processes are an important and meaningful aspect of mathematical analysis, with several significant works \cite{Brzezniak-2022, Dong-2011} emerging in this field. However, to our knowledge, there seem to be no results on ergodicity for highly degenerate L\'{e}vy noises results.
In \cite{Sun-2019}, Sun et al. proves that a family of SPDEs driven by highly degenerate pure jump L\'{e}vy noises are exponentially mixing using a coupling method, however, they required that the nonlinear term was globally bounded, which cannot be applied to certain concrete equations, such as stochastic Navier-Stokes equations and stochastic Boussinesq equations.
Fortunately, inspired by the techniques and method in the recent article \cite{Peng-2024}, which studied the ergodicity for stochastic 2D Navier-Stokes equations driven by a highly degenerate pure jump L\'{e}vy noise by using Malliavin calculus and anticipating stochastic calculus, we attempt to use this thought to establish the ergodicity of the Boussinesq equation.

We apply the a curl operator to the first equation of \eqref{1.1} to obtain an equivalent, simpler vorticity formulation. To close the system \eqref{1.1}, we calculate $u$ from $w$ by the Biot-Savart law, that is $u=\mathcal{K}*w$, where $\mathcal{K}$ is the Biot-Savart kernel, so that $\nabla^{\bot}\cdot u=\partial_{x}u_{2}-\partial_{y}u_{1}=:w$ and $\nabla\cdot u=0$.
Then, by a standard calculation, it follows that
\begin{equation}\label{2.1}
\begin{aligned}
\mathrm{d} w&+(B_1(u,w)-\nu_{1}\Delta w)\mathrm{d}t
=g\partial_{x}\theta \mathrm{d}t,\\
\mathrm{d}\theta&+(B_2(u,\theta)-\nu_{2}\Delta\theta)\mathrm{d}t\\
&=\alpha_1\cos(x_1)\mathrm{d}W_{S_t}^1+\alpha_{2}\sin(x_1)\mathrm{d}W_{S_t}^2
+\alpha_3\cos(x_2)\mathrm{d}W_{S_t}^3+\alpha_{4}\sin(x_2)\mathrm{d}W_{S_t}^4,
\end{aligned}
\end{equation}
where $B_1(u,w)=(u\cdot\nabla)w$ and $B_2(u,\theta)=(u\cdot\nabla)\theta$, $\{\alpha_k\}_{k\in\{1,2,3,4\}}$ is a sequence of non-zero constants, $x_1,x_2$ are the two components of spatial variable,
and $\{W_{S_t}^k\}_{k\in\{1,2,3,4\}}$ is a 4-dimensional pure jump L\'{e}vy process. 
The system \eqref{2.1} is posed on $\mathbb{T}^2\times(0,\infty)$.

Note that in \eqref{2.1} we consider a highly degenerate stochastic forcing $\sigma_\theta \mathrm{d}W_{S_t}=\alpha_1\cos(x_1)\mathrm{d}W_{S_t}^1+\alpha_{2}\sin(x_1)\mathrm{d}W_{S_t}^2
+\alpha_3\cos(x_2)\mathrm{d}W_{S_t}^3+\alpha_{4}\sin(x_2)\mathrm{d}W_{S_t}^4$, which acts only on a few Fourier modes and exclusively through the temperature equation.
This shall lead to some difficulties that are different from those in \cite{Foldes-2015} and \cite{Peng-2024}. To be more precise,

$(i)$ the interaction between the nonlinear and stochastic terms in \eqref{2.1} is delicate, and leads us to develop a novel infinite-dimensional form of the H\"{o}rmander bracket condition. However, due to the different noise levels and the introduction of a new generalized condition (Lemma \ref{lemma3.2}) for the L\'{e}vy process, our analysis of the H\"{o}rmander bracket condition is more concise;

$(ii)$ for proving the ergodicity of \eqref{2.1}, we need to
use Malliavin calculus and anticipating stochastic calculus to establish equi-continuity of semigroup, and provide the weak irreducibility for the solution process of \eqref{2.1};

$(iii)$ due to the strong intensity and the unboundedness of the jumps, we introduce a stopping time sequence $\eta_n$ (see \eqref{2.7} and \eqref{2.8}) and new preliminary estimates to establish the uniqueness of invariant measures, which differs from the proofs in \cite{MH-2006,MH-2011} that strongly depend on the Girsanov transformation and exponential martingale estimate of the Gaussian noise.

By introducing some new technical methods and estimates, we overcome the above difficulties and obtain the main result of this paper concerning the ergodicity of the stochastic 2D Boussinesq equation \eqref{2.1} driven by a highly degenerate pure jump L\'{e}vy noise.
\begin{theorem}[Ergodicity]\label{ergodicity}
Under Condition \ref{Condition-2.1} and Condition \ref{Condition-2.2}, the Boussinesq equation with a highly degenerate L\'{e}vy noise admits that there exists a unique invariant measure $\mu^*$ for the system \eqref{2.1}, i.e., $\mu^*$ is a unique probability measure on $H$ such that $P_t^*\mu^*=\mu^*$ for every $t\geq0$.
\end{theorem}

This work enhances our understanding of stochastic Boussinesq equations, emphasizing their significance in the context of highly degenerate pure jump L\'{e}vy noise that affects only the temperature equation. In particular, to prove Theorem \ref{ergodicity}, we establish the equi-continuity of the semigroup, apply the tools of Malliavin calculus to transform the gradient bound into a control problem involving the linearization of equations \eqref{2.1}, and prove some weak irreducibility for the solution process of \eqref{2.1}.
It is crucial to clarify that this control problem can be addressed by determining suitable spectral bounds for the Malliavin covariance matrix $\mathcal{M}$. Furthermore, we extend Malliavin's techniques to prove the asymptotic strong Feller property for certain infinite-dimensional stochastic systems.

The paper is organized as follows. In Section 2, we give functional setting, notations and the Markovian framework, as well as the control problem. In Section 3, we give some moment estimates of the solution
for the stochastic Boussinesq system with pure jump L\'{e}vy noise. Section 4 is devoted to the invertibility of the Malliavin matrix and a dissipative property of the Boussinesq system. We will state the main theorem in Section 5, i.e. the existence and uniqueness of the invariant measure for the system \eqref{2.1} (see Theorem \ref{ergodicity}), to achieve this goal, we need to achieve some weak irreducibility for the solution process of the system \eqref{2.1} (see Proposition \ref{irre}), and establish an equi-continuity of the semigroup, the so-called e-property (see Proposition \ref{propo1.4}) by using Malliavin calculus and anticipating stochastic calculus.
%we establish weak irreducibility and ergodicity for \eqref{2.1}.

\section{Preliminaries}

\subsection{Notations}

We introduce a functional setting for the system \eqref{2.1}. The phase space is composed of mean zero, square integrable functions
\begin{equation}
H:=\left\{U:=(w,\theta)^{T}\in(L^2(\mathbb{T}^2))^2:
\int_{\mathbb{T}^2}w \mathrm{d}x=\int_{\mathbb{T}^2}\theta \mathrm{d}x=0\right\}.
\end{equation}
Fix the trigonometric basis:
\begin{equation}
\sigma_{k}^{0}(x):=\left(0,\cos(k\cdot x)\right)^{T},\quad\sigma_{k}^{1}(x):=\left(0,\sin(k\cdot x)\right)^{T},
\end{equation}
and
\begin{equation}\label{psi}
\psi_{k}^{0}(x):=\left(\cos(k\cdot x),0\right)^{T},\quad\psi_{k}^{1}(x):=\left(\sin(k\cdot x),0\right)^{T},
\end{equation}
where $k\in\mathbb{Z}^2,x\in \mathbb{T}^2$.
We denote
\begin{equation}
\mathbb{Z}_+^2:=\left\{j=(j_1,j_2)\in\mathbb{Z}_0^2:
j_1>0\mathrm{~or~}j_1=0,~j_2>0\right\}.
\end{equation}
For any $N \geq 1$, define
\begin{equation*}
H_N:=\mathrm{span}\left\{\sigma_k^m,\psi_k^m:|k|\leq N, m\in\{0,1\}\right\},
\end{equation*}
along with the associated projection operators
\begin{equation*}
P_N:H\to H_N \text{ the orthogonal projection onto }H_N,\quad Q_N:=I-P_N.
\end{equation*}
Note that $Q_N$ maps $H$ onto  $\  \mathrm{span}\left\{\sigma_k^m,\psi_k^m:|k|>N,m\in\{0,1\}\right\}$.

Now we provide the details for the subordinated Brownian motion $W_{S_t}$. Let $(\mathbb{W},\mathbb{H},\mathbb{P}^{\mu_{\mathbb{W}}})$ be the classical Wiener space, where $\mathbb{W}$ is the space of all continuous functions from $\mathbb{R}^+$ to $\mathbb{R}^d$, with vanishing values at starting point 0,  $\mathbb{H}\subseteq\mathbb{W}$ is the Cameron-Martin space composed of all absolutely continuous functions with
square integrable derivatives, and $\mathbb{P}^{\mu_{\mathbb{W}}}$ is the Wiener measure so that the coordinate
process $W_t(\text{w}):=\text{w}_t$ is a $d$-dimensional standard Brownian motion. Let $\mathbb{S}$ be the space of all c\`{a}dl\`{a}g increasing functions $\ell$ from $\mathbb{R}^+$ to $\mathbb{R}^d$ with $\ell_0= 0$. Suppose that $\mathbb{S}$ is endowed with the Skorohod metric and a probability measure $\mathbb{P}^{\mu_{\mathbb{S}}}$ so that the coordinate process $S_t(\ell):=\ell_t$ is a pure jump subordinator with L\'{e}vy
measure $\nu_S$ satisfying
\begin{equation*}
 \int_{0}^{\infty}(1\wedge u)\nu_S(\mathrm{d}u)<\infty.
\end{equation*}
Then we consider the following product probability space $(\Omega,\mathcal{F},\mathbb{P}):=(\mathbb{W}\times \mathbb{S},\mathcal{B}(\mathbb{W})\times \mathcal{B}(\mathbb{S}),\mathbb{P}^{\mu_{\mathbb{W}}}\times \mathbb{P}^{\mu_{\mathbb{S}}})$, and define for $\omega=(\text{w},\ell)\in \mathbb{W}\times\mathbb{S}$, $L_t(\omega):=\text{w}_{\ell_t}$. Then $(L_t=W_{S_t})_{t\geq 0}$ is a $d$-dimensional pure jump L\'{e}vy process with L\'{e}vy measure $\nu_L$ given by
\begin{equation}\label{1.8}
\nu_{L}(E)=\int_{0}^{\infty}\int_{E}(2\pi u)^{-\frac{d}{2}}e^{-\frac{|z|^{2}}{2u}}\mathrm{d}z\nu_{S}(\mathrm{d}u),\quad E\in\mathcal{B}(\mathbb{R}^{d}).
\end{equation}

Let $N_L(\mathrm{d}t,\mathrm{d}z)$ denote the Poisson random measure associated with the L\'{e}vy process $L_t=W_{S_t}$, i.e.,
\begin{equation*}
N_L((0,t]\times U)=\sum_{s\leq t}I_U(L_t-L_{t-}),U\in\mathcal{B}(\mathbb{R}^d\setminus\{0\}).
\end{equation*}
Let $\tilde{N}_L(\mathrm{d}t,\mathrm{d}z)$ be the compensated Poisson random measure associated with $N_L(\mathrm{d}t, \mathrm{d}z)$, i.e.,
\begin{equation*}
\tilde{N}_L(\mathrm{d}t,\mathrm{d}z)=N_L((0,t]\times U)-\mathrm{d}t\nu_L(\mathrm{d}z).
\end{equation*}
Similar notation can also be applied to $N_S(\mathrm{d}t,\mathrm{d}z)$ and $\tilde{N}_S(\mathrm{d}t, \mathrm{d}z)$. Given that the measure $\nu_L(\mathrm{d}z)$ is symmetric, the L\'{e}vy process $L_t$  can be represented as follows:
\begin{equation*}
L_t=\int_0^t\int_{\mathbb{R}^d\setminus \{0\}}
z\tilde{N}_L(\mathrm{d}s,\mathrm{d}z).
\end{equation*}
Let $F=F(\mathrm{w},\ell)$ be a random variable on the space $(\Omega,\mathcal{F}, \mathbb{P})$.
We denote the expectation of $F$ as $\mathbb{E}^{\mu_{\mathbb{W}}}F$ when the element $\ell$ is held fixed.
This notation signifies computing the expected value of $F$ with respect to the measure $\mu_{\mathbb{W}}$, while treating $\ell$ as a constant, i.e.,
\begin{equation*}
\mathbb{E}^{\mu_\mathbb{W}}F=\int_\mathbb{W}F(\mathrm{w},\ell)
\mathbb{P}^{\mu_\mathbb{W}}(\mathrm{dw}).
\end{equation*}
The notation $\mathbb{E}^{\mu_{\mathbb{S}}}F$ has the similar meaning. We use $\mathbb{E}[F]$ to denote the expectation of $F$ under the measure $\mathbb{P}=\mathbb{P}^{\mu_{\mathbb{W}}}\times\mathbb{P}^{\mu_{\mathbb{S}}}$.

The filtration used in this paper is
$$
\mathcal{F}_t:=\sigma(W_{S_t},S_s:s\leq t).
$$
For any fixed $\ell\in\mathbb{S}$ and positive number $a = a(\ell)$ which is independent of the Brownian motion $(W_t)_{t\geq 0}$, the filtration $\mathcal{F}_a^W$ is defined by
$$
\mathcal{F}_a^W:=\sigma(W_s:s\leq a).
$$
If $\tau:\Omega\rightarrow[0,\infty]$ is a stopping time with respect to the filtration $\mathcal{F}_t$, $\mathcal{F}_{\tau}$ denotes the past $\sigma$-field defined by
$$
\mathcal{F}_{\tau}=\{A\in\mathcal{F}:\forall t\geq0,A\cap\{\tau\leq t\}\in\mathcal{F}_t\}.
$$

Denote by $M_b(H)$ and $C_b(H)$ the spaces of bounded measurable and bounded continuous real valued functions on $H$, respectively.

Fix $U=U(t,U_0)=U(t,U_0,W_{S_t})$, and note that the solution defines a Markov process. Accordingly the transition function associated to \eqref{2.14} is given by
\begin{equation*}
  P_t(U_0,A)=\mathbb{P}(U(t,U_0)\in A)~\text{for all} ~U_0\in H,
\end{equation*}
for every Borel set $A\subseteq H$, and
\begin{equation*}
  P_t\Phi(U_0)=\int_H \Phi(U)P_t(U_0,\mathrm{d}U),~P^*_t \mu(A)=\int_H P_t(U_0,A)\mu (\mathrm{d}U_{0}),
\end{equation*}
for every $\Phi: H\rightarrow\mathbb{R}$ and probability measure $\mu$ on $H$.

\subsection{Abstract framework and conditions}

In order to rewrite \eqref{2.1} in a functional form, we introduce the following abstract operators associated to the various terms in the equation.

The higher order Sobolev spaces are denoted as
\begin{equation*}
H^s:=\left\{U:=(w,\theta)^T\in(W^{s,2}(\mathbb{T}^2))^2:\int_{\mathbb{T}^2}w \mathrm{d}x=\int_{\mathbb{T}^2}\theta \mathrm{d}x=0\right\}~\text{for any}~s\geq0,
\end{equation*}
where $W^{s,2}(\mathbb{T}^2)$  is classical Sobolev–Slobodeckii space, and $H^s$ is equipped with the norm
\begin{equation}\label{norm}
\|U\|_{s}^n:=\|U\|_{H^s}^n:=\left(\frac{\nu_1\nu_2}{g^2}\|w\|_{W^{s,2}}^{2n}
+\|\theta\|_{W^{s,2}}^{2n}\right)^{\frac{1}{2}},~n\geq 1,
\end{equation}
when $s=0$, we denote $\|\cdot\|_0:=\|\cdot\|$.

For $U=(w,\theta)$ and $\tilde{U}=(\tilde{w},\tilde{\theta})$, let $A:D(A)=H^2\subset H\rightarrow H$ be the linear, symmetric, positive definite operator defined by
\begin{equation*}
AU:=(-\nu_1\Delta w,-\nu_2\Delta\theta)^T,
\end{equation*}
for any $U\in H^2$. Note that $A$ is the infinitesimal generator of a semigroup $e^{-tA}:H\rightarrow H^2$.

For the inertial (non-linear) terms define $B:H^1\times H^1\rightarrow H$ by
\begin{equation}
B(U,\tilde{U})
:=((\mathcal{K}*w)\cdot\nabla\tilde{w},
(\mathcal{K}*w)
\cdot\nabla\tilde{\theta})^T,
\end{equation}
for $U,\tilde{U}\in H^1$. It is well known that $\|\mathcal{K}*w\|_{H^s}\leq C\|w\|_{H^{s-1}}$, and since $H^2\hookrightarrow L^{\infty}$, we indeed obtain that $B(U,\tilde{U})\in H$. Also set $B(U)=B(U,U)$. For any $(s_1,s_2,s_3)\in \mathbb{R}_{+}^3$ with $\sum_{i=1}^{3}s_i \geq 1$ and $(s_1,s_2,s_3)\neq (1,0,0),(0,1,0),(0,0,1)$, regarding some calculations, this article will frequently use the following relationships
\begin{equation}\label{2.2}
\begin{aligned}
\langle B_j(u,v),w\rangle & =-\langle B_j(u,w),v\rangle,\quad\mathrm{if~}\nabla\cdot u=0, \\
\left|\langle B_j(u,v),w\rangle\right|& \leq C\|u\|_{s_{1}}\|v\|_{1+s_{2}}\|w\|_{s_{3}},~j=1,2,\\
\|\mathcal{K}u\|_{\alpha}& =\|u\|_{\alpha-1}, \\
\|w\|_{1/2}^{2}& \leq\|w\|\|w\|_{1}.
\end{aligned}
\end{equation}

Finally, for the `buoyancy term' define $G :H^1 \rightarrow H$ by
\begin{equation}
GU:=(g\partial_x \theta,0)^T,
\end{equation}
for $U\in H^1$.

The equations \eqref{2.1} may be written as an
abstract stochastic evolution equation on $H$
\begin{equation}\label{2.14}
\mathrm{d}U=(-AU-B(U)+ GU)\mathrm{d}t+\sigma_\theta \mathrm{d}W_{S_t},\quad U(0)=U_0=(w_0,\theta_0)\in H,
\end{equation}
set $F(U):=-AU-B(U)+GU$, \eqref{2.14} can be rewritten as
\begin{equation}\label{2.15}
\mathrm{d}U=F(U)\mathrm{d}t+\sigma_\theta \mathrm{d}W_{S_t},\quad U(0)=U_0.
\end{equation}

We introduce a finite set $\mathcal{Z}\subset \mathbb{Z}_+^2$ which represents the forced directions in Fourier space. The driving noise process $W_{S_t}:=(W_{S_t}^{k,m})_{k\in\mathcal{Z},m\in\{0,1\}}$ is a $d:=2\cdot |\mathcal{Z}|-$dimensional subordinated Brownian motion.
%defined relative to a filtered probability space $(\Omega,\mathcal{F},\{\mathcal{F}_t\}_{t\geq 0},\mathbb{P})$.
% and we refer to the resulting tuple $\mathcal{S}=(\Omega,\mathcal{F},\{\mathcal{F}_t\}_{t\geq 0},\mathbb{P},W)$ as a stochastic basis.
Let
$\{e_k^m\}_{k\in \mathcal{Z},m\in\{0,1\}}$ be the standard basis of $\mathbb{R}^{2|\mathcal{Z}|}$ and $\{\alpha_k^m\}_{k\in \mathcal{Z},m\in\{0,1\}}$ be a sequence of non-zero numbers. We define a linear map  $\sigma_{\theta}:\mathbb{R}^{2|\mathcal{Z}|}\rightarrow H$ such that
\begin{equation*}
\sigma_{\theta}e_k^m:=\alpha_k^m\sigma_k^m.
\end{equation*}
Denote the Hilbert-Schmidt norm of $\sigma_\theta$ by
$$
\|\sigma_{\theta}\|^2:=\|\sigma_{\theta}^* \sigma_{\theta}\|
=\sum_{k\in \mathcal{Z},m\in\{0,1\}}(\alpha_k^m)^2.
$$
We consider stochastic forcing of the following form
\begin{equation}\label{sigma}
\sigma_{\theta}\mathrm{d}W_{S_t}:=\sum_{k\in \mathcal{Z},m\in\{0,1\}}\alpha_k^m\sigma_k^m\mathrm{d}W_{S_t}^{k,m}.
\end{equation}
Recall that the set $\mathcal{Z}\subset\mathbb{Z}^2_+$ represents the forced directions in Fourier space.
Note that there are many choices of finite $\mathcal{Z}$ that can accommodate our main result.
Therefore, let the set $\mathcal{Z}\subset\mathbb{Z}^2_*$ be a generator, if any element of $\mathbb{Z}^2_*$
is a finite linear combination of elements of $\mathcal{Z}$ with integer coefficients.
In the remainder of this article, we assume that the following two conditions have been satisfied.

\begin{condition}\label{Condition-2.1}
The set $\mathcal{Z}\subset \mathbb{Z}_+^2$ appeared in \eqref{sigma} is finite and contains the two coordinate vectors $(1,0)$ and $(0,1)$, i.e., $\{(1,0),(0,1)\}\subseteq\mathcal{Z}$; in particular, $\mathcal{Z}$ is a generator in the above sense.	
\end{condition}

The set
\begin{equation*}
  \mathcal{Z}=\{(1,0), (0,1)\}\subset\mathbb{Z}_*^2:=\mathbb{Z}^2\backslash\{(0,0)\}
\end{equation*}
is an example satisfying this condition.

%\textbf{Condition 2.2}
\begin{condition}\label{Condition-2.2}
	Assume that $\nu_S$ satisfies
	\begin{equation*}
		\int_0^\infty(e^{\zeta u}-1)\nu_S(\mathrm du)<\infty~\text{for some}~\zeta>0
	\end{equation*}
	and
	\begin{equation*}
		\nu_S((0,\infty))=\infty.
	\end{equation*}
\end{condition}

\subsection{Elements in Malliavin calculus}

In this subsection, we explain how the estimate on $\nabla P_t\Phi$ can be translated to a control problem through the Malliavin integration by parts formula, which is closely related to the proof of Proposition \ref{propo1.4}.

Let $U=U(\cdot,U_0)$ be the solution of \eqref{2.14} and let $d:=2\cdot|\mathcal{Z}|$. Then for any $\Phi\in C_b^1(H)$, and $\xi\in H$ we have
\begin{equation}\label{P.1}
\nabla P_t\Phi(U_0)\cdot\xi
=\mathbb{E}(\nabla\Phi(U(t,U_0))\cdot J_{0,t}\xi),\quad t\geq0.
\end{equation}

For $\xi=(\xi^{(1)},\xi^{(2)})\in H,t\geq s\geq 0$, the Jacobian $J_{s,t}\xi$ is actually the unique solution of the linearised problem:
\begin{equation}\label{P.2}
\partial_tJ_{s,t}\xi=-AJ_{s,t}\xi-\nabla B(U)J_{s,t}\xi +GJ_{s,t}\xi,\quad J_{s,s}\xi=\xi,
\end{equation}
and $\nabla B(U)J_{s,t}\xi:=B(U,J_{s,t}\xi)+B(J_{s,t}\xi,U)$.

Let $J_{s,t}^{(2)}:H\rightarrow \mathcal{L}(H,\mathcal{L}(H))$ be the second derivative of $U$ with respect to an initial condition $U_0$. Observe that, for fixed $U_0$ in the directions of $\phi,\psi\in H$, 
\begin{equation}
\begin{aligned}
&\partial_t J_{s,t}^{(2)}(\phi,\psi)
=-AJ_{s,t}^{(2)}(\phi,\psi)-\nabla B(U)J_{s,t}^{(2)}(\phi,\psi)+\nabla B(J_{s,t}\phi)J_{s,t}\psi+GJ_{s,t}^{(2)}(\phi,\psi),\\
& J_{s,s}^{(2)}(\phi,\psi)=0.
\end{aligned}
\end{equation}
For any $0\leq t\leq T$ and $\xi\in H$, let $K_{t,T}$ be the adjoint of $J_{t,T}$, i.e., $\varrho_t:= K_{t,T}\xi=J^*_{t,T}\xi$ satisfies the following `backward' system
\begin{equation}\label{backward}
\partial_t\varrho_t=A\varrho_t+(\nabla B(U(t)))^*\varrho_t-G^*\varrho_t=-(\nabla F(U))^*\varrho_t,
\quad\varrho_T=\xi,
\end{equation}
where $\langle(\nabla B(U))^*\varrho,\psi \rangle=\langle\varrho,\nabla B(U)\psi\rangle$ and $\nabla B(U)\psi=B(U,\psi)+B(\psi,U)$, $G^*$ is the adjoint of the operator $G$.

The crucial step in estimating $\|\nabla P_t\Phi(U_0)\|$ is to ``approximately remove" the gradient from $\Phi$ in \eqref{P.1}.
As such we seek to (approximately) identify $J_{0,t}\xi$ with a Malliavin derivative of some suitable random process and integrate by parts, in the Malliavin sense.

For given $\ell\in\mathbb{S}$, $t>0$,
let $\Psi(t,W)$ be a $\mathcal{F}_{\ell_t}^W$-measurable random variable. For $v\in L^2([0,\ell_t];\mathbb{R}^d)$, the Malliavin derivative of $\Psi$ in the direction $v$ is defined by
\begin{equation}
\mathcal{D}^{v}\Psi(t,W)
=\lim\limits_{\varepsilon\to0}\frac{1}{\varepsilon}
\left(\Psi(t,U_0,W+\varepsilon\int_0^\cdot vds)-\Psi(t,U_0,W)\right),
\end{equation}
where the limit holds almost surely (e.g. see \cite{MH-2006,MH-2011} for Hilbert space case).
In the definition of Malliavin derivative, the element $\ell$ is taken as
fixed. Then, $\mathcal{D}^{v}U_t$ satisfies the following equation:
\begin{equation}
\mathrm{d}\mathcal{D}^v U_t
=-A\mathcal{D}^v U_t\mathrm{d}t
-\nabla B(U_t)\mathcal{D}^v U_t \mathrm{d}t
+G\mathcal{D}^v U_t \mathrm{d}t
+\sigma_{\theta}d\left(\int_0^{\ell_t}v_s \mathrm{d}s\right).
\end{equation}

Define
$$\gamma_{u}=\inf\{t\geq0, S_t(\ell)\geq u\}.
$$
By the formula of constant variations or Fubini's theorem, for any $v\in L^2([0,\ell_t];\mathbb{R}^d)$, we can deduce
\begin{equation}\label{P.4}
\begin{aligned}
\mathcal{D}^v U_t
&=\int_0^t J_{r,t}\sigma_{\theta} \mathrm{d}\big(\int_0^{\ell_r}v_s\mathrm{d}s\big)
=\sum_{r\le t}J_{r,t}\sigma_{\theta} \int_{\ell_{r-}}^{\ell_r}v_s \mathrm{d}s\\
&=\sum_{r\le t}\int_{\ell_{r-}}^{\ell_{r}}J_{\gamma_{u},t}\sigma_{\theta} v_{u}\mathrm{d}u
=\int_{0}^{\ell_{t}}J_{\gamma_{u},t}\sigma_{\theta}  v_{u}\mathrm{d}u,
\end{aligned}
\end{equation}
where the second equality is obvious, and we have used the fact that  $\gamma_u=r, u\in(\ell_{r-},\ell_{r})$ in the third equality.

Inspired by \eqref{P.4}, for any $s\leq t$ and $\ell\in\mathbb{S}$, we define the random operator $\mathcal{A}_{s,t}:L^2([\ell_{s},\ell_{t}];\mathbb{R}^d)\rightarrow H$ by
\begin{equation}\label{P.7}
\mathcal{A}_{s,t}v
=\int_{\ell_s}^{\ell_t}J_{\gamma_u,t}\sigma_{\theta}v_u \mathrm{d}u,
~v\in L^2([\ell_s,\ell_t];\mathbb{R}^d).
\end{equation}
For any $s<t$, let $\mathcal{A}_{s,t}^*:H\rightarrow L^2([\ell_{s},\ell_{t}];\mathbb{R}^d)$ be the adjoint of $\mathcal{A}_{s,t}$, defined by
\begin{equation*}\label{P.8}
\mathcal{A}_{s,t}^* \phi
=(\alpha_k^m\langle \phi,J_{\gamma_u,t}\sigma_k^m \rangle)_{k\in\mathcal{Z},m\in\{0,1\} },~~\text{for any}~\phi \in H,~u\in[\ell_{s},\ell_{t}].
\end{equation*}

Let $\rho(t):= J_{0,t}\xi-\mathcal{A}_{0,t}v$, then it satisfies
\begin{equation}\label{2.12}
\partial_t\rho=-A\rho-\nabla B(U)\rho+G\rho-\sigma_\theta v,\quad\rho(0)=\xi.
\end{equation}

The object $\mathcal{M}_{s,t}:=\mathcal{A}_{s,t}\mathcal{A}_{s,t}^*: H\rightarrow H$, referred to as the Malliavin covariance matrix, plays an
important role in the theory of stochastic analysis. By a simple calculation, we have (c.f. \cite[Lemma 2.2]{Zhang-2014})
\begin{equation}\label{Malliavin}
\langle\mathcal{M}_{s,t}\phi,\phi\rangle
=\sum_{k\in\mathcal{Z},m\in\{0,1\}}(\alpha_k^m)^{2}\int_{s}^{t}\langle K_{r,t}\phi,\sigma_k^m\rangle^{2}\mathrm{d}\ell_{r}.
\end{equation}

If the invertibility of $\mathcal{M}_{s,t}$ can be proven, then we can choose appropriate direction $v$ with suitable bounds to determine an exact control of $\rho_{t}$, indicating that the Markov semigroup is smoothing in finite time (i.e. it is strongly Feller).

We may select the control $v$ in \eqref{2.12} from the use of determining modes when a sufficient number of directions in Fourier space are forced, or if $\nu_1$, $\nu_2$ are big enough, see \cite{Foias-1967} or, more recently, \cite{Foias-2001}.
Consequently, for finite dimensional situations, the invertibility of the Malliavin matrix is well understood. However in infinite dimensions, this invertibility is considerably more difficult to determine and might not hold in general.
Therefore, following the insights in \cite{MH-2006}, we can use a Tikhonov regularization of the Malliavin matrix to construct a control $v$, and corresponding $\rho_t$, which will be provided in detail in Section 4.

\section{Moment estimates on $J_{s,t}\xi$, $J_{s,t}^{(2)}(\phi,\psi)$, $U_t$}
In this section, we will provide some estimates for $J_{s,t}$, $J_{s,t}^{(2)}$, $\mathcal{A}_{s,t}$, $\mathcal{A}_{s,t}^*$, $\mathcal{M}_{s,t}$ and their Malliavin derivatives, as well as for the solution $U_t$ of the system \eqref{2.1}. These estimates will be used in subsequent sections.

\begin{lemma}\label{lemma2.4}
  There exists a constant $C_0$ only depending on $\nu=\min\{\nu_1,\nu_2\}$ and $g$ such that for any $\xi,\phi,\psi\in H$, $a>0$ and $0\leq s\leq t\leq T$, $J_{s,t}$ and $J_{s,t}^{(2)}$ satisfy almost surely
\begin{equation}\label{2.16}
\sup_{t\in[s,T]}\|J_{s,t}\xi\|^{2}
\leq C_{0}\|\xi\|^{2}e^{C_{0}
\int_{s}^{T}(\|U_{r}\|_{1}^{4/3}+1)\mathrm{d}r},
\end{equation}
\begin{equation}\label{2.17}
\int_{s}^{t}\|J_{s,r}\xi\|_{1}^{2}\mathrm{d}r
\leq C_{0}\|\xi\|^{2}e^{C_{0}
\int_{s}^{t}(\|U_{r}\|_{1}^{4/3}+1)\mathrm{d}r},
\end{equation}
\begin{equation}\label{4.12}
\begin{aligned}
&\nu\int_s^t e^{-mr}\|J_{s,r}\xi\|_1^2\mathrm{d}r
\leq C_0\|\xi\|^2e^{\int_s^te^{-mr}[C_0(\|U_r\|_1^{4/3}+1)-m]\mathrm{d}r},
\end{aligned}
\end{equation}
\begin{equation}\label{2.18}
\sup\limits_{t\in[s,T]}\|J_{s,t}^{(2)}(\phi,\psi)\|^2
\leq C_{0}\|\phi\|^2\|\psi\|^2e^{C_{0}\int_{s}^{T}
(\|U_{r}\|_{1}^{4/3}+1)\mathrm{d}r}.
\end{equation}
Moreover, for each $0\leq s\leq T$ and $\kappa$, we have the following estimate
\begin{equation}\label{2.19}
\begin{aligned}
\|J_{s,T}\xi\|^{2}
\leq& C_{0}\exp\Big\{\frac{\nu\kappa}{120}\int_{s}^{T}
(\|U_{r}\|_{1}^{2}+1)e^{-\nu(T-r)+8\mathfrak{B}_{0}\kappa(\ell_{T}-\ell_{r})}\mathrm{d}r\\
&+C_{\kappa}\int_{s}^{T}e^{2\nu(T-r)-16\mathfrak{B}_{0}\kappa(\ell_{T}-\ell_{r})}
\mathrm{d}r\Big\}\|\xi\|^{2},
\end{aligned}
\end{equation}
where $C_0$ is taken from \eqref{2.16}-\eqref{2.18}, and $C_\kappa$ is a positive constant depending on $\kappa,\nu,g$.
\end{lemma}
\begin{proof} By \eqref{2.2}, we have
 \begin{equation}\label{2.20}
\langle B(J_{s,t}\xi,U_{t}),J_{s,t}\xi\rangle
\leq C\|U_{t}\|_{1}\|J_{s,t}\xi\|_{\frac{1}{2}}\|J_{s,t}\xi\| \leq \frac{\nu}{4}\|J_{s,t}\xi\|_{1}^{2}
+C\|U_{t}\|_{1}^{4/3}\|J_{s,t}\xi\|^{2},
\end{equation}
and
 \begin{equation}\label{2.21}
\langle G(J_{s,t}\xi),J_{s,t}\xi\rangle
\leq |g|\|J_{s,t}\xi^{(2)}\|_{1}\|J_{s,t}\xi\|\leq \frac{\nu}{4}\|J_{s,t}\xi\|_{1}^{2}+C\|J_{s,t}\xi\|^{2},
\end{equation}
where $J_{s,t}\xi=(J_{s,t}\xi^{(1)},J_{s,t}\xi^{(2)})^T$, and obviously, there holds $\|J_{s,t}\xi^{(2)}\|^2_1 \leq \|J_{s,t}\xi\|_1^2$. Applying the chain rule to $\|J_{s,t}\xi\|^2$, we obtain
\begin{equation}\label{2.22}
\mathrm{d}\|J_{s,t}\xi\|^2
\leq -\nu\|J_{s,t}\xi\|_1^2 \mathrm{d}t+C_0(\|U_t\|_1^{4/3}+1)\|J_{s,t}\xi\|^2 \mathrm{d}t,
\end{equation}
which implies
\begin{equation*}
\|J_{s,t}\xi\|^2
\leq C\|\xi\|^2e^{C_0\int_{s}^{t}(\|U_{r}\|_{1}^{4/3}+1)\mathrm{d}r},
\end{equation*}
and
\begin{equation*}
\begin{aligned}
\nu\int_{s}^t\|J_{s,r}\xi\|_1^2\mathrm{d}r
&\leq\|\xi\|^2+C\int_{s}^{t}(\|U_{r}\|_{1}^{4/3}+1)\|J_{s,r}\xi\|^2\mathrm{d}r\\
&\leq \|\xi\|^2+C\|\xi\|^2\int_s^t(\|U_{r}\|_{1}^{4/3}+1)\mathrm{d}r e^{C\int_{s}^{t}(\|U_{r}\|_{1}^{4/3}+1)\mathrm{d}r}\\
&\leq C\|\xi\|^2e^{C\int_{s}^{t}(\|U_{r}\|_{1}^{4/3}+1)\mathrm{d}r}.
\end{aligned}
\end{equation*}
From \eqref{2.22}, for the constant $m>0$, we can deduce that
\begin{align*}
&\frac{\mathrm{d}e^{-mt}\|J_{s,t}\xi\|^2}{\mathrm{d}t}
=-me^{-mt}\|J_{s,t}\xi\|^2+e^{-mt}\frac{\mathrm{d}\|J_{s,t}\xi\|^2}{\mathrm{d}t}\\
&\leq -me^{-mt}\|J_{s,t}\xi\|^2-\nu e^{-mt}\|J_{s,t}\xi\|_1^2
+e^{-mt}C_0(\|U_t\|_1^{4/3}+1)\|J_{s,t}\xi\|^2\\
&=e^{-mt}\left(C_0(\|U_t\|_1^{4/3}+1)-m\right)\|J_{s,t}\xi\|^2
-\nu e^{-mt}\|J_{s,t}\xi\|^2_1.
\end{align*}
By Gronwall's inequality, we derive
\begin{equation*}
\begin{aligned}
&e^{-mt}\|J_{s,t}\xi\|^2+\nu\int_s^t e^{-mr}\|J_{s,r}\xi\|_1^2\mathrm{d}r\\
&\leq C_0\|\xi\|^2e^{\int_s^te^{-mr}[C_0(\|U_r\|_1^{4/3}+1)-m]\mathrm{d}r}.
\end{aligned}
\end{equation*}
Then we complete \eqref{2.16}-\eqref{4.12}.

Turn to prove \eqref{2.18}, by using \eqref{2.2} we obtain
\begin{equation*}
\langle B(J_{s,t}^{(2)}(\phi,\psi),U_{t}),J_{s,t}^{(2)}(\phi,\psi)\rangle
\leq\frac{\nu}{8}\|J_{s,t}^{(2)}(\phi,\psi)\|_{1}^{2}
+C\|U_{t}\|_{1}^{4/3}\|J_{s,t}^{(2)}(\phi,\psi)\|^{2},
\end{equation*}
and
\begin{equation*}
\begin{aligned}
\langle B(J_{s,t}\phi,J_{s,t}\psi),J_{s,t}^{(2)}(\phi,\psi)\rangle
&\leq C\|J_{s,t}^{(2)}(\phi,\psi)\|_{1}\|J_{s,t}\phi\|_{1/2}\|J_{s,t}\psi\|\\
&\leq\frac{\nu}{8}\|J_{s,t}^{(2)}(\phi,\psi)\|_{1}^{2}
+C\|J_{s,t}\phi\|_{1/2}^{2}\|J_{s,t}\psi\|^{2}.
\end{aligned}
\end{equation*}
Similarly,
\begin{equation*}
\langle B(J_{s,t}\psi,J_{s,t}\phi),J_{s,t}^{(2)}(\phi,\psi)\rangle
\leq\frac{\nu}{8}\|J_{s,t}^{(2)}(\phi,\psi)\|_{1}^{2}
+C\|J_{s,t}\psi\|_{1/2}^{2}\|J_{s,t}\phi\|^{2},
\end{equation*}
 \begin{equation*}
\langle G(J_{s,t}^{(2)}(\phi,\psi)),J_{s,t}^{(2)}(\phi,\psi)\rangle
\leq \frac{\nu}{8}\|J_{s,t}^{(2)}(\phi,\psi)\|_{1}^{2}
+C\|J_{s,t}^{(2)}(\phi,\psi)\|^{2}.
\end{equation*}
Applying the chain rule to $\|J_{s,t}^{(2)}(\phi,\psi)\|^2$, with the help of \eqref{2.16} and \eqref{2.17}, one arrives at
\begin{equation*}
\begin{aligned}
\|J_{s,t}^{(2)}(\phi,\psi)\|^{2}&\leq Ce^{C\int_{s}^{t}(\|U_{r}\|_{1}^{4/3}+1)\mathrm{d}r}
\int_{s}^{t}\big[\|J_{s,r}\phi\|_{1/2}^{2}\|J_{s,r}\psi\|^{2}
+\|J_{s,r}\psi\|_{1/2}^{2}\|J_{s,r}\phi\|^{2}\big]\mathrm{d}r \\
&\leq Ce^{C\int_{s}^{t}(\|U_{r}\|_{1}^{4/3}+1)\mathrm{d}r}
\|\phi\|^{2}\|\psi\|^{2},
\end{aligned}
\end{equation*}
which shows \eqref{2.18}. Furthermore, by Young ’s inequality and \eqref{2.16} we can obtain
\eqref{2.19}. The proof is completed.
\end{proof}

Recall that $P_N$ is the orthogonal projection from $H$ into $H_N=\mathrm{span}\{\sigma_k^m,\psi_k^m:|k|\leq N,m\in\{0,1\}\}$ and $Q_N=I-P_N$. For any $N\in\mathbb{N}$, $t\geq 0$ and $\xi\in H$, denote $\xi_t^Q=Q_N J_{0,t}\xi$, $\xi_t^P=P_N J_{0,t}\xi$ and $\xi_t=J_{0,t}\xi$.

\begin{lemma}\label{lemma2.6}
  For any $t\geq 0$ and $\xi\in H$, one has
\begin{equation}
\|\xi_t^Q\|^2\leq \exp\{-\nu N^2 t\}\|\xi\|^2+\frac{C\|\xi\|^2}{\sqrt{N}}
\exp\left\{C\int_0^t(\|U_s\|_1^{4/3}+1)\mathrm{d}s\right\}\sup_{s\in[0,t]}\|U_s\|,
\end{equation}
where C is a constant depending on $\nu=\min\{\nu_1,\nu_2\}, g$.
\end{lemma}
\begin{proof}
Similar to \eqref{2.20} and \eqref{2.21}, there hold
\begin{equation*}
\langle B(
\xi_{t},U_{t}),\xi_{t}^{Q}\rangle+\langle B(U_{t},\xi_{t}),\xi_{t}^{Q}\rangle
\leq C\|\xi_{t}^{Q}\|_{1}\|U_{t}\|_{1/2}\|\xi_{t}\|\leq\frac{\nu}{4}\|\xi_t^Q\|_1^2+C\|U_t\|_{1/2}^2\|\xi_t\|^2,
\end{equation*}
and
 \begin{equation*}
\begin{aligned}
\langle G(\xi_t),\xi_t^Q\rangle
&\leq \frac{\nu}{4}\|\xi_t^Q\|_1^{2}+C\|\xi_t\|^{2}.
\end{aligned}
\end{equation*}
 Then
applying the chain rule to $\|\xi_t^Q\|^2$, note that $\|\xi_t^{Q}\|_1^2\geq N^2\|\xi_t^{Q}\|^2$, we find that
\begin{align*}
\|\xi_{t}^{Q}\|^{2}&\leq\exp\{-\nu N^{2}t\}\|\xi\|^{2}+C\int_{0}^{t}\exp\{-\nu N^{2}(t-s)\}(\|U_{s}\|_{1/2}^{2}\|\xi_{s}\|^{2}+\|\xi_{s}\|^{2})\mathrm{d}s \\
&\leq\exp\{-\nu N^{2}t\}\|\xi\|^{2} \\
&~~~+C\exp\{C\int_0^t(\|U_s\|_1^{4/3}+1)\mathrm{d}s\}\|\xi\|^2
\sup_{s\in[0,t]}\|U_s\|\int_0^t\exp\{-\nu N^2(t-s)\}\|U_s\|_1\mathrm{d}s \\
%&~~~+C\int_0^t\exp\{-\nu N^2(t-s)\}\|\xi_s\|_1^2 \mathrm{d}s\\
&\leq\exp\{-\nu N^2t\}\|\xi\|^2
+C\exp\{C\int_0^t(\|U_s\|_1^{4/3}+1)\mathrm{d}s\}\|\xi\|^2\sup_{s\in[0,t]}\|U_s\| \\
&~~~\times\big(\int_{0}^{t}\exp\{-4\nu N^{2}(t-s)\}\mathrm{d}s\big)^{1/4}\big(\int_{0}^{t}
\|U_{s}\|_{1}^{4/3}\mathrm{d}s\big)^{3/4},
\end{align*}
where we used \eqref{2.16} and \eqref{2.17} in the second inequality. The above inequality implies the desired result.
\end{proof}

\begin{lemma}\label{lemma 2.7}
Assume that $\xi_0^P=0$, then for any $t\geq 0$,
\begin{equation*}
\|\xi_t^P\|^2\leq C\|\xi\|^2\exp\left\{C\int_0^t\|U_s\|_1^{4/3}\mathrm{d}s\right\}
\sup\limits_{s\in[0,t]}(\|U_s\|^{5/2}+1)\frac{1+t}{N^{1/4}},
\end{equation*}
where $C>0$ is a constant depending on $\nu=\min\{\nu_1,\nu_2\}$ and $g$. Furthermore, combining the above inequality with Lemma \ref{lemma2.6}, for any $\xi\in H$ and $t\geq 0$, we have
\begin{equation*}
\begin{aligned}
&\|J_{0,t}Q_{N}\xi\|^{2} \\
&\leq C\big(e^{-\nu N^{2}t}+\frac{1+t}{N^{1/4}}\big)\exp
\left\{C\int_{0}^{t}\|U_{s}\|_{1}^{4/3}\mathrm{d}s\right\} \sup_{s\in[0,t]}(\|U_{s}\|^{5/2}+1)\|\xi\|^{2}.
\end{aligned}
\end{equation*}
\end{lemma}
\begin{proof}
 Similar to \eqref{2.20} and \eqref{2.21}, we have
\begin{equation*}
\langle B(\xi_t^P,U_t),\xi_t^P\rangle
\leq\frac{\nu}{6}\|\xi_t^P\|_1^2+C\|U_t\|_1^{4/3}\|\xi_t^P\|^2,
\end{equation*}
and
\begin{equation*}
\langle B(U_t,\xi_t^Q),\xi_t^P\rangle+\langle B(\xi_t^Q,U_t),\xi_t^P\rangle \leq C\|\xi_t^P\|_1\|U_t\|\|\xi_t^Q\|_{1/2}
\leq\frac{\nu}{6}\|\xi_t^P\|_1^2
+C\|U_t\|^2\|\xi_t^Q\|_{1/2}^2.
\end{equation*}
Moreover, one also has
 \begin{equation*}
\begin{aligned}
&\langle G(\xi_t),\xi_t^P \rangle
\leq \frac{\nu}{6}\|\xi_{t}^{P}\|_{1}^{2}+C\|\xi_{t}^P\|_1^{2}.
\end{aligned}
\end{equation*}
Applying the chain rule to $\|\xi_t^P\|^2$ and using Lemma \ref{lemma2.6} as well as \eqref{2.17} yields
\begin{align}
&\|\xi_{t}^{P}\|^{2}\leq C\exp\{C\int_{0}^{t}\|U_{s}\|_{1}^{4/3}\mathrm{d}s\}
\int_{0}^{t}\|U_{s}\|^{2}\|\xi_{s}^Q\|_{1/2}^{2}\mathrm{d}s \notag\\
&\leq C\exp\{C\int_{0}^{t}\|U_{s}\|_{1}^{4/3}\mathrm{d}s\}\sup_{s\in[0,t]}\|U_{s}\|^{2}
\int_{0}^{t}\|\xi_{s}^{Q}\|\|\xi_{s}^Q\|_{1}\mathrm{d}s \notag\\
&\leq
C\exp\{C\int_0^t\|U_s\|_1^{4/3}\mathrm{d}s\} \sup\limits_{s\in[0,t]}\|U_s\|^2\Big(\int_0^t\|\xi_s^Q\|^2\mathrm{d}s\Big)^{1/2}
\Big(\int_0^t\|\xi_s^Q\|_1^2\mathrm{d}s\Big)^{1/2} \notag\\
&\leq
C\|\xi\|\exp\{C\int_0^t\|U_s\|_1^{4/3}\mathrm{d}s\} \sup_{s\in[0,t]}\|U_s\|^2\Big(\int_0^t\|\xi_s^Q\|^2\mathrm{d}s\Big)^{1/2} \notag\\
&\leq
 C\|\xi\|\exp\{C\int_0^t\|U_s\|_1^{4/3}\mathrm{d}s\}\sup_{s\in[0,t]}\|U_s\|^2 \notag\\
&~~~\times\Big(\int_{0}^{t}\Big[\exp\{-\nu N^{2}s\}\|\xi\|^{2}
+\frac{C\|\xi\|^{2}}{\sqrt{N}}\exp\{C\int_{0}^{s}\|U_{r}\|_{1}^{4/3}\mathrm{d}r\}
\sup_{r\in[0,s]}\|U_{r}\|\Big]\mathrm{d}s\Big)^{1/2} \notag\\
&\leq
 C\|\xi\|^2\exp\{C\int_0^t\|U_s\|_1^{4/3}\mathrm{d}s\} \sup_{s\in[0,t]}\|U_s\|^2\Big(\frac{1}{N}
+\frac{\sup_{s\in[0,t]}\|U_r\|^{1/2}\sqrt{t}}{N^{1/4}}\Big) \notag\\
&\leq
 C\|\xi\|^{2}\exp\{C\int_{0}^{t}\|U_{s}\|_{1}^{4/3}\mathrm{d}s\} \sup_{s\in[0,t]}(\|U_{s}\|^{5/2}+1)\frac{1+t}{N^{1/4}}.    \notag
\end{align}
This completes the proof.
\end{proof}

Next, we present some estimates on the operators $\mathcal{A}_{s,t}$, $ J_{s,t}$, and the inverse of the regularized Malliavin matrix. Further details can be found in \cite{MH-2006}.
\begin{lemma}\label{lemma2.9}
For $0<s<t$ and $\beta>0$, there exists a constant $C=C(\{\alpha_k^m\}_{k\in\mathcal{Z},m\in\{0,1\}}, d)>0$ such that
\begin{equation*}
\|\mathcal{A}_{s,t}\|_{\mathcal{L}(L^2([\ell_s,\ell_t];\mathbb{R}^d),H)}\leq C\left(\int_{\ell_s}^{\ell_t}\| J_{\gamma_r,t}\|_{\mathcal{L}(H,H)}^2\mathrm{d}r\right)^{1/2}.
\end{equation*}
Moreover,
\begin{equation*}
\begin{aligned}
&\|\mathcal{A}_{s,t}^{*}(\mathcal{M}_{s,t}+\beta I)^{-1/2}\|
_{\mathcal{L}(H,L^{2}([\ell_s,\ell_t];\mathbb{R}^{d}))}\leq1,
\\
&\|(\mathcal{M}_{s,t}+\beta I)^{-1/2}\mathcal{A}_{s,t}\|
_{\mathcal{L}(L^{2}([\ell_s,\ell_t];\mathbb{R}^{d}),H)}\leq1,
\\
&\|(\mathcal{M}_{s,t}+\beta I)^{-1/2}\|_{\mathcal{L}(H,H)}\leq\beta^{-1/2}.
\end{aligned}
\end{equation*}
Here $\mathcal{L}(X,Y)$ denotes the operator norm of the linear map between the given Hilbert spaces X and Y.
\end{lemma}

By the Riesz representation theorem, and \eqref{P.4}, one has
\begin{equation}\label{riesz}
\mathcal{D}_{u}^{j}U_t
=J_{\gamma_u,t}\sigma_{\theta}e_j,~~\text{for any}~ j=1,\ldots,d,~u\in[0,\ell_t],
\end{equation}
where $\mathcal{D}_u^jF:=(\mathcal{D}F)^j(u)$ is the $j$th component of $\mathcal{D}F$ evaluated at time $u$, and $\{e_j\}_{j=1,\ldots,d}$ is the standard basis of $\mathbb{R}^d$.

Recall that $\mathcal{D}$ is the Malliavin derivative. As in \cite{MH-2006}, if $A:\mathcal{H}_1\rightarrow\mathcal{H}_2$ is a random linear map between two Hilbert spaces, we denote by $\mathcal{D}_{s}^j A:\mathcal{H}_1\rightarrow\mathcal{H}_2$ the random linear map defined by
\begin{equation*}
(\mathcal{D}_{s}^j A)h:=\langle\mathcal{D}_s(Ah),e_j\rangle.
\end{equation*}
Then observe that for any $0\leq s\leq t$, $j=\{1,\cdots,d\}$ and $u\in [0,\ell_t]$, we have
\begin{equation*}
\mathcal{D}_{u}^{j}J_{s,t}\xi
=\begin{cases}
J_{\gamma_u,t}^{(2)}(\sigma_{\theta}e_{j},J_{s,\gamma_u}\xi),~u\in [\ell_s,\ell_t],
\\[2ex]
J_{s,t}^{(2)}(J_{\gamma_u,s}\sigma_{\theta}e_{j},\xi),~\mathrm{if}~u<\ell_s.
\end{cases}
\end{equation*}

\begin{lemma}\label{lemma2.10}
For any $0\leq s<t$, the random operators $ J_{s,t}, \mathcal{A}_{s,t}, \mathcal{A}_{s,t}^{*}$ are differentiable in the Malliavin sense. Moreover, for any $\hbar>0$ and $p\geq1$, there is a constant $C=C(\hbar,p)$, we have the bounds
\begin{equation*}
\begin{aligned}
&\mathbb{E}\|\mathcal{D}_{u}^{j} J_{s,t}\|^{p}
\leq C\exp(\hbar\|U_{0}\|^{2}),
\\
&\mathbb{E}\|\mathcal{D}_{u}^{j}\mathcal{A}_{s,t}\|
_{\mathcal{L}(L^{2}([\ell_s,\ell_t],\mathbb{R}^{d}),H)}^{p}
\leq C\exp(\hbar\|U_{0}\|^{2}),
\\
&\mathbb{E}\|\mathcal{D}_{u}^{j}\mathcal{A}_{s,t}^{*}\|
_{\mathcal{L}(H,L^{2}([\ell_s,\ell_t],\mathbb{R}^{d}))}^{p}
\leq C\exp(\hbar\|U_{0}\|^{2}).
\end{aligned}
\end{equation*}
\end{lemma}

In what follows, we provide details for the moment bounds throughout the manuscript. Denote
$$
\zeta^*:=\frac{\nu_1\nu_2}{g^2},
$$
then, from \eqref{norm}, we find that for $n\geq1$,
$$
\|U\|^n=\left(\zeta^*\|w\|^{2n}+\|\theta\|^{2n}\right)^{\frac{1}{2}},
~~~\|U\|_1^{n}=\left(\zeta^*\|w\|_1^{2n}+\|\theta\|_1^{2n}\right)^{\frac{1}{2}}.
$$

\begin{lemma}\label{lemma2.1}
Let $U_t$ be the solution to equation \eqref{2.1} with initial value $U_0$. Then there exist positive constants $C_1, C_2$, depending on the parameters $\{\alpha_k^m\}_{k\in \mathcal{Z}, m\in\{0,1\}}$, $\nu_S$, $d$, $g$, such that for all $\nu=\min\{\nu_1,\nu_2\}>0$, there hold
\begin{equation}\label{2.3}
\mathbb{E}\left[\|U_{t}\|^{2}\right]\leq e^{-\nu t}\|U_{0}\|^{2}+C_{1},\quad \forall t\geq0,
\end{equation}
\begin{equation}\label{2.4}
\nu \mathbb{E}\left[\int_{0}^{t}\|U_{s}\|_{1}^{2}\mathrm{d}s\right]
\leq\|U_{0}\|^{2}+C_{2}t,\quad\forall t\geq0.
\end{equation}
\end{lemma}
\begin{proof}
Let $\nu_L$ be the intensity measure defined in \eqref{1.8}. By \cite[Lemma 2.1]{Peng-2024}, we know that
\begin{equation}\label{2.5}
\int_{|z|\leq1}|z|^2\nu_L(\mathrm{d}z)
+\int_{|z|\geq1}|z|^n\nu_L(\mathrm{d}z)
<\infty,\quad\forall n\geq 2.
\end{equation}
By applying It\^{o}'s formula to $\zeta^*\|w_t\|^2$ and $\|\theta_t\|^2$, we derive
\begin{equation}\label{2.6}
\begin{aligned}
  &\mathrm{d}\zeta^*\|w\|^2+2\zeta^*\nu_1\|\nabla w\|^2\mathrm{d}t =2\zeta^*g\langle\partial_x\theta,w\rangle \mathrm{d}t,\\
     &\mathrm{d}\|\theta\|^2+2\nu_2\|\nabla \theta\|^2\mathrm{d}t=2\int_{z\in \mathbb{R}^{d}\setminus \{0\}}\langle \theta,\sigma_\theta z\rangle \tilde{N}_L(\mathrm{d}t,\mathrm{d}z)
     +\int_{z\in\mathbb{R}^d\setminus\{0\}}
     \|\sigma_\theta z\|^2N_L(\mathrm{d}t,\mathrm{d}z).
  \end{aligned}
\end{equation}
By Cauchy-Schwarz inequality and Poincar\'{e}'s inequality, we obtain
\begin{equation}
2\zeta^*g|\langle \partial_x\theta,w\rangle|\leq \nu_{1}\zeta^*\| w\|^2+\nu_2\|\nabla\theta\|^2\leq \nu_1\zeta^*\|\nabla w\|^2+\nu_2\|\nabla\theta\|^2.
\end{equation}
Set $C=\int_{z\in\mathbb{R}^d\setminus\{0\}}\|\sigma_\theta z\|^2\nu_L(\mathrm{d}z)$, which is a constant depending on $\{\alpha_k^m\}_{k\in \mathcal{Z},m\in\{0,1\}},\nu_S,d$. Then we can get
\begin{equation*}
\mathbb{E}\left[\|U_t\|^2\right]+\nu\mathbb{E}\left[\int_0^t\|U_s\|_{1}^2
\mathrm{d}s\right]
\leq\|U_0\|^2+Ct,
\end{equation*}
the above inequality implies the derivation of \eqref{2.3} and \eqref{2.4}.
\end{proof}

Next, we introduce some stopping times to build new preliminary estimates.

Let $\eta_0=0$ and $\mathfrak{B}_0=\sum\limits_{k\in\mathcal{Z},m\in\{0,1\}}(\alpha_k^m)^2$. For any $\kappa>0$, $n\in \mathbb{N}$, $\ell\in \mathbb{S}$ and $\nu=\min\{\nu_1,\nu_2\}$, we define
\begin{equation}\label{2.7}
\eta=\eta(\ell)=\eta_1(\ell)=\inf\{t\ge0:\nu t-8\mathfrak{B}_0\kappa\ell_t>1
\},
\end{equation}
and
\begin{equation}\label{2.8}
\eta_n=\eta_n(\ell)=\inf\{t\geq\eta_{n-1}, \nu(t-\eta_{n-1})-8\mathfrak{B}_0\kappa(\ell_t-\ell_{\eta_{n-1}})>1\}.
\end{equation}

For the solution of equation \eqref{2.1} and the stopping times $\eta_n $(with respect to $\mathcal{F}_t$), we obtain the following moment estimates. Given the numerous constants in the remaining part of this section, we will use the following convention: the letter $C$ is a positive constant depending on $g$, $\nu$, $\{\alpha_k^m\}_{k\in\mathcal{Z},m\in\{0,1\}},\nu_S$ and $d=2\cdot|\mathcal{Z}|$, the letter $C_{\kappa}$ is a positive constant depending on $\kappa$ and $g$, $\nu$, $\{\alpha_k^m\}_{k\in\mathcal{Z},m\in\{0,1\}},\nu_S$, $d$, and $C_{n,\kappa}$ is a positive constant depending on $n,\kappa$ and $g$, $\nu$, $\{\alpha_k^m\}_{k\in\mathcal{Z},m\in\{0,1\}},\nu_S$, $d$.

\begin{lemma}\label{moment}
There exists a constant $\kappa_0\in(0,\nu]$ only depending on $\nu$, $\{\alpha_k^m\}_{k\in\mathcal{Z},m\in\{0,1\}},\nu_S$ and $d=2\cdot|\mathcal{Z}|$ such that the following statements hold:

$(1)$ For any $\kappa\in(0,\kappa_0]$ and the stopping time $\eta$ defined in \eqref{2.7}, we have
\begin{equation}\label{2.9}
\mathbb{E}^{\mu_{\mathbb{S}}}[\exp\{10\nu\eta\}]\leq C_\kappa,
\end{equation}
%where $C_{\kappa}$ is a constant depending on $\kappa$ and $g$, $\nu$, $\{\alpha_k^m\}_{k\in\mathcal{Z},m\in\{0,1\}},\nu_S$, $d$.
hence,
\begin{equation}\label{2.10}
\mathbb{E}[\exp\{10\nu\eta\}]\leq C_\kappa.
\end{equation}

$(2)$ For any $\kappa \in (0,\kappa_{0}]$, almost all $\ell\in \mathbb{S}$(under the measure $\mathbb{P}^{\mu_{\mathbb{S}}}$) and the stopping times $\eta_{k}$ defined in \eqref{2.7} and \eqref{2.8}, we have
\begin{equation}\label{3.11}
\begin{aligned}
\mathbb{E}^{\mu_{\mathbb{W}}}\Big[\exp&\Big\{
\kappa\|U_{\eta_k}\|^2-\kappa\|U_{\eta_{k-1}}\|^2e^{-1} + \nu\kappa\int_{\eta_{k-1}}^{\eta_{k}}
e^{-\nu(\eta_{k}-s)+8\mathfrak{B}_{0}\kappa(\ell_{\eta_{k}}-\ell_{s})}
\|U_{s}\|_{1}^{2}\mathrm{d}s \\
&-\kappa\mathfrak{B}_{0}(\ell_{\eta_{k}}-\ell_{\eta_{k-1}})\Big\}
\Big|\mathcal{F}_{\ell_{\eta_{k-1}}}^{W}\Big]\leq C.
\end{aligned}
\end{equation}
%where $C$ is a constant depending on $g$, $\nu$, $\{\alpha_k^m\}_{k\in\mathcal{Z},m\in\{0,1\}},\nu_S$ and $d$.
Moreover, the following estimate hold:
\begin{equation}
\begin{aligned}
\mathbb{E}\bigg[\exp&\bigg\{\kappa\|U_{\eta_{k}}\|^{2}-\kappa\|U_{\eta_{k-1}}\|^{2}
e^{-1}+\nu\kappa\int_{\eta_{k-1}}^{\eta_{k}}
e^{-\nu(\eta_{k}-s)+8\mathcal{B}_{0}\kappa(\ell_{\eta_{k}}-\ell_{s})}
\|U_{s}\|_{1}^{2}\mathrm{d}s\\
&-\kappa\mathfrak{B}_{0}(\ell_{\eta_{k}}-\ell_{\eta_{k-1}})\Big\}
\Big|\mathcal{F}_{\eta_{k-1}}\Big]\leq C,
\end{aligned}
\end{equation}
where $C>0$ is the constant appearing in \eqref{3.11}.

$(3)$ For any $\kappa\in (0,\kappa_0]$ and $k\in \mathbb{N}$, one has
\begin{equation}
\mathbb{E}\Big[\exp\big\{\kappa\|U_{\eta_{k+1}}\|^{2}\big\}
\Big|\mathcal{F}_{\eta_{k}}\Big]\leq C_{\kappa}\exp\big\{\kappa e^{-1}\|U_{\eta_{k}}\|^{2}\big\}.
\end{equation}
%where $C_{\kappa}$ is a constant depending on $\kappa$ and $g$, $\nu$, $\{\alpha_k^m\}_{k\in\mathcal{Z},m\in\{0,1\}},\nu_S$, $d$.

$(4)$ For any $\kappa\in(0,\kappa_{0}]$, there exists a constant $C_{\kappa}>0$ %depending on $\kappa$ and $g,\nu,\{\alpha_k^m\}_{k\in\mathcal{Z},m\in\{0,1\}},\nu_{S},d$
such that for any $n\in\mathbb{N}$ and $U_{0}\in H$, one has
\begin{equation}\label{moment.4}
\mathbb{E}_{U_0}\left[\exp\{\kappa\sum_{i=1}^n\|U_{\eta_i}\|^2-C_\kappa n\}\right]
\leq e^{a\kappa\|U_0\|^2},
\end{equation}
where $a=\frac{1}{1-e^{-1}}$. In this paper, we use the notation $\mathbb{E}_{U_0}$ for expectations under the measure $\mathbb{P}$ of the solutions to equation \eqref{2.1} with initial condition $U_0$.

$(5)$ For any $\kappa\in(0,\kappa_0]$, $U_0\in H$ and $n\in\mathbb{N}$, one has
\begin{equation}\label{moment.5}
\mathbb{E}_{U_0}\left[\sup_{s\in[0,\eta]}\|U_s\|^{2n}\right]\leq C_{n,\kappa}(1+\|U_0\|^{2n}).
\end{equation}
%where $C_{n,\kappa}$ is a constant depending on $n,\kappa$ and $g$, $\nu$, $\{\alpha_k^m\}_{k\in\mathcal{Z},m\in\{0,1\}},\nu_S$, $d$.
\end{lemma}

The remaining part of this section is entirely dedicated to proving the above lemma. We first provide some preparatory work and knowledge.

For $\kappa>0$, $\varepsilon\in(0,1]$ and $\ell\in \mathbb{S}$, we set
\begin{equation*}
\ell_t^\varepsilon=\frac1\varepsilon\int_t^{t+\varepsilon}\ell_s\mathrm{d}s
+\varepsilon t,
\end{equation*}
and
\begin{equation*}
\eta^{\varepsilon}=\eta^{\varepsilon}(\ell):=\inf \{t\geq0:\nu t-8\mathfrak{B}_{0}\kappa\ell_{t}^{\varepsilon}>1\}.
\end{equation*}
Recall that $\ell$ is a c\`{a}dl\`{a}g increasing function from $\mathbb{R}^+$ to $\mathbb{R}^+$ with $\ell_0=0$, it is easy to see that the following lemma is valid.

\begin{lemma}\label{lemma A0}
For $\ell \in \mathbb{S}$,

$(i)$ $\ell_{\cdot }^{\varepsilon }: [ 0, \infty ) \to [ 0, \infty )$ is continuous and strictly increasing;

$(ii)$ for any $t\geq0$,  $\ell_{t}^{\varepsilon}$ strictly decreases to $\ell_t$ as $\varepsilon$ decreases to $0$.
\end{lemma}

Referring to \cite[Lemma A.2]{Peng-2024}, the following moment estimates hold for the stopping times $\eta^{\varepsilon}$ and $\eta$.

\begin{lemma}\label{lemma A1}
  There exists a constant $\tilde{\kappa}_0>0$ such that for any $\kappa\in(0,\tilde{\kappa}_0]$,
\begin{equation}
\sup_{\varepsilon\in(0,1]}\mathbb{E}^{\mu_{\mathbb{S}}}
\left[\exp\{10\nu\eta^{\varepsilon}\}\right]\leq C_\kappa,
\end{equation}
and
\begin{equation}
\mathbb{E}^{\mu_{\mathbb{S}}}\left[\exp\{10\nu\eta\}\right]\leq C_\kappa.
\end{equation}
%where $C_{\kappa}$ is a constant depending on $\kappa$, $\nu$, $\{\alpha_k^m\}_{k\in\mathcal{Z},m\in\{0,1\}},\nu_S,d$.
\end{lemma}

For any $\kappa\in(0,\tilde{\kappa}_0]$, let
\begin{equation}\label{S1}
  \mathbb{S}_1=\{\ell\in\mathbb{S}:\eta^1(\ell)<\infty\}.
\end{equation}
We have the following lemma.

\begin{lemma}\label{lemma A2}
We claim that ${\mathbb{P}^{\mu_{\mathbb{S}}}(\mathbb{S}_1)=1}$. And for any $\ell\in {\color{blue}{\mathbb{S}_1}} $, the following statements hold.

$(1)$ For any $\varepsilon\in(0,1)$, $\eta^{\varepsilon}<\eta^1<\infty$ and $\nu\eta^\varepsilon-8\mathfrak{B}_0\kappa
\ell_{\eta^\varepsilon}^\varepsilon=1$;

$(2)$ $\eta^{\varepsilon}$ strictly decreases to $\eta$ as $\varepsilon$ decreases to $0$;

$(3)$ $\ell_{\eta^\varepsilon}^\varepsilon$ strictly decreases to $\ell_{\eta}$ as $\varepsilon$ decreases to $0$;

$(4)$ $\nu\eta-8\mathfrak{B}_0\kappa\ell_{\eta}=1$;

$(5)$ $
\lim\sup\limits_{\varepsilon\to0}\int_{0}^{\eta^{\varepsilon}}
\exp\{-\nu(\eta^{\varepsilon}-s)
+8\mathfrak{B}_{0}\kappa(\ell_{\eta^{\varepsilon}}^{\varepsilon}
-\ell_{s}^{\varepsilon})\}\mathrm{d}\ell_{s}^{\varepsilon}\leq\ell_{\eta}.
$
\end{lemma}
\begin{proof}
  The proof follows similarly to \cite{Peng-2024}.
\end{proof}

Let $\mathcal{H}_0=\text{span}\{\sigma_k^m:k\in\mathcal{Z},m\in\{0,1\}\}$ and $D([0,\infty);\mathcal{H}_0)$ be the space of all c\`{a}dl\`{a}g functions taking values in $\mathcal{H}_0$. Keeping in mind that $d=2\cdot|\mathcal{Z}|<\infty$, it is well-known that, for any $U_0\in H$ and $f\in D([0,\infty);\mathcal{H}_0)$, there exists a unique solution $\Psi(U_{0},f)\in C([0,\infty);H)\cap L_{loc}^{2}([0,\infty);V)$ to the following PDE:
\begin{equation*}
\begin{aligned}
\Psi(U_{0},f)(t)
=&U_{0}-\int_{0}^{t}A(\Psi(U_{0},f)(s)+f_{s})\mathrm{d}s\\
&-\int_{0}^{t}B(\Psi(U_{0},f)(s)
+f_{s})\mathrm{d}s+\int_{0}^{t}G(\Psi(U_{0},f)(s)+f_{s})\mathrm{d}s.
\end{aligned}
\end{equation*}
Here $V=\{h\in H:\|h\|_1<\infty\}$.

Define $\Psi(U_0,f)
:=(\tilde{w},\tilde{\theta})^T
=(w,\theta)^T-(0,f)^T$, where $\tilde{w}$ and $\tilde{\theta}$ satisfy
\begin{equation}\label{A.2}
\begin{aligned}
&\mathrm{d}\tilde{w}+(B_1(\tilde{w},\tilde{w})-\nu_{1}\Delta \tilde{w})\mathrm{d}t
=g\partial_{x}(\tilde{\theta}+f) \mathrm{d}t,\\
&\mathrm{d}\tilde{\theta}+(B_2(\tilde{w},\tilde{\theta}+f)
-\nu_{2}\Delta\tilde{\theta}
-\nu_{2}\Delta f)\mathrm{d}t=0.
\end{aligned}
\end{equation}

We denote $b_t^{\varepsilon}=\sigma_{\theta}(W_{\ell_{t}^{\varepsilon}} - W_{\ell_{0}^{\varepsilon}}),b_{t}= \sigma_{\theta}W_{\ell_{t}},V_{t}^{\varepsilon} = \Psi(U_{0},b^{\varepsilon})(t)$ and $V_{t} = \Psi(U_{0},b)(t)$. It is easy to see that $V_t+b_t$ is the unique solution $U_t$ to \eqref{2.1}, i.e., $U_t=V_t+b_t$, and for any $\ell\in \mathbb{S}$ and $\varepsilon\in(0,1]$,  $U_t^{\varepsilon}=V_t^{\varepsilon}+b_t^{\varepsilon}$ is the solution of the following PDE:
\begin{equation*}
 U_t^{\varepsilon}=U_0-\int_0^t[A(U_s^{\varepsilon})
+B(U_s^{\varepsilon})-G(U_s^{\varepsilon})]\mathrm{d}s
+\sigma_{\theta}(W_{\ell_t^{\varepsilon}}-W_{\ell_0^{\varepsilon}})
\end{equation*}

Recall $\mathbb{S}_1$ introduced in \eqref{S1}. We have

\begin{lemma}
  For any $\ell\in \mathbb{S}_1$, the following statements hold:
\begin{equation}\label{A.3}
\begin{gathered}
\operatorname*{lim}_{\varepsilon\to0}
\int_{0}^{\eta^{\varepsilon}}e^{-\nu(\eta^{\varepsilon}-s)
+8\mathfrak{B}_{0}\kappa(\ell_{\eta^{\varepsilon}}^{\varepsilon}
-\ell_{s}^{\varepsilon})}\|U_{s}^{\varepsilon}\|_{1}^{2}\mathrm{d}s =\int_{0}^{\eta}e^{-\nu(\eta-s)
+8\mathfrak{B}_{0}\kappa(\ell_{\eta}-\ell_{s})}
\|U_{s}\|_{1}^{2}\mathrm{d}s,
\end{gathered}
\end{equation}
and
\begin{equation}\label{A.4}
\begin{aligned}
\lim_{\varepsilon\to0}\|U_{\eta^\varepsilon}^\varepsilon
-U_\eta\|^2=0.
\end{aligned}
\end{equation}
\end{lemma}
\begin{proof}
  To prove this lemma, we first need some a priori estimates for $\Psi$.

By the chain rule, there exists a constant $C=C(\nu_1,\nu_2,g)>0$ such that, for any $U_0\in H, f\in D([0,\infty);\mathcal{H}_0)$ and $t\geq 0$,
\begin{equation}
\begin{aligned}
\frac{\zeta^*}{2}\frac{\mathrm{d}}{\mathrm{d}t}\|\tilde{w}\|^2+\zeta^*\nu_{1}\|\tilde{w}\|_1^2
&\leq \zeta^*g(\|\tilde{\theta}\|_1+\|f\|_1)\|\tilde{w}\|\\
&\leq \frac{\nu_2}{4}\|\tilde{\theta}\|_1^2+C\|\tilde{w}\|^2
+C\|f\|_1^2+\frac{\zeta^*\nu_1}{2}\|\tilde{w}\|_1^2  ,   \\
\frac{1}{2}\frac{\mathrm{d}}{\mathrm{d}t}\|\tilde{\theta}\|^2+\nu_{2}\|\tilde{\theta}\|_1^2
&\leq\|\tilde{w}\|\|f\|_1\|\tilde{\theta}\|_1
-\nu_{2}(\Delta f,\tilde{\theta})\\
&\leq \frac{\nu_2}{4}\|\tilde{\theta}\|_1^2+C\|f\|_1^2\|\tilde{w}\|^2
+C\|f\|_2^2+C\|\tilde{\theta}\|^2.
\end{aligned}
\end{equation}
Note that for any $\alpha>0$, there exists a constant $C_{\alpha}$ such that $\|h\|_{\alpha}\leq C_{\alpha}\|h\|,\forall h\in\mathcal{H}_0$.
Then applying the Gronwall lemma, for any $T > 0$, there holds
\begin{equation}\label{A.10}
\begin{aligned}
&\sup_{t\in[0,T]}\|\Psi(U_{0},f)(t)\|^{2}
+\nu\int_{0}^{T}\|\Psi(U_{0},f)(s)\|_{1}^{2}\mathrm{d}s\\
\leq& C\Big(\|U_{0}\|^{2}+\int_{0}^{T}(1+\|f_{s}\|^{2})\mathrm{d}s\Big)
e^{C\int_{0}^{T}(1+\|f_{s}\|^{2})\mathrm{d}s}.
\end{aligned}
\end{equation}

For any $f^1,f^2 \in D([0,\infty);\mathcal{H}_0)$, put $\Psi^1(t)=\Psi(U_0,f^1)(t)$ and $\Psi^2(t)=\Psi(U_0,f^2)(t)$, simplifying the notation. For the convenience of reading, in the following proof of this lemma, we use $(w,\theta)$ instead of $(\tilde{w},\tilde{\theta})$, then applying the chain rule for $\|w^1-w^2\|^2$ and $\|\theta^1-\theta^2\|^2$ yields
\begin{equation}
\begin{aligned}
\frac{\zeta^*}{2}\frac{\mathrm{d}}{\mathrm{d}t} \|w^1-w^2\|^2&=-\zeta^*\nu_1 \|\nabla(w^1-w^2)\|^2-\zeta^*\langle B_1(w^1,w^1)-B_1(w^2,w^2),w^1-w^2\rangle\\
&~~~+\zeta^*g\langle\partial_{x}(\theta^1-\theta^2+f^1-f^2),w^1-w^2\rangle ,\\
\frac{1}{2}\frac{\mathrm{d}}{\mathrm{d}t}\|\theta^1-\theta^2\|^2
&=-\nu_2\|\nabla(\theta^1-\theta^2)\|^2-\langle B_2(w^1,\theta^1+f^1)-B_2(w^2,\theta^2+f^2),\theta^1-\theta^2\rangle\\
&~~~+\nu_{2}(\Delta f^1-\Delta f^2,\theta^1-\theta^2),
\end{aligned}
\end{equation}
where $(w^i,\theta^i), (i=1,2)$ are the solutions of \eqref{A.2}.
By \eqref{2.2} and Young's inequality, we have
\begin{align*}
&\zeta^*\langle B_1(w^1,w^1)-B_1(w^2,w^2),w^1-w^2\rangle
=\zeta^*\langle B_1(w^1-w^2,w^1),w^1-w^2\rangle\\
&\leq C\zeta^*\| w^1-w^2\|_1\|w^1\|_1\|w^1-w^2\|
\leq\frac{\zeta^*\nu_1}{4}\| w^1-w^2\|_1^2+C\|w^1\|_1^2\|w^1-w^2\|^2,
\end{align*}
and
\begin{equation*}\begin{aligned}
&\langle B_2(w^1,\theta^1+f^1)-B_2(w^2,\theta^2+f^2),\theta^1-\theta^2\rangle\\
=&\langle B_2(w^1-w^2,\theta^1+f^1)+B_2(w^2,f^1-f^2),\theta^1-\theta^2\rangle
\\
\leq&C\|w^1-w^2\|_1 \|\theta^1+f^1\|_1 \|\theta^1-\theta^2\|
+C\|w^2\|_1\|f^1-f^2\|\|\theta^1-\theta^2\|_1
\\
\leq&\frac{\zeta^*\nu_1}{4}\| w^1-w^2\|_1^2+C\|\theta^1+f^1\|_1^2 \|\theta^1-\theta^2\|^2+\frac{\nu_2}{6}\| \theta^1-\theta^2\|_1^2+C\|w^2\|_1^2\|f^1-f^2\|^2.
\end{aligned}
\end{equation*}
Using Young's inequality again, we get
\begin{align*}
\zeta^*g\langle\partial_{x}(\theta^1-\theta^2+f^1-f^2),
w^1-w^2\rangle
&\leq \frac{\nu_2}{6}\|\theta^1-\theta^2\|_1^2+C\|f^1-f^2\|_1^2+C\|w^1-w^2\|^2,
\\
\nu_{2}(\Delta f^1-\Delta f^2,\theta^1-\theta^2)
&\leq \frac{\nu_2}{6}\|\theta^1-\theta^2\|_1^2+C\|f^1-f^2\|_1^2.
\end{align*}
Combining the above estimates, we have
\begin{align}
&\zeta^*\|w^1_t-w^2_t\|^2+\|\theta_t^1-\theta_t^2\|^2
+\nu\int_0^t (\zeta^*\|w^1_s-w^2_s\|_1^2+\|\theta_s^1-\theta_s^2\|_1^2)\mathrm{d}s\\ \notag
&\leq \int_0^t \left[(C\|w^1_s\|_1^2+C)\|w^1_s-w^2_s\|^2
+C\|\theta_s^1+f^1_s\|_1^2\|\theta^1_s-\theta^2_s\|^2\right]\mathrm{d}s\\\notag
&+\int_0^t( C\|w^2_s\|_1^2\|f^1_s-f^2_s\|^2+C\|f^1_s-f^2_s\|_1^2)\mathrm{d}s \\\notag
&\leq C(1+\sup_{s\in[0,t]}\|w^1_s\|_1^2+\|\theta_s^1+f^1_s\|_1^2)\int_0^t [\|w^1_s-w^2_s\|^2
+\|\theta^1_s-\theta^2_s\|^2]\mathrm{d}s\\\notag
&+C(1+\sup_{s\in[0,t]}\|w^2_s\|_1^2)\int_0^t( \|f^1_s-f^2_s\|^2+\|f^1_s-f^2_s\|_1^2)\mathrm{d}s.
\end{align}
Using Gronwall's inequality, we obtain that for any $T\geq0$
\begin{align}\label{A.11}
&\sup_{t\in[0,T]}\|\Psi^{1}(t)-\Psi^{2}(t)\|^{2}
+\nu\int_{0}^{T}\|\Psi^{1}(s)-\Psi^{2}(s)\|_{1}^{2}\mathrm{d}s \notag\\
&\leq C(1+\sup_{s\in[0,T]}\|w^2_s\|_1^2)\int_0^T( \|f^1_s-f^2_s\|^2+\|f^1_s-f^2_s\|_1^2)\mathrm{d}s \\
&~~~\times\exp\Big\{C\int_0^T\big(1+\sup_{s\in[0,t]}\|w^1_s\|_1^2+\|\theta_s^1+f^1_s\|_1^2\big)\mathrm{d}s\Big\}.\notag
\end{align}

For any $(\mathrm{w},\ell)\in \mathbb{W}\times\mathbb{S}$, from the definitions of $b_t^{\varepsilon}$ and $b_t$, we derive that for any $T\geq0$,
\begin{equation}\label{A.15}
\sup\limits_{\varepsilon\in(0,1]}
\sup\limits_{t\in[0,T]}
\left(\|b_t^\varepsilon(\text{w},\ell)\|
+\|b_t(\text{w},\ell)\|\right)
\leq C\sup\limits_{t\in[0,\ell_{T+1}+T]}\|\text{w}_t\|
<\infty,
\end{equation}
and
\begin{equation}\label{A.16}
\lim_{\varepsilon\to0}
\int_0^T\|b_t^\varepsilon(\mathrm{w},\ell)
-b_t(\mathrm{w},\ell)\|^2\mathrm{d}t=0.
\end{equation}
Combining \eqref{A.15}-\eqref{A.16}, \eqref{A.10} and \eqref{A.11}, there exists a constant $C>0$ depending on $U_0$, $T$, and $\sup\limits_{t\in[0,\ell_{T+1}+T]}\|w_t\|$, such that the following estimates hold:
\begin{equation}\label{A.12}
\begin{aligned}
&\sup_{\varepsilon\in(0,1]}\Big(\sup_{t\in[0,T]}
\|U_{t}^{\varepsilon}\|^{2}
+\int_{0}^{T}\|U_{t}^{\varepsilon}\|_{1}^{2}\mathrm{d}t\Big)(\mathrm{w},\ell)\\
&~~~~+\Big(\sup_{t\in[0,T]}\|U_{t}\|^{2}
+\int_{0}^{T}\|U_{t}\|_{1}^{2}\mathrm{d}t\Big)(\mathrm{w},\ell)
\leq  C,
\end{aligned}
\end{equation}
and
\begin{equation}\label{A.13}
\lim\limits_{\varepsilon\to0}\Big(\sup\limits_{t\in[0,T]}
\|V_t^\varepsilon-V_t\|^2
+\int_0^T\|U_t^\varepsilon-U_t\|_1^2\mathrm{d}t\Big)(\mathrm{w},\ell)
=0.
\end{equation}
Notice that for any $\ell\in \mathbb{S}_1$ and $\text{w}\in \mathbb{W}$,
\begin{equation}\label{A.14}
\begin{aligned}
\|U_{\eta^{\varepsilon}}^{\varepsilon}-U_{\eta}\|
&\leq\|V_{\eta^{\varepsilon}}^{\varepsilon}-V_{\eta}\|
+\|b_{\eta^{\varepsilon}}^{\varepsilon}-b_{\eta}\| \\
&\leq\|V_{\eta^{\varepsilon}}^{\varepsilon}-V_{\eta^{\varepsilon}}\|
+\|V_{\eta^{\varepsilon}}-V_{\eta}\|
+\|b_{\eta^{\varepsilon}}^{\varepsilon}-b_{\eta}\| \\
&\leq\sup_{t\in[0,\eta^{1}]}\|V_{t}^{\varepsilon}-V_{t}\|
+\|V_{\eta^{\varepsilon}}-V_{\eta}\|
+\|\sigma_{\theta}(W_{\ell_{\eta^{\varepsilon}}^{\varepsilon}}
-W_{\ell_{0}^{\varepsilon}})-\sigma_{\theta}W_{\ell_{\eta}}\|.
\end{aligned}
\end{equation}
Applying Lemma \ref{lemma A0}, Lemma \ref{lemma A1}, and \eqref{A.12}-\eqref{A.14}, together with the fact that $V_t$ is continuous in $H$, we can deduce \eqref{A.3} and \eqref{A.4}. This completes the proof.
\end{proof}

We also have the following estimate on $U_t^{\varepsilon}$.

\begin{lemma}\label{lemma A4}
There exists a positive constant $C$,
%, which depends on $\nu,g,\nu_S$, $\{\alpha_k^m\}_{k\in\mathcal{Z},m\in\{0,1\}}$ and $d=2\cdot|\mathcal{Z}|$
such that, for any $\kappa\in(0,\tilde{\kappa}_0],\varepsilon\in(0,1]$, $\ell\in\mathbb{S}_1$, there holds
\begin{equation}
\begin{aligned}
&\mathbb{E}^{\mu_{\mathbb{W}}}\left[\exp\left\{
\kappa\|U_{\eta^{\varepsilon}}^{\varepsilon}\|^2
-\kappa\|U_0\|^2 e^{-\nu\eta^{\varepsilon}
+8\mathfrak{B}_{0}\kappa \ell_{\eta^{\varepsilon}}^{\varepsilon}}\right.\right.\\
&+\kappa^*\int_{0}^{\eta^{\varepsilon}}
e^{-\nu(\eta^{\varepsilon}-s)
+8\mathfrak{B}_{0}\kappa(\ell_{\eta^{\varepsilon}}^{\varepsilon}
-\ell_{s}^{\varepsilon})}\|U_{s}^{\varepsilon}\|_{1}^{2}{\mathrm{d}}s\\
&\left.\left.-\kappa\mathfrak{B}_{0}\int_{0}^{\eta^{\varepsilon}}
e^{-\nu(\eta^{\varepsilon}-s)
+8\mathfrak{B}_{0}\kappa(\ell_{\eta^{\varepsilon}}^{\varepsilon}
-\ell_{s}^{\varepsilon})}\mathrm{d}\ell_{s}^{\varepsilon}\right\}\right]
\leq C,
\end{aligned}
\end{equation}
where $\kappa^*$ satisfies formula \eqref{A.17} below.
\end{lemma}
\begin{proof}
  Now, we fix $\kappa\in (0,\tilde{\kappa}], \varepsilon\in (0,1]$ and $\ell\in \mathbb{S}_1$.
Let $\gamma^{\varepsilon}$ be the inverse function of $\ell^{\varepsilon}$. By a change of variable, for $t\geq\ell_0^{\varepsilon}$, we set $(X_t^\varepsilon,Y_t^\varepsilon)
:=(w_{\gamma_t^{\varepsilon}}^{\varepsilon}
,\theta_{\gamma_t^{\varepsilon}}^{\varepsilon})$, and they satisfy the following stochastic equations
\begin{equation*}
\begin{aligned}
X_t^\varepsilon&=w_0
+\int_{\ell_0^\varepsilon}^t\left[\nu_1 \Delta X_s^\varepsilon
-B(X_s^\varepsilon,X_s^\varepsilon)
+g\partial_x \Delta Y_s^\varepsilon\right]
\dot{\gamma}_s^\varepsilon \mathrm{d}s,
\\
Y_t^\varepsilon&=\theta_0
+\int_{\ell_0^\varepsilon}^t\left[\nu_2 \Delta Y_s^\varepsilon
-B(X_s^\varepsilon,Y_s^\varepsilon)
\right]
\dot{\gamma}_s^\varepsilon \mathrm{d}s
+\sigma_{\theta}(W_t-W_{\ell_0^\varepsilon}).
\end{aligned}
\end{equation*}
Applying It\^{o}'s formula, we have
\begin{equation*}
\begin{aligned}
\mathrm{d}(\zeta^*\|X_t^\varepsilon\|^2)&=
-2\zeta^*\nu_1 \|X_t^\varepsilon\|_1^2\dot{\gamma}_t^\varepsilon \mathrm{d}t
+2\zeta^*\langle g\partial_x Y_t^\varepsilon, X_t^\varepsilon\rangle \dot{\gamma}_t^\varepsilon \mathrm{d}t,
\\
\mathrm{d}\|Y_t^\varepsilon\|^2
&=-2\nu_2\|Y_t^\varepsilon\|_1^2 \dot{\gamma}_t^\varepsilon \mathrm{d}t
+2\langle Y_t^\varepsilon,\sigma_{\theta} \mathrm{d}W_t\rangle+\mathfrak B_0 \mathrm{d}t.
\end{aligned}
\end{equation*}
By using Poincar\'{e}'s inequality, we get
\begin{equation}\label{A.1}
2\kappa\zeta^*\langle g\partial_x Y_t^\varepsilon,X_t^\varepsilon\rangle\leq \kappa\nu_1\zeta^*\|X_t^\varepsilon\|_1^2
+\kappa\nu_2\|Y_t^\varepsilon\|_1^2,
\end{equation}
and
\begin{equation*}
\nu(\zeta^*\|X_{t}^{\varepsilon}\|^{2}
+\|Y_{t}^{\varepsilon}\|^{2})
\leq \nu_{1}\zeta^* \|X_{t}^{\varepsilon}\|_1^{2}
+\nu_2\|Y_{t}^{\varepsilon}\|_1^{2}.
\end{equation*}
Thus, from the above estimates, we derive that
\begin{equation*}
\begin{aligned}
&\mathrm{d}\left[\kappa(\zeta^*\|X_{t}^{\varepsilon}\|^{2}
+\|Y_{t}^{\varepsilon}\|^{2})
e^{\nu\gamma_{t}^{\varepsilon}-8\mathfrak{B}_{0}\kappa t} \right] \\
=& e^{\nu\gamma_{t}^{\varepsilon}-8\mathfrak{B}_{0}\kappa t}
\big[-2\kappa(\nu_1\zeta^*\|X_{t}^{\varepsilon}\|_1^{2}
+\nu_2\|Y_{t}^{\varepsilon}\|_{1}^{2})
\dot{\gamma}_{t}^{\varepsilon}\mathrm{d}t
+2\kappa\zeta^*\langle g\partial_x Y_t^\varepsilon,X_t^\varepsilon\rangle
\dot{\gamma}_{t}^{\varepsilon}\mathrm{d}t
\\
&+2\kappa\langle Y_{t}^{\varepsilon},\sigma_{\theta} \mathrm{d}W_{t}\rangle+\kappa\mathfrak{B}_{0}\mathrm{d}t\big]
+\kappa(\zeta^*\|X_{t}^{\varepsilon}\|^{2}
+\|Y_{t}^{\varepsilon}\|^{2})e^{\nu\gamma_{t}^{\varepsilon}-8\mathfrak{B}_{0}\kappa t}\big(\nu\dot{\gamma}_{t}^{\varepsilon}-8\mathfrak{B}_{0}\kappa\big)\mathrm{d}t \\
\leq& \kappa e^{\nu\gamma_{t}^{\varepsilon}-8\mathfrak{B}_{0}\kappa t}(-\nu_1\zeta^*\|X_{t}^{\varepsilon}\|_1^2
-\nu_2\|Y_{t}^{\varepsilon}\|_{1}^{2}
+\nu(\zeta^*\|X_{t}^{\varepsilon}\|^{2}
+\|Y_{t}^{\varepsilon}\|^{2}))\dot{\gamma}_{t}^{\varepsilon}dt
+\kappa\mathfrak{B}_{0}e^{\nu\gamma_{t}^{\varepsilon}-8\mathfrak{B}_{0}\kappa t}\mathrm{d}t\\
&+2\kappa e^{\nu\gamma_{t}^{\varepsilon}-8\mathfrak{B}_{0}\kappa t}\langle Y_{t}^{\varepsilon},\sigma_{\theta} \mathrm{d}W_{t}\rangle
-8\mathfrak{B}_{0}\kappa^{2}(\|X_{t}^{\varepsilon}\|^{2}
+\|Y_{t}^{\varepsilon}\|^{2})e^{\nu\gamma_{t}^{\varepsilon}-8\mathfrak{B}_{0}\kappa t}\mathrm{d}t.
\end{aligned}
\end{equation*}
The above inequality can be rewritten as
\begin{equation}\label{A.7}
\begin{aligned}
&\kappa(\zeta^*\|X_{t}^{\varepsilon}\|^{2}+\|Y_{t}^{\varepsilon}\|^{2})
+\kappa^* \int_{\ell_{0}^{\varepsilon}}^{t} e^{-\nu(\gamma_{t}^{\varepsilon}-\gamma_{s}^{\varepsilon})
+8\mathfrak{B}_{0}\kappa(t-s)}
(\zeta^*\|X_{s}^{\varepsilon}\|_1^{2}+\|Y_{s}^{\varepsilon}\|_1^{2})
\dot{\gamma}_{s}^{\varepsilon}\mathrm{d}s\\
\leq& \kappa(\zeta^*\|w_0\|^2+\|\theta_0\|^2)e^{-\nu\gamma_{t}^{\varepsilon}
+8\mathfrak{B}_{0}\kappa t}
+\kappa\mathfrak{B}_{0}\int_{\ell_{0}^{\varepsilon}}^{t}
e^{-\nu(\gamma_{t}^{\varepsilon}-\gamma_{s}^{\varepsilon})
+8\mathfrak{B}_{0}\kappa(t-s)}\mathrm{d}s+\tilde{M}_{t},
\end{aligned}
\end{equation}
where
\begin{equation*}
\begin{aligned}\tilde{M}_{t}
=\tilde{M}_{t}^{\kappa,\varepsilon}
&=2\kappa\int_{\ell_{0}^{\varepsilon}}^{t}e^{-\nu(\gamma_{t}^{\varepsilon}-\gamma_{s}^{\varepsilon})+8\mathfrak{B}_{0}\kappa(t-s)}\langle Y_{s}^{\varepsilon},\sigma_{\theta} \mathrm{d}W_{s}\rangle
\\
&-8\mathfrak{B}_{0}\kappa^{2}
\int_{\ell_{0}^{\varepsilon}}^{t}\|Y_{s}^{\varepsilon}\|^{2}
e^{-\nu(\gamma_{t}^{\varepsilon}-\gamma_{s}^{\varepsilon})
+8\mathfrak{B}_{0}\kappa(t-s)}\mathrm{d}s,
\end{aligned}
\end{equation*}
and
\begin{equation}\label{A.17}
\begin{aligned}
\kappa^*(\zeta^*\|X_{t}^{\varepsilon}\|_1^{2}
+\|Y_{t}^{\varepsilon}\|_1^{2})
:=\kappa\nu_1\zeta^*\|X_{t}^{\varepsilon}\|_1^2
+\kappa\nu_2\|Y_{t}^{\varepsilon}\|_{1}^{2}
-\kappa\nu(\zeta^*\|X_{t}^{\varepsilon}\|^{2}
+\|Y_{t}^{\varepsilon}\|^{2})
\geq0.
\end{aligned}
\end{equation}

From \cite[Lemma A.5]{Peng-2024}, we know that
\begin{equation}\label{A.5}
\mathbb{E}^{\mu_{\mathbb{W}}}\left[
\exp\{\tilde{M}_{\ell_{\eta^{\varepsilon}}^{\varepsilon}}\}\right]
\leq C.
\end{equation}
Replacing $t$ in \eqref{A.7} by $\ell_{\eta^{\varepsilon}}^{\varepsilon}$, it follows that
\begin{align}\label{A.6}
&\mathbb{E}^{\mu_{\mathbb{W}}}
\bigg[\exp\bigg\{\kappa(\zeta^*
\|X_{\ell_{\eta^{\varepsilon}}^\varepsilon}
^{\varepsilon}\|^{2}
+\|Y_{\ell_{\eta^{\varepsilon}}^\varepsilon}
^{\varepsilon}\|^{2})+\kappa^*\int_{\ell_{0}^{\varepsilon}}
^{\ell_{\eta^{\varepsilon}}^{\varepsilon}}e^{-\nu(\eta^{\varepsilon}
-\gamma_{s}^{\varepsilon})
+8\mathfrak{B}_{0}\kappa(\ell_{\eta^{\varepsilon}}^{\varepsilon}-s)}
(\zeta^*\|X_{s}^{\varepsilon}\|_{1}^{2}
+\|Y_{s}^{\varepsilon}\|_{1}^{2})\dot{\gamma}_{s}^{\varepsilon}\mathrm{d}s \notag\\
&~~~-\kappa(\zeta^*\|w_0\|^2+\|\theta_0\|^2)e^{-\nu\eta^{\varepsilon}
+8\mathfrak{B}_{0}\kappa\ell_{\eta^{\varepsilon}}^{\varepsilon}}
-\kappa\mathfrak{B}_{0}\int_{\ell_{0}^{\varepsilon}}^{\ell_{\eta^{\varepsilon}}^{\varepsilon}}
e^{-\nu(\eta^{\varepsilon}-\gamma_{s}^{\varepsilon})
+8\mathfrak{B}_{0}\kappa(\ell_{\eta^{\varepsilon}}^{\varepsilon}-s)}\mathrm{d}s\Big\} \bigg] \\
&\leq\mathbb{E}^{\mu_{\mathbb{W}}}
\bigg[\exp\left\{\tilde{M}_{\ell_{\eta^{\varepsilon}}^{\varepsilon}}\right\}\bigg].\notag
\end{align}
Using the fact that $X_{\ell_{\eta^{\varepsilon}}^{\varepsilon}}^{\varepsilon}
=w_{\eta^{\varepsilon}}^{\varepsilon},
Y_{\ell_{\eta^{\varepsilon}}^{\varepsilon}}^{\varepsilon}
=\theta_{\eta^{\varepsilon}}^{\varepsilon}$, and
$\gamma_{s}^{\varepsilon}|_{s=\ell_{r}^{\varepsilon}}=r$, we obtain
\begin{equation*}
\begin{aligned}
&\int_{\ell_{0}^{\varepsilon}}^{\ell_{\eta^{\varepsilon}}^{\varepsilon}}
e^{-\nu(\eta^{\varepsilon}-\gamma_{s}^{\varepsilon})
+8\mathfrak{B}_{0}\kappa(\ell_{\eta^{\varepsilon}}^{\varepsilon}-s)}
(\zeta^*\|X_{s}^{\varepsilon}\|_{1}^{2}
+\|Y_{s}^{\varepsilon}\|_{1}^{2})\dot{\gamma}_{s}^{\varepsilon}\mathrm{d}s
=\int_{0}^{\eta^{\varepsilon}}e^{-\nu(\eta^{\varepsilon}-r)
+8\mathfrak{B}_{0}\kappa(\ell_{\eta^{\varepsilon}}^{\varepsilon}-\ell_{r}^{\varepsilon})}
\||U_{r}^{\varepsilon}|\|_{1}^{2}\mathrm{d}r,
\\
&\int_{\ell_{0}^{\varepsilon}}^{\ell_{\eta^{\varepsilon}}^{\varepsilon}}
e^{-\nu(\eta^{\varepsilon}-\gamma_{s}^{\varepsilon})
+8\mathfrak{B}_{0}\kappa(\ell_{\eta^{\varepsilon}}^{\varepsilon}-s)}{d}s
=\int_{0}^{\eta^{\varepsilon}}e^{-\nu(\eta^{\varepsilon}-r)
+8\mathfrak{B}_{0}\kappa(\ell_{\eta^{\varepsilon}}^{\varepsilon}
-\ell_{r}^{\varepsilon})}\mathrm{d}\ell_{r}^{\varepsilon}.
\end{aligned}\end{equation*}
Together with the above formulas and \eqref{A.5}-\eqref{A.6}, we can derive the desired result. This completes the proof.
\end{proof}

With the groundwork laid by the previous lemmas, we can derive \eqref{2.9}-\eqref{moment.4} following the arguments in the proof of \cite[Appendix A]{Peng-2024}.
Finally, we turn to prove \eqref{moment.5}.

\begin{proof}[Proof of \eqref{moment.5}] Applying It\^{o}'s formula to $\zeta^*\|X_{s}^{\varepsilon}\|^{2n}+\|Y_{s}^{\varepsilon}\|^{2n}$ yields
\begin{align}\label{A.8}
&~~\zeta^*\|X_{s}^{\varepsilon}\|^{2n}+\|Y_{s}^{\varepsilon}\|^{2n}+2\nu_1\zeta^* n \int^s_{\ell_{0}^{\varepsilon}}
\|X_{u}^{\varepsilon}\|^{2n-2}\|X_{u}^{\varepsilon}\|_{1}^{2}
\dot{\gamma}_{u}^{\varepsilon}\mathrm{d}u
+2\nu_2 n \int^s_{\ell_{0}^{\varepsilon}}
\|Y_{u}^{\varepsilon}\|^{2n-2}\|Y_{u}^{\varepsilon}\|_{1}^{2}
\dot{\gamma}_{u}^{\varepsilon}\mathrm{d}u \notag\\
&\leq \zeta^*\|w_{0}\|^{2n}+\|\theta_{0}\|^{2n}+2\zeta^* n\int_{\ell_{0}^{\varepsilon}}^{s}
\|X_{u}^{\varepsilon}\|^{2n-2}\langle g\partial_xY_{u}^{\varepsilon},
X_{u}^{\varepsilon}\rangle\dot{\gamma}_{u}^{\varepsilon}\mathrm{d}u
+2 n\int_{\ell_{0}^{\varepsilon}}^{s}
\|Y_{u}^{\varepsilon}\|^{2n-2}\langle Y_{u}^{\varepsilon},
\sigma_\theta \mathrm{d}W_{u}\rangle \notag\\
&~~~+n\int_{\ell_{0}^{\varepsilon}}^{s}\|Y_{u}^{\varepsilon}\|^{2n-2}\mathfrak{B}_{0}\mathrm{d}u
+2n(n-1)\int_{\ell_{0}^{\varepsilon}}^{s}\|Y_{u}^{\varepsilon}\|^{2n-4}
\sum_{j\in\mathcal{Z}}\langle Y_{u}^{\varepsilon},\sigma_j^m\rangle^2 (\alpha_j^m)^2 \mathrm{d}u,
\end{align}
we use a similar estimate as in \eqref{A.1}, and then take expectations with respect to $\mathbb{P}^{\mu_{\mathbb{W}}}$ to derive
\begin{equation*}
\begin{aligned}
&~~\mathbb{E}^{\mu_{\mathbb{W}}}\left[\zeta^*\|X_{s}^{\varepsilon}\|^{2n}
+\|Y_{s}^{\varepsilon}\|^{2n}\right]\\
&\leq \zeta^*\|w_{0}\|^{2n}+\|\theta_{0}\|^{2n}
+C_n\int_{\ell_{0}^{\varepsilon}}^{s}\mathbb{E}^{\mu_{\mathbb{W}}}
\left[\|Y_{u}^{\varepsilon}\|^{2n-2}\right]\mathrm{d}u
+C_n\int_{\ell_{0}^{\varepsilon}}^{s}\mathbb{E}^{\mu_{\mathbb{W}}}
\left[\|X_{u}^{\varepsilon}\|^{2n}\right]\mathrm{d}u\\
&\leq \zeta^*\|w_{0}\|^{2n}+\|\theta_{0}\|^{2n}
+\frac{1}{2}\sup\limits_{u\in[\ell_0^\varepsilon,s]}
\mathbb{E}^{\mu_{\mathbb{W}}}
\left[\|Y_{u}^{\varepsilon}\|^{2n}\right]+C_n t^n.
\end{aligned}
\end{equation*}
Then we can get that for any $t\geq \ell_0^{\varepsilon}$,
\begin{equation*}
\begin{aligned}
&\sup\limits
_{s\in[\ell_0^\varepsilon,t]}
\mathbb{E}^{\mu_{\mathbb{W}}}\left[\zeta^*\|X_{s}^{\varepsilon}\|^{2n}
+\|Y_{s}^{\varepsilon}\|^{2n}\right]
\leq2\| U_{0}\|^{2n}+C_n t^{n}.
\end{aligned}
\end{equation*}
Using \eqref{A.8} and the Burkholder-Davis-Gundy inequality, we obtain
\begin{align*}
&\mathbb{E}^{\mu_{\mathbb{W}}}\left[\sup\limits
_{s\in[\ell_0^\varepsilon,t]}\left(\zeta^*\|X_{s}^{\varepsilon}\|^{2n}
+\|Y_{s}^{\varepsilon}\|^{2n}\right)\right]\\
&\leq \|U_0\|^{2n}
+2 n\mathbb{E}^{\mu_{\mathbb{W}}}\left[\sup\limits
_{s\in[\ell_0^\varepsilon,t]}\left|\int_{\ell_{0}^{\varepsilon}}^{s}
\|Y_{u}^{\varepsilon}\|^{2n-2}\langle Y_{u}^{\varepsilon},
\sigma_\theta \mathrm{d}W_{u}\rangle\right|\right]
+C_n \mathbb{E}^{\mu_{\mathbb{W}}}\left[
\int_{\ell_{0}^{\varepsilon}}^{s}\|Y_{u}^{\varepsilon}\|^{2n-2}\mathrm{d}u\right]\\
&\leq \|U_0\|^{2n}
+C_n \left(\mathbb{E}^{\mu_{\mathbb{W}}}\left[\sup\limits
_{s\in[\ell_0^\varepsilon,t]}\left|\int_{\ell_{0}^{\varepsilon}}^{s}
\|Y_{u}^{\varepsilon}\|^{2n-2}\langle Y_{u}^{\varepsilon},
\sigma_\theta \mathrm{d}W_{u}\rangle\right|^2\right]\right)^{1/2}+C_n \mathbb{E}^{\mu_{\mathbb{W}}}\left[
\int_{\ell_{0}^{\varepsilon}}^{s}\|Y_{u}^{\varepsilon}\|^{2n-2}\mathrm{d}u\right]\\
&\leq \|U_0\|^{2n}
+C_n \left(\mathbb{E}^{\mu_{\mathbb{W}}}\left[\int_{\ell_{0}^{\varepsilon}}^{t}
\|Y_{u}^{\varepsilon}\|^{4n-2}\mathrm{d}u\right]\right)^{1/2}
+C_n \mathbb{E}^{\mu_{\mathbb{W}}}\left[
\int_{\ell_{0}^{\varepsilon}}^{s}\|Y_{u}^{\varepsilon}\|^{2n-2}\mathrm{d}u\right]\\
&\leq \|U_0\|^{2n}
+C_n \left[\| U_{0}\|^{4n-2}t+t^{2n}\right]^{1/2}
+C_n \left[ \| U_{0}\|^{2n-2}t+t^n\right]\\
&\leq C_n(1+t)\|U_0\|^{2n}+C_n(1+t^n).
\end{align*}
Following the argument in \cite[Appendix A.3]{Peng-2024}, we obtain the desired result. This completes the proof.
\end{proof}

\section{The invertibility of the Malliavin matrix $\mathcal{M}_{0,t}$}
For any $\alpha>0$, $N\in\mathbb{N}$, we define
$$
\mathcal{S}_{\alpha,N}=\{\phi\in H:\|P_{N}\phi \| \geq \alpha , \| \phi \| = 1\}.
$$

This section aims to establish a significant proposition that elucidates the probability of eigenvectors with sizable projections in unstable directions having small eigenvalues.  Broadly speaking,
this result ensures the invertibility of the Malliavin matrix within the space spanned by these unstable directions. Given the finite-dimensional nature of this space in the current context, we can formulate a control problem using the Malliavin integration by parts formula to derive gradient estimates on the Markov semigroup, which are instrumental in demonstrating ergodicity.

\begin{proposition}\label{propo3.4} For any $\alpha \in ( 0, 1] , N\in \mathbb{N}$ and $U_0\in H$, one has
\begin{equation}\label{3.4}
\mathbb{P}\Big(\inf_{\phi\in\mathcal{S}_{\alpha,N}}
\langle\mathcal{M}_{0,\eta}\phi,\phi\rangle=0\Big)=0.
\end{equation}
\end{proposition}

Before starting to prove this proposition, we need to make some necessary preparations regarding L\'{e}vy processes and perform some technical work.

Throughout this section, we denote $f(s)-f(s-)$ by $\Delta f(s)$.
Recall that the probability space $(\Omega,\mathcal{F},\mathbb{P})$ is as mentioned in Section 2, and $L_t=(W_{S_t}^{k,m})_{k\in \mathcal Z,m\in \{0,1\}}$ is a $d$-dimensional L\'{e}vy process with a $\sigma$-finite intensity measure $\nu_L$.

The next lemma can be seen as the pure jump version of \cite[Theorem 7.1]{MH-2011}, which deals with the Wiener case.
\begin{lemma}[\cite{Peng-2024}]\label{lemma3.2}
There exists a set $\Omega_0\subseteq \Omega $ depending only on the
$d$-dimensional  L\'{e}vy process $W_{S_t}$
such that   $\mathbb  P(\Omega_0)=1$
 and for any   $\omega\in \Omega _0$, if  the following three conditions are satisfied:

$(1)$ $a( \omega) , b( \omega) \in [ 0, \infty )$ and $a( \omega) < b( \omega)$;

$(2)$ $g_0, g_{k,m}( \omega, \cdot ) : [ a( \omega) , b( \omega ) ] \rightarrow \mathbb{R} ,k\in \mathcal Z,m\in \{0,1\} $, are continuous functions;

$(3)$
\begin{equation}\label{3.1}
g_0(\omega,r)+\sum_{k \in \mathcal Z,m\in\{0,1
\}}g_{k,m}(\omega,r)W_{S_t}^{k,m}=0,\quad\forall r\in[a(\omega),b(\omega)].
\end{equation}
Then
\begin{equation*}
g_0(\omega,r)=0~\text{and}~g_{k,m}(\omega,r)=0,\quad\forall r\in[a(\omega),b(\omega)],\quad k\in\mathcal{Z},m\in\{0,1\}.
\end{equation*}
\end{lemma}

Recall that the assumption in Condition \ref{Condition-2.2}: $\nu_S((0,\infty))=\infty$. For the process $S_t,t\geq0$, by \cite[Lemma 3.1]{Peng-2024},
we have the following lemma.
\begin{lemma}\label{lemma3.3}
Let  $\tilde{\mathbb S}_0= \big\{\ell:  \{r:\Delta\ell_r:=\ell_r-\ell_{r-}>0\} \text{ is dense in }   [0,\infty)\big\} $
and  denote  $\tilde \Omega_0 =\mathbb W \times  \tilde{\mathbb S}_0.$
Then
\begin{equation*}
\mathbb{P}^{\mu_{\mathbb{S}}}\left(  \tilde{\mathbb S}_0    \right)=1 \text{ and } \mathbb P(\tilde \Omega_0) =1.
\end{equation*}
\end{lemma}

We introduce an operation $[E_1,E_2]$, which is called as the Lie bracket of the two `vector fields' $E_1$, $E_2$, specifically, for any Fr\'{e}chet differentiable $E_1, E_2:H\rightarrow H$,
\begin{equation*}
[E_1,E_2](U):=\nabla E_2(U)E_1(U)-\nabla E_1(U)E_2(U).
\end{equation*}

Next, we perform some Lie bracket computations to prepare for the subsequent analysis. Interested readers can refer to \cite[Section 5]{Foldes-2015} for a detailed explanation.
For any $j\in \mathbb{Z}^2$ we define $j^{\bot}:=(-j_2,j_1)$, for $m,m'\in\{0,1\}$, we introduce
\begin{equation}
\begin{aligned}\label{K.01}
Y_{j}^{m}(U):=[F(U),\sigma_j^m]=A\sigma_j^m+B(\sigma_j^m,U)+B(U,\sigma_j^m)-G\sigma_j^m,
\end{aligned}
\end{equation}
and
\begin{equation}
\begin{aligned}\label{K.02}
Z_{j}^{m}(U)&:=[F(U),Y_j^m(U)]
=\nabla Y_j^m(U)F(U)-\nabla F(U)Y_j^m(U)\\
&=B(F(U),\sigma_j^m)+\nu_2^2|j|^4\sigma_j^m+(-1)^m(\nu_1+\nu_2)gj_1|j|^2\psi_j^{m+1}+A(B(U,\sigma_j^m)) \\
&~~+(-1)^mgj_1B(\psi_j^{m+1},U)-B(U,-\nu_2|j|^2\sigma_j^m+(-1)^{m+1}gj_1\psi_j^{m+1})
\\
&~~+B(U,B(U,\sigma_j^m))-GB(U,\sigma_j^m),  \end{aligned}\end{equation}
%based on the above two equalities,
and for all $k\in\mathbb{Z}_+^2$, we have
\begin{equation}
\begin{aligned}\label{K.03}
[Z_{j}^{m}(U),\sigma_k^{m'}]&=
g(-1)^{(m+1)(m'+1)}\frac{j^{\bot}\cdot k}{2}
\left[(-1)^{m'}b(j,k)\sigma_{j-k}^{m+m'+1}
-a(j,k)\sigma_{j+k}^{m+m'+1}\right],
\end{aligned}\end{equation}
and
\begin{equation}\label{K.04}
\begin{aligned}
[Z_{j}^{m}(U),Y_k^{m'}(U)]&=[[F(U),Y_j^{m}(U)],Y_k^{m'}(U)]\\
&=[[Z_{j}^{m}(U),F(U)],\sigma_k^{m'}]
-[[Z_{j}^{m}(U),\sigma_k^{m'}],F(U)] \\
&=(-1)^{m+m'+1}g^2j_1k_1\left(B(\psi_j^{m+1},\psi_k^{m'+1})
+B(\psi_k^{m'+1},\psi_j^{m+1})\right) \\
&~~~+(-1)^{m'}gk_1 GB(\psi_k^{m'+1},\sigma_j^m)
+H_{j,k}^{m,m'}(U),
\end{aligned}\end{equation}
where
\begin{equation*}
a(j,k):=\frac{j_1}{|j|^2}+\frac{k_1}{|k|^2},~~
b(j,k):=\frac{j_1}{|j|^2}-\frac{k_1}{|k|^2},
\end{equation*}
and $U\mapsto H_{j,k}^{m,m'}(U)$ is affine and it is completely concentrated in the $\theta$ component.

From these relations, the following proposition follows easily.
\begin{proposition}\label{proper5.2}
Let $j,k\in\mathbb{Z}_+^2$, $a(j,k),b(j,k)$ be as in the above. Then, there are
\begin{equation}
\begin{aligned}
g(j^\perp\cdot k)a(j,k)\sigma_{j+k}^0 &=-[Z_j^0(U),\sigma_k^1]-[Z_j^1(U),\sigma_k^0], \\
g(j^\perp\cdot k)a(j,k)\sigma_{j+k}^1 &=[Z_j^0(U),\sigma_k^0]-[Z_j^1(U),\sigma_k^1], \\
g(j^\perp\cdot k)b(j,k)\sigma_{j-k}^0 &=[Z_j^1(U),\sigma_k^0]-[Z_j^0(U),\sigma_k^1], \\
g(j^\perp\cdot k)b(j,k)\sigma_{j-k}^1 &=-[Z_j^1(U),\sigma_k^1]-[Z_j^0(U),\sigma_k^0].
\end{aligned}
\end{equation}
\end{proposition}

We denote that $\bar{U}=U-\sigma_\theta W_{S_t}$, then
\begin{equation}\label{K.2}
 \partial_t \bar{U}=F(U)=F(\bar{U}+\sigma_\theta W_{S_t}),
~~~\bar{U}(0)=U_0.
\end{equation}

\begin{lemma}\label{5.25}
We define the error term in the $\psi$-directions as follows
 \begin{equation}\label{5.26}
\begin{aligned}
&J_{j,m}(U)=
\begin{cases} (-1)^m\frac{\nu_2|j|^2}{gj_1}\sigma_j^{m+1}+(-1)^m\frac{1}{gj_1}B(U,\sigma_j^{m+1})
&\text{if }j_1\neq0,m\in\{0,1\},\\ \frac{1+|j|^2}{g^2|j|^3}(-H_{j+e_1,e_1}^{0,0}(U)-H_{j+e_1,e_1}^{1,1}(U))
&\text{if }j_1=0, m=0,\\ \frac{1+|j|^2}{g^2|j|^3}(-H_{j+e_1,e_1}^{0,1}(U)+H_{j+e_1,e_1}^{1,0}(U))
&\text{if }j_1=0, m=1,
\end{cases}
\end{aligned}
\end{equation}
then we have for each $j\in\mathbb{Z}_{+}^{2},m\in\{0,1\}$,
\begin{equation*}\begin{aligned}
&~\psi_j^m+J_{j,m}(U)\\
&=\begin{cases}\frac{(-1)^m}{gj_1}Y_j^{m+1}(U)&\text{if }j_1\neq0,m\in\{0,1\},
\\
\frac{1+|j|^2}{g^2|j|^3}(-[Z_{j+e_1}^0(U),Y_{e_1}^0(U)]-[Z_{j+e_1}^1(U),Y_{e_1}^1(U)])
&\text{if }j_1=0, m=0,
\\
\frac{1+|j|^2}{g^2|j|^3}([Z_{j+e_1}^1(U),Y_{e_1}^0(U)]-[Z_{j+e_1}^0(U),Y_{e_1}^1(U)])
&\text{if }j_1=0, m=1.\end{cases}\end{aligned}\end{equation*}
\end{lemma}

\begin{lemma}\label{lemma0}
For any  $\omega=\mathrm{w}\times \ell\in $ $\tilde{\Omega}_0$,
\begin{equation}\label{K.1}
\langle\mathcal{M}_{0,\eta}\phi,\phi\rangle=0~~
\Rightarrow~~\sup_{t\in[\eta/2,\eta]}|\langle
K_{t,\eta}\phi,\sigma_k^m\rangle|(\omega)=0,
\end{equation}
for all $\phi\in \mathcal{S}_{\alpha,N}$, $k\in \mathcal{Z}$ and $m\in\{0,1\}$.
\end{lemma}
\begin{proof}
Recall that
\begin{equation*}
\langle\mathcal{M}_{0,\eta}\phi,\phi \rangle
=\sum_{k\in\mathcal{Z},m=\{0,1\}}
(\alpha_{k}^{m})^2 \int_{0}^{\eta} \langle K_{r,\eta} \phi, \sigma_{k}^m \rangle^{2}\mathrm{d}\ell_{r},
\end{equation*}
when $\langle\mathcal{M}_{0,\eta}\phi,\phi \rangle=0$, we have $\int_{0}^{\eta} \langle K_{r,\eta} \phi, \sigma_{k}^m \rangle^{2}\mathrm{d}\ell_{r}=0$. By Lemma \ref{lemma3.3}, for any $\ell\in\tilde{\Omega}_0$, $\{r:\Delta\ell_r:=\ell_r-\ell_{r-}>0\}$ is dense in $[0,\infty)$, since $\langle K_{r,\eta} \phi, \sigma_{k}^m \rangle$ is a nonnegative continuous function on $[0,\eta]$ and $\int_{0}^{\eta} \langle K_{r,\eta} \phi, \sigma_{k}^m \rangle^{2}\mathrm{d}\ell_{r}=0$, then $\langle K_{r,\eta} \phi, \sigma_{k}^m \rangle=0, r\in[0,\eta]$ holds on the set $\tilde{\Omega}_0$, which implies \eqref{K.1}.
%\begin{equation*}
%\sum_{r} \langle K_{r,\eta} \phi, \sigma_{k}^m \rangle^{2}\Delta\ell_{r}=0,
%\end{equation*}
%then on the set $\Omega_0$,
%it is not difficult to find that \eqref{K.1} holds and $\mathbb{P}(\Omega_0^c)=0$.
\end{proof}

\begin{lemma}\label{lemmaB}
Recall that  $\Omega_{0}$ is   given by Lemma \ref{lemma3.2}.
For $\omega\in\Omega_{0}$, the following
\begin{equation}\label{K.3}
\sup_{t\in[\eta/2,\eta]}|\langle
K_{t,\eta}\phi,\sigma_k^m\rangle|(\omega)=0~~
\Rightarrow~~\sup_{t\in[\eta/2,\eta]}|\langle
K_{t,\eta}\phi,Y_k^m(U)\rangle|(\omega)=0
\end{equation}
holds
for all $\phi\in \mathcal{S}_{\alpha,N}$, $k\in \mathcal{Z}$ and $m\in\{0,1\}$.
\end{lemma}
\begin{proof} Define $g_{\phi}(t):=\langle
K_{t,\eta}\phi,\sigma_k^m\rangle, \forall t\in[0,\eta]$.
Observe from \eqref{backward} that, for $\phi\in \mathcal{S}_{\alpha,N}$
\begin{equation*}
\begin{aligned}
g_{\phi}'(t)=\langle K_{t,\eta}\phi,[F(U),\sigma_k^m]\rangle=\langle K_{t,\eta}\phi,Y_k^m(U)\rangle,
\end{aligned}
\end{equation*}
since $g_{\phi}(t):=\langle K_{t,\eta}\phi,\sigma_k^m\rangle$ is a nonnegative continuous function on $[0,\eta]$, we can get that there exists a set $\Omega_0$, such that $g_{\phi}'(t)=\langle K_{t,\eta}\phi,Y_j^m(U)\rangle=0$ holds on the set $\Omega_0$.
\end{proof}

\begin{lemma}\label{lemmaC}  For $\omega\in\Omega_0$,
\begin{equation}
\begin{aligned}
\sup_{t\in[\eta/2,\eta]}
|\langle K_{t,\eta}\phi,Y_{j}^{m}(U)\rangle|(\omega)
=0
\Rightarrow
\begin{cases}
\sup\limits_{t\in[\eta/2,\eta]}
|\langle K_{t,\eta}\phi,Z_{j}^{m}(\bar{U})\rangle|(\omega)
=0,\\
\sup\limits_{k\in\mathcal{Z},l\in\{0,1\}}
\sup\limits_{t\in[\eta/2,\eta]}|
\langle K_{t,\eta}\phi,[Z_{j}^{m}(U),\sigma_{k}^{l}]\rangle|(\omega)
=0,
\end{cases}
\end{aligned}\end{equation}
for all $\phi\in \mathcal{S}_{\alpha,N},m\in\{0,1\}$ and $j\in\mathbb{Z}_+^2$.
\end{lemma}
\begin{proof}
Due to $B(U,\bar{U})=0$ if $U=(0,\theta)$, then $Y_{j}^m(U)=Y_{j}^m(\bar{U})$. Therefore,
for $t\in[\frac{\eta}{2},\eta]$, and for fixed $m\in\{0,1\}$, $\phi\in \mathcal{S}_{\alpha,N}$, let
\begin{equation*}
\begin{aligned}
g_{\phi}(t)&:=\langle K_{t,\eta}\phi,Y_j^m(U)\rangle
=\langle K_{t,\eta}\phi,Y_j^m(\bar{U})\rangle,\\
g_{\phi}'(t)&:=\langle K_{t,\eta}\phi,[F(U),Y_j^m(\bar{U})]\rangle
=\langle K_{t,\eta}\phi,Z_j^m(U)\rangle.
\end{aligned}
\end{equation*}
From the above formulas, since $\sup_{t\in[\eta/2,\eta]}g_{\phi}(t)=0$, it is not difficult to derive that on $\Omega_0$, one has
\begin{equation}\label{K.4}
\sup\limits_{t\in[\eta/2,\eta]}\langle K_{t,\eta}\phi,Z_j^m(U)\rangle=0.
\end{equation}
Next, by expanding $U=\bar{U}+\sigma W$, we can get that
\begin{equation}
Z_j^m(U)=Z_j^m(\bar{U})-\sum_{k\in\mathcal{Z},l\in\{0,1\}}
\alpha_k^l[Z_j^m(U),\sigma_k^l]W_{S_t}^{k,l}.
\end{equation}
Substituting the above formula into \eqref{K.4}, we obtain
\begin{equation}\label{K.41}
\sup_{t\in[\eta/2,\eta]}\left\{\langle K_{t,\eta}\phi,Z_j^m(\bar{U})\rangle
-\sum_{k\in\mathcal{Z},l\in\{0,1\}}
W_{S_t}^{k,l}\alpha_k^l\langle K_{t,\eta}\phi,[Z_j^m(U),\sigma_k^l]\rangle\right\}=0,
\end{equation}
since $\langle K_{t,\eta}\phi,Z_j^m(\bar{U})\rangle$ and $\langle K_{t,\eta}\phi,[Z_j^m(U),\sigma_k^l]\rangle$ are continuous, by Lemma \ref{lemma3.2}, we can derive $\sup\limits_{k\in\mathcal{Z},l\in\{0,1\}}
\sup\limits_{t\in[\eta/2,\eta]}|
\langle K_{t,\eta}\phi,[Z_{j}^{m}(U),\sigma_{k}^{l}]\rangle|
=0$ holds on $\Omega_0$, for all $\phi\in \mathcal{S}_{\alpha,N},m\in\{0,1\}$ and $j\in\mathbb{Z}_+^2$. This completes the proof.
\end{proof}

The next lemma corresponds to the brackets of the form $Y\rightarrow Z\rightarrow[Z,Y]$. For fixed $j\in \mathbb{Z}_+^2$, define $\mathcal{Z}_j$ as the union of $j$ with the set of points in $\mathbb{Z}_+^2$ adjacent to $j$, that is,
\begin{equation*}
\mathcal{Z}_j:=\{k\in \mathbb{Z}_+^2:k=j\pm m~\text{for some}~ m\in\{0\}\cup\mathcal{Z}\}.
\end{equation*}

\begin{lemma}\label{lemmaD}  Fix $j\in\mathbb{Z}_{+}^{2}$. For $\omega\in\Omega_0$,
\begin{equation}\label{K.05}
\begin{aligned}
&\sum_{i\in\mathcal{Z}_j,m\in\{0,1\}}\sup_{t\in[\eta/2,\eta]}
|\langle K_{t,\eta}\phi,Y_{i}^{m}(U)\rangle|(\omega)
=0\\
&\Rightarrow\sum_{\substack{k\in\mathcal{Z} \\ m,l\in\{0,1\}}}
\sup\limits_{t\in[\eta/2,\eta]}
|\langle K_{t,\eta}\phi,[Z_{j}^{m}(U),Y_k^l(U)]\rangle|(\omega)
=0,
\end{aligned}
\end{equation}
for every $\phi\in \mathcal{S}_{\alpha,N}$.
\end{lemma}
\begin{proof}
%By \eqref{K.04}, it suffices to find $\Omega_0$ satisfying $\mathbb{P}(\Omega_0^c)=0$ such that on $\Omega_0$,
If \eqref{K.05} is established, then for $j\in\mathbb{Z}_{+}^{2}$, we just need to prove that
\begin{equation}\label{K.5}
\sum_{\substack{k\in\mathcal{Z} \\ m,l\in\{0,1\}}}
\sup\limits_{t\in[\eta/2,\eta]}\left|\langle
K_{t,\eta}\phi,
[[Z_{j}^{m}(U),\sigma_k^l],F(U)]\rangle\right|=0,
\end{equation}
and
\begin{equation}\label{K.6}
\sum_{\substack{k\in\mathcal{Z} \\ m,l\in\{0,1\}}}
\sup\limits_{t\in[\eta/2,\eta]}\left|\langle
K_{t,\eta}\phi,
[[Z_{j}^{m}(U),F(U)],\sigma_k^l]\rangle\right|=0.
\end{equation}
In order to obtain \eqref{K.5}, by \eqref{K.03}, for $j\pm k\in\mathcal{Z}_j, k\in\mathcal{Z}$, one arrives at
\begin{equation}
\begin{aligned}
&\langle K_{t,\eta}\phi,
[[Z_{j}^{m}(U),\sigma_k^l],F(U)]\rangle\\
&\leq\left|g(-1)^{(m+1)(l+1)}\frac{j^{\bot}\cdot k}{2}\right|\left|\left\langle K_{t,\eta}\phi,
\left[-a(j,k)\sigma_{j+k}^{m+l+1}
+(-1)^{l}b(j,k)\sigma_{j-k}^{m+l+1},F(U)\right]\right\rangle\right|\\
&\leq C|j|\left(|\langle K_{t,\eta}\phi,Y_{j+k}^{m+l+1}(U)\rangle|
+|\langle K_{t,\eta}\phi,Y_{j-k}^{m+l+1}(U)\rangle|\right),
\end{aligned}
\end{equation}
thus \eqref{K.5} follows by the given known conditions in Lemma \ref{lemmaD}.

To prove \eqref{K.6}, for fixed $m\in\{0,1\}$, $j\in\mathbb{Z}_{+}^{2}$ and $\phi\in\mathcal{S}_{\alpha,N}$, let
\begin{equation*}
\begin{aligned}
g_{\phi}(t)
=\langle K_{t,\eta}\phi,Z_j^m(\bar{U})\rangle,~~~
g_{\phi}'(t)
=\langle K_{t,\eta}\phi,[F(U),Z_j^m(\bar{U})]\rangle.
\end{aligned}
\end{equation*}
By Lemma \ref{lemmaC}, we know that on $\Omega_0$, for each $m\in\{0,1\}$, and $\phi\in\mathcal{S}_{\alpha,N}$, one has
\begin{equation*}
\begin{aligned}
&\sup_{t\in[\eta/2,\eta]}
|\langle K_{t,\eta}\phi,Y_{j}^{m}(U)\rangle|
=0
\Rightarrow
\sup\limits_{t\in[\eta/2,\eta]}
|\langle K_{t,\eta}\phi,Z_{j}^{m}(\bar{U})\rangle|
=0.
\end{aligned}
\end{equation*}
Thus, we get on the set $\Omega_0$
\begin{equation*}
\begin{aligned}
\sup\limits_{t\in[\eta/2,\eta]}
|\langle K_{t,\eta}\phi,[Z_{j}^{m}(\bar{U}),F(U)]\rangle|
=0.
\end{aligned}
\end{equation*}
Using $U$ to replace $\bar{U}$ for the above equation yields
\begin{equation*}
\begin{aligned}
\sup\limits_{t\in[\eta/2,\eta]}
\left|\langle K_{t,\eta}\phi,[Z_{j}^{m}(U),F(U)]
\rangle
+\sum_{k\in\mathcal{Z},l\in\{0,1\}}
\alpha_k^lW_{S_t}^{k,l}
\langle K_{t,\eta}\phi,
[[Z_j^m(U),\sigma_k^l],F(U)]
\rangle\right|
=0.
\end{aligned}
\end{equation*}
Since \eqref{K.5} has already been verified, and combining the above formula, thus for each $m\in\{0,1\}$, $j\in\mathbb{Z}_{+}^{2}$ and each $\phi\in\mathcal{S}_{\alpha,N}$, there holds
\begin{equation*}
\begin{aligned}
\sup\limits_{t\in[\eta/2,\eta]}
|\langle K_{t,\eta}\phi,[Z_{j}^{m}(U),F(U)]
\rangle|
=0.
\end{aligned}
\end{equation*}
Then, similarly as in Lemma \ref{lemmaC}, we expand $[Z_j^m(U),F(U)]$ with respect to $U=\bar{U}+\sigma_{\theta} W_{S_t}$ and again we use Lemma \ref{lemma3.2} to get
\begin{equation*}
\begin{aligned}
\sup\limits_{l\in\{0,1\},k\in\mathcal{Z}}\sup\limits_{t\in[\eta/2,\eta]}
|\langle K_{t,\eta}\phi,
[[Z_j^m(\bar{U}),F(\bar{U})],\sigma_k^l]
\rangle|=0.
\end{aligned}
\end{equation*}
We know the fact that $B(U,\tilde{U})=0$ if $U=(0,\theta)$, then combining with \eqref{K.01}, we can get that for any $k,k'\in\mathcal{Z}\subset\mathbb{Z}_+^2$, $m,m'\in\{0,1\}$, there holds
$$
[[F(U),\sigma_k^m],\sigma_{k'}^{m'}]
=B(\sigma_k^m,\sigma_{k'}^{m'})
+B(\sigma_{k'}^{m'},\sigma_k^m)=0,
$$
combining the above formula with \eqref{K.01} and \eqref{K.03}, we note that
\begin{equation*}
[[Z_j^m(\bar{U}),F(\bar{U})],\sigma_{k}^{l}]
=[[Z_j^m(U),F(U)],\sigma_{k}^{l}],
\end{equation*}
then we derive that \eqref{K.6} holds on the set $\Omega_0$. This completes the proof.
\end{proof}

\begin{lemma}\label{lemmaE}
For $\phi\in \mathcal{S}_{\alpha,N}$, $j\in \mathbb{Z}_{+}^2$, $m\in\{0,1\}$ and $\omega\in\Omega_0 \cap \tilde{\Omega}_0$, it holds that
\begin{equation}\label{K.7}
\begin{aligned}
&
\langle\mathcal{M}_{0,\eta}\phi,\phi\rangle
=0
&\Rightarrow
\begin{cases}
\sup\limits_{t\in[\eta/2,\eta]}
|\langle K_{t,\eta}\phi,\sigma_{j}^{m}\rangle|(\omega)
=0,\\
\sup\limits_{t\in[\eta/2,\eta]}
|\langle K_{t,\eta}\phi,Y_{j}^{m}(U)\rangle|(\omega)
=0.
\end{cases}
\end{aligned}\end{equation}
\end{lemma}
\begin{proof}
%In order to span sets for $H^N$ from brackets and associated tails, we
Denote
\begin{equation*}
\mathcal{I}_N:=\{j\in\mathbb{Z}_+^2:|j_1|+|j_2|\leq N+1\}\setminus\{(0,N+1),(0,N),(N+1,0),(N,0)\}.
\end{equation*}
We proceed by induction in $N\geq 1$.

For the first step, $N=1$, we show that the result holds on the set $\mathcal{I}_1=\{(1,1),(-1,1)\}$.
By Lemma \ref{lemma0} and Lemma \ref{lemmaB}, we know that \eqref{K.7} holds on the set $\Omega_0 \cap \tilde{\Omega}_0$, for each $j\in\mathcal{Z}=\{(0,1),(1,0)\},m\in\{0,1\}$.

We then establish \eqref{K.7} for $j\in\mathcal{I}_1$. By Lemma \ref{lemmaC}, one has
\begin{equation*}
\sup_{\substack{m,l\in\{0,1\} \\ j',k\in\mathcal{Z}}
}
\sup_{t\in[\eta/2,\eta]}|\langle K_{t,\eta}\phi,[Z_{j'}^m(U),\sigma_k^l]\rangle|
=0.
\end{equation*}
Using Proposition \ref{proper5.2} with $j'=(0,1),k=(1,0)$, and all combinations of $m,l\in\{0,1\}$, since $(j'^{\bot}\cdot k),a(j',k),b(j',k)\neq0$ and $j'\pm k\in\mathcal{I}_1$, we obtain that for each $j\in\mathcal{I}_1,m\in\{0,1\}$,
\begin{equation*}
\sup\limits_{t\in[\eta/2,\eta]}
|\langle K_{t,\eta}\phi,\sigma_{j}^{m}\rangle|
=0,
\end{equation*}
by Lemma \ref{lemmaB} we can complete the proof of the case $N=1$.

Next we use induction to obtain the desired result. That is, we assume that \eqref{K.7} holds for each $j\in\mathcal{I}_{N-1}$ ($N-1\geq 1$), our aim is to show that \eqref{K.7} holds true for $j\in\mathcal{I}_{N}$ on the set $\Omega_0 \cap \tilde{\Omega}_0$. We introduce the set
\begin{equation*}
\mathcal{B}_{N-1}:=\{j^{\prime}
=(j_1^{\prime},j_2^{\prime})\in
\mathcal{I}_{N-1}:|j_1^{\prime}|+|j_2^{\prime}|=N\},
\end{equation*}
which is the  `boundary of $\mathcal{I}_{N-1}$ excluding the $x$ and $y$ axes'.
Based on the assumption we know that on the set $\Omega_0 \cap \tilde{\Omega}_0$, there holds
\begin{equation*}
\sup\limits_{t\in[\eta/2,\eta]}
|\langle K_{t,\eta}\phi,\sigma_{j}^{m}\rangle|
=0,~~\forall j\in \mathcal{I}_{N-1}, m\in\{0,1\}.
\end{equation*}
Therefore, by Lemma \ref{lemmaB} and Lemma \ref{lemmaC},  we get on the set $\Omega_0 \cap \tilde{\Omega}_0$,
\begin{equation}\label{K.8}
\begin{aligned}
\langle\mathcal{M}_{0,T}\phi,\phi\rangle
=0
&\Rightarrow
\sup\limits_{j'\in \mathcal{B}_{N-1},m\in\{0,1\}}\langle K_{t,T}\phi,Y_{j'}^m(U)\rangle
=0\\
&\Rightarrow
\sup_{\substack{m,l\in\{0,1\} \\ j'\in\mathcal{B}_{N-1},k\in\mathcal{Z}}}
\sup\limits_{t\in[T/2,T]}|\langle
K_{t,T}\phi,[Z_{j'}^m(U),\sigma_k^l]\rangle|
=0.
\end{aligned}
\end{equation}

Since we have known that the result holds for $j\in \mathcal{I}_{N-1}$, therefore it is enough to establish \eqref{K.7} for any fixed $j\in \mathcal{I}_{N}\backslash\mathcal{I}_{N-1}$
and $m\in\{0,1\}$. We know that for each $j\in \mathcal{I}_{N}\backslash\mathcal{I}_{N-1}$, there exists $j'\in\mathcal{B}_{N-1}$ such that $k:=j-j'\in \mathcal{Z}\cup(-\mathcal{Z})$. Since $k$ is parallel to one of the axes and $j'$ is not, thus $k^{\bot}\cdot j'\neq0$ and $a(j',k)\neq0$, $b(j',k)\neq0$. Using \eqref{K.8} and Proposition \ref{proper5.2}, we infer that for each fixed $j\in \mathcal{I}_{N}\backslash\mathcal{I}_{N-1}$,
$m\in\{0,1\}$
\begin{equation}\label{K.14}
\begin{aligned}
&\langle\mathcal{M}_{0,\eta}\phi,\phi\rangle
=0
\Rightarrow\sup_{t\in[\eta/2,\eta]}|\langle
K_{t,\eta}\phi,\sigma_j^m\rangle|
=0,
\end{aligned}
\end{equation}
by combining the initial assumption we also can deduce that the above formula holds true for any fixed $j\in\mathcal{I}_{N}, m\in\{0,1\}$.
By Lemma \ref{lemmaB} and \eqref{K.14}, we deduce that the rest of \eqref{K.7} holds on the set $\Omega_0 \cap \tilde{\Omega}_0$, for any fixed $j\in\mathcal{I}_{N},m\in\{0,1\}$.
The proof is completed.
\end{proof}

\begin{lemma}\label{lemmaF} %Fix $N\geq 2$,
For all $\phi\in \mathcal{S}_{\alpha,N}$, $j\in \mathbb{Z}_+^2$, $m\in\{0,1\}$, and $\omega\in\Omega_0\cap\tilde{\Omega}_0$,
\begin{equation}\label{K.9}
\begin{aligned}
\langle\mathcal{M}_{0,\eta}\phi,\phi\rangle
=0
\Rightarrow
\begin{cases}
\sup\limits_{t\in[\eta/2,\eta]}
|\langle K_{t,\eta}\phi,\sigma_j^m\rangle|(\omega)
=0,
&\quad(\mathrm{a})\\
\sup\limits_{t\in[\eta/2,\eta]}|\langle K_{t,\eta}\phi,\psi_j^m\rangle|(\omega)
=0.
&\quad(\mathrm{b})
\end{cases}
\end{aligned}
\end{equation}
\end{lemma}
\begin{proof}

By Lemma \ref{lemmaE}, for each $j$ such that $j\in \mathbb{Z}_+^2$, \eqref{K.9}(a) follows on the set $\Omega_0 \cap \tilde{\Omega}_0$.

For $(b)$, we fix $|j|\leq N$, $j_1\neq0,m\in\{0,1\}$, from $\psi_j^m+J_{j,m}(U)=\frac{(-1)^m}{gj_1}Y_{j}^{m+1}(U)
,j_1\neq0$, it follows that
\begin{equation}\label{K.12}
\begin{aligned}
|\langle K_{t,\eta}\phi,Y_j^{m+1}(U)\rangle|
&=g|j_1||\langle K_{t,\eta}\phi,\psi_j^m+J_{j,m}(U)\rangle|.
%&\geq g|j_1||\langle K_{t,T}\phi,\psi_j^m+J_{j,m}^N(U)\rangle|
%-g|j_1||\langle K_{t,T}\phi,P_NJ_{j,m}(U)\rangle|.
\end{aligned}
\end{equation}
From Lemma \ref{5.25} and \eqref{K.01}, we can get that $\sup\limits_{t\in[\eta/2,\eta]}|\langle K_{t,\eta}\phi,\psi_j^m\rangle|=0$ holds on the set $\Omega_0$.

Next, fix $|j|\leq N$ with $j_1=0$ ($j$ on the y-axis) and set $j':=\vec{e}_1+j$ ($j'$ not on the y-axis). It is easy to check that $j',j'\pm\vec{e}_2,j'\pm\vec{e}_1$ belong to $\mathcal{I}_{2N}$ whenever $N\geq 2$, therefore by Lemma \ref{lemmaD}, we get that on the set $\Omega_0$
\begin{equation}\label{K.10}
\sum_{ \substack{k\in\mathcal{Z} \\ m,l\in\{0,1\}}}
\sup_{t\in[\eta/2,\eta]}
|\langle K_{t,\eta}\phi,[Z_{j^{\prime}}^m(U),Y_k^l(U)]\rangle|
=0.
\end{equation}
By Lemma \ref{5.25} and \eqref{K.10}, for $j\in\mathbb{Z}_+^2$, $m\in\{0,1\}$, one has
\begin{equation}\label{K.11}
\sup_{t\in[\eta/2,\eta]}
|\langle K_{t,\eta}\phi,
\psi_j^m+J_{j,m}(U)\rangle|
=0.
\end{equation}
Notice that \eqref{5.26} in Lemma \ref{5.25}, for $j_1=0$,
we already know that $U\mapsto H_{j,k}^{m,m'}(U)$ is affine and it is concentrated entirely in the $\theta$ component, and recall that the definition of $\psi_j^m$ in \eqref{psi}, they are not related to the $\theta$ component, we can deduce that \eqref{K.9}(b) holds on the set $\Omega_0$. In summary, \eqref{K.9} holds on the set $\Omega_0\cap\tilde{\Omega}_0$. This proof is finished.
\end{proof}

Next, we will prove the following stronger result than \eqref{3.4} for later use:
\begin{equation}\label{3.5}
\mathbb{P}\Big(\omega=(\mathrm w,\ell):\inf\limits_{\phi\in\mathcal{S}_{\alpha,N}}
\sum\limits_{j\in \mathcal{Z}_0,m\in\{0,1\}}\int_{\eta/2}^{\eta}\left|\langle K_{r,\eta}\phi,\sigma_j^m\rangle\right|^2
\mathrm d\ell_r=0\Big)=0.
\end{equation}

\begin{proof}[Proof of Proposition \ref{propo3.4}] We will take use of Lemma \ref{lemmaF} to prove \eqref{3.5}.
Set
\begin{equation}\label{3.6}
\mathcal{L}:=
\Big\{\omega:\inf\limits_{\phi\in\mathcal{S}_{\alpha,N}}
\sum\limits_{j\in \mathcal{Z}_0,m\in\{0,1\}}\int_{\eta/2}^{\eta}
\langle K_{r,\eta}\phi,\sigma_j^m\rangle^2
\mathrm d\ell_r=0\Big\}
\cap\Omega_{0}\cap\tilde{\Omega}_{0}.
\end{equation}
Assume that $\mathcal{L}\neq\emptyset$ and $\omega=(\mathrm w,\ell)$ belongs to the event $\mathcal{L}$. Then for some $\phi$ with
\begin{equation}\label{3.7}
\|P_{N}\phi\|\geq\alpha,
\end{equation}
on the set $\Omega_{0}\cap\tilde{\Omega}_{0}$, there hold
\begin{equation*}
\begin{aligned}
\int_{\eta/2}^{\eta}\langle K_{r,\eta}\phi,\sigma_j^m\rangle^2
\mathrm{d}\ell_{r}=0
~\text{and}~
\int_{\eta/2}^{\eta}\langle K_{r,\eta}\phi,\psi_j^m
\rangle^2 \mathrm{d}\ell_{r}=0,\forall~j\in \mathcal{Z},~m\in\{0,1\}.
\end{aligned}
\end{equation*}
In the following, by the property of $\ell\in\Omega_0\cap\tilde{\Omega}_0$ stated above and the continuity of $\langle K_{t,\eta}\phi,\sigma_{j}^{m}\rangle$ and $\langle K_{t,\eta}\phi,\psi_j^m\rangle$ with respect to $t$, it follows that
\begin{equation*}
\begin{aligned}
\begin{cases}
&\sup\limits_{t\in[\eta/2,\eta]}|\langle K_{t,\eta}\phi,\sigma_{j}^{m}\rangle|=0,\\
&\sup\limits_{t\in[\eta/2,\eta]}
|\langle K_{t,\eta}\phi,\psi_j^m
\rangle|=0,~\forall j\in\mathcal{Z},~m\in\{0,1\}.
\end{cases}
\end{aligned}
\end{equation*}

By Lemma \ref{lemmaF}, we know that there exists a $\phi\in \mathcal{S}_{\alpha,N}$, for all $j\in \mathbb{Z}_+^2$ and $m\in\{0,1\}$, one has
\begin{equation*}
\begin{cases}
\sup\limits_{t\in[\eta/2,\eta]}|\langle K_{t,\eta}\phi,\sigma_j^m\rangle|
=0,\\
\sup\limits_{t\in[\eta/2,\eta]}|\langle K_{t,\eta}\phi,\psi_j^m\rangle|
=0.
\end{cases}
\end{equation*}
Let $t\rightarrow \eta$, we can get $\phi=0$, which contradicts \eqref{3.7}. Therefore, $\mathcal{L}=\emptyset$.
\end{proof}

Note that Proposition \ref{propo3.4} is not sufficient for the proof of Proposition \ref{propo1.4}(e-property), we need to make a necessary supplement to the proposition. For $\alpha\in(0,1]$, $U_0\in H$, $N\in\mathbb{N}$, $\mathfrak{R}>0$ and $\varepsilon>0$, let
\begin{equation*}
X^{U_0,\alpha,N}=\inf_{\phi\in\mathcal{S}_{\alpha,N}}
\langle\mathcal{M}_{0,\eta}\phi,\phi\rangle.
\end{equation*}
We denote
\begin{equation*}
r(\varepsilon,\alpha,\mathfrak{R},N)
=\sup_{\|U_0\|\leq\Re}\mathbb{P}(X^{U_0,\alpha,N}<\varepsilon).
\end{equation*}
Recall that $\mathcal{M}_{0,t}$ is the Malliavin matrix of $U_t$, where $U_t$ is the solution of equation \eqref{2.1} at time $t$ with initial value $U_0$. Therefore, $\mathcal{M}_{0,\eta}$ also depends on $U_0$.
Based on \eqref{3.5} and the dissipative property of Boussinesq system, our aim is to verify that
\begin{equation}\label{r}
r(\varepsilon,\alpha,\mathfrak{R},N)\rightarrow0
~~\text{as}~~ \varepsilon\rightarrow 0.
\end{equation}

We use the method of contradiction. We suppose that \eqref{r} is not true, then there exist the sequences $\{U_0^{(k)}\}\subseteq B_H(\mathfrak{R})$, $\{\varepsilon_k\}\subseteq(0,1)$ and a positive number $\delta_0$ such that
\begin{equation}\label{B.1}
\lim\limits_{k\to\infty}\mathbb{P}(X^{U_0^{(k)},\alpha,N}
<\varepsilon_k)\ge\delta_0>0
~~\text{and}~\lim\limits_{k\to\infty}\varepsilon_k=0.
\end{equation}
Thus we need to search out something to contradict \eqref{B.1}, then \eqref{r} can be proved.

Since $H$ is a Hilbert space, there exists a subsequence $\{U_0^{(n_k)},k\geq 1\}$ of $\{U_0^{(k)},k\geq 1\}$ such that $U_0^{(n_k)}$ converges weakly to some $U_0^{(0)}\in H$. For convenience, we denote this subsequence by $\{U_0^{(k)},k\geq 1\}$. Let  $U_t^{(k)}$ denote the solution of equation \eqref{2.1} with $U_t|_{t=0}=U_0^{(k)}(k\geq 0)$
\begin{equation}\label{B.2}
\partial_t J_{s,t}\xi+AJ_{s,t}\xi+\nabla B(U_t)J_{s,t}\xi=G(J_{s,t}\xi),\quad J_{s,s}\xi=\xi,
\end{equation}
when $U_t$ is replaced by $U_t^{(k)}$, where $\nabla B(U_t)J_{s,t}\xi:=B(U_t,J_{s,t}\xi)+B(J_{s,t}\xi,U_t)$. With regard to $U_t^{(k)}$, we use $J_{s,t}^{(k)}\xi$ to denote the solution of \eqref{B.2}, and $K_{s,t}^{(k)}\xi$ is the adjoint of $J_{s,t}^{(k)}\xi$.

Recall that for any $M\in\mathbb{N},~H_M=\{\sigma_k^m,\psi_k^m:|k|\leq M,m\in\{0,1\}\}$, $P_M$ denotes the orthogonal projections from $H$ onto $H_M$ and $Q_M u:=u-P_M u,~\forall u\in H$.

Next, we will give some estimates for $\|Q_{M}U_{t}^{(k)}\|$, $\|P_{M}U_{t}^{(k)}-P_{M}U_{t}^{(0)}\|$ and $\|J_{s,t}^{(k)}\xi-J_{s,t}^{(0)}\xi\|,s,t\in(0,T]$. Finally, we combine the arguments in \cite[Appendix B]{Peng-2024} to achieve \eqref{r}.

\begin{lemma}\label{lemma 4.4} For any $t\geq0,k\in\mathbb{N}$ and $M>\max\{|k|:k\in\mathcal{Z}\}$, there exist two positive constants $C$ and $C_{\mathfrak{R}}$, they depend on $\nu=\min\{\nu_1,\nu_2\},\nu_S,d,g$ and  $\mathfrak{R},\nu,\nu_S,d,g$, respectively, such that
 \begin{equation}\label{4.5}
\begin{aligned}
\|Q_{M}U_{t}^{(k)}\|^{2}\leq& e^{-\nu M^2 t}\|Q_{M}U_{0}^{(k)}\|^{2}
+C\|\xi\|^2e^{\int_s^te^{-\nu M^2r}[C_0(\|U_r^{(k)}\|_1^{4/3}+1)-\nu M^2]\mathrm{d}r} \cdot \sup\limits_{s\in[0,t]}\|U_{s}^{(k)}\|^2\\
&+C\|\xi\|^2e^{\int_s^te^{-\nu M^2r}[C_0(\|U_r^{(k)}\|_1^{4/3}+1)-\nu M^2]\mathrm{d}r},
\end{aligned}
\end{equation}
and
\begin{equation}\label{4.6}
\begin{aligned}
&\|P_{M}U_{t}^{(k)}-P_{M}U_{t}^{(0)}\|^{2}
\leq\|P_{M}U_{0}^{(k)}-P_{M}U_{0}^{(0)}\|^{2}
e^{C\int_{0}^{t}\|U_{s}^{(k)}\|_{1}^{2}\mathrm{d}s} \\
&~~~+C_{\mathfrak{R}}e^{C\int_{0}^{t}\|U_{s}^{(k)}\|_{1}^{2}\mathrm{d}s}
\times\left(\int_0^t\|U_{s}^{(k)}\|_1^{2}\mathrm{d}s
+\sup_{s\in[0,t]}(\|U_s^{(0)}\|^2+1)\right)\\
&~~~\times \left[e^{-\nu M^2 t}\|Q_{M}U_{0}^{(k)}\|^{2}
+C\|\xi\|^2e^{\int_s^te^{-\nu M^2r}[C_0(\|U_r^{(k)}\|_1^{4/3}+1)-\nu M^2]\mathrm{d}r} \cdot \sup\limits_{s\in[0,t]}\|U_{s}^{(k)}\|^2\right.\\
&~~~~~~\left.+C\|\xi\|^2e^{\int_s^te^{-\nu M^2r}[C_0(\|U_r^{(k)}\|_1^{4/3}+1)-\nu M^2]\mathrm{d}r}\right].
\end{aligned}\end{equation}
\end{lemma}
\begin{proof}
First, by \eqref{2.1} we can obtain that
\begin{equation*}
\begin{aligned}
\zeta^*\langle B_1(w_{t}^{(k)},w_{t}^{(k)}),Q_{M}w_{t}^{(k)}\rangle & \leq C\zeta^*\|Q_{M}w_{t}^{(k)}\|_{1}\|w_{t}^{(k)}\|_{1/2}\|w_{t}^{(k)}\| \\
&\leq\frac{\nu_1\zeta^*}{2}\|Q_{M}w_{t}^{(k)}\|_{1}^{2}
+C\|w_{t}^{(k)}\|_{1}\|w_{t}^{(k)}\|^{3},
\end{aligned}
\end{equation*}
and
\begin{equation*}
\begin{aligned}
\langle B_2(w_{t}^{(k)},\theta_{t}^{(k)}),Q_{M}\theta_{t}^{(k)}\rangle & \leq C\|Q_{M}\theta_{t}^{(k)}\|_{1}\|w_{t}^{(k)}\|\|\theta_{t}^{(k)}\|_1 \\
&\leq\nu_2\|Q_{M}\theta_{t}^{(k)}\|_{1}^{2}
+C\|w_{t}^{(k)}\|^2\|\theta_{t}^{(k)}\|_1^{2}.
\end{aligned}
\end{equation*}
Moreover, we have
\begin{equation*}
\begin{aligned}
\zeta^*\langle g\partial_x \theta_{t}^{(k)},Q_{M}w_{t}^{(k)}\rangle &
\leq \zeta^*g\|\theta_{t}^{(k)}\|_{1}\|Q_{M}w_{t}^{(k)}\| \\
&\leq\frac{\nu_1\zeta^*}{2}\|Q_{M}w_{t}^{(k)}\|_1^{2}
+\frac{\nu_2}{2}\|\theta_{t}^{(k)}\|_1^{2}.
\end{aligned}
\end{equation*}
We take the inner product of \eqref{2.1} with $Q_{M}w_{t}^{(k)}$ and $Q_{M}\theta_{t}^{(k)}$, respectively, there holds
\begin{equation}\label{4.10}
\begin{aligned}
&~~\mathrm{d}(\zeta^*\|Q_{M}w_{t}^{(k)}\|^{2}+\|Q_{M}\theta_{t}^{(k)}\|^{2})\\
&=-2\nu_1\zeta^*\|Q_{M}w_{t}^{(k)}\|_{1}^{2}\mathrm{d}t
-2\nu_2\|Q_{M}\theta_{t}^{(k)}\|_{1}^{2}\mathrm{d}t
+2\zeta^*\langle B_1(w_{t}^{(k)},w_{t}^{(k)}),Q_{M}w_{t}^{(k)}\rangle\mathrm{d}t
\\
&~~~+2\langle B_2(w_{t}^{(k)},\theta_{t}^{(k)}),Q_{M}\theta_{t}^{(k)}\rangle\mathrm{d}t +2\zeta^*\langle g\partial_x \theta_{t}^{(k)},Q_{M}w_{t}^{(k)}\rangle\mathrm{d}t\\
&\leq(-\nu_1\zeta^* M^2\|Q_{M}w_{t}^{(k)}\|^{2}-\nu_2 M^2 \|Q_{M}\theta_{t}^{(k)}\|^{2})\mathrm{d}t
\\
&~~~+\left(C\|w_{t}^{(k)}\|_{1}^2\|w_{t}^{(k)}\|^{2}
+C\|w_{t}^{(k)}\|^2\|\theta_{t}^{(k)}\|_1^{2}
+C\|\theta_{t}^{(k)}\|_1^{2}\right)\mathrm{d}t.
\end{aligned}
\end{equation}

Applying Gronwall formula to \eqref{4.10} yields
\begin{align*}
\|Q_{M}U_{t}^{(k)}\|^{2}&
\leq e^{-\nu M^2 t}\|Q_{M}U_{0}^{(k)}\|^{2}\\
&~~~+C\int_0^t e^{-\nu M^2 (t-s)} \left((\zeta^*\|w_{s}^{(k)}\|_{1}^2
+\|\theta_{s}^{(k)}\|_1^{2})\zeta^*\|w_{s}^{(k)}\|^2
+\|\theta_{s}^{(k)}\|_1^{2}\right)\mathrm{d}s\\
&\leq e^{-\nu M^2 t}\|Q_{M}U_{0}^{(k)}\|^{2}
+C\int_0^t e^{-\nu M^2 (t-s)} \|U_{s}^{(k)}\|_{1}^2(\zeta^*\|w_{s}^{(k)}\|^2+1)\mathrm{d}s\\
&\leq e^{-\nu M^2 t}\|Q_{M}U_{0}^{(k)}\|^{2}
+C\int_0^t e^{-\nu M^2(t-s)}\|U_{s}^{(k)}\|_{1}^2\mathrm{d}s \cdot \sup\limits_{s\in[0,t]}\|U_{s}^{(k)}\|^2\\
&~~~+C\int_0^t e^{-\nu M^2 (t-s)} \|U_{s}^{(k)}\|_{1}^2\mathrm{d}s\\
&\leq e^{-\nu M^2 t}\|Q_{M}U_{0}^{(k)}\|^{2}
+C\|\xi\|^2e^{\int_s^te^{-\nu M^2r}[C_0(\|U_r^{(k)}\|_1^{4/3}+1)-\nu M^2]\mathrm{d}r} \cdot \sup\limits_{s\in[0,t]}\|U_{s}^{(k)}\|^2\\
&~~~+C\|\xi\|^2e^{\int_s^te^{-\nu M^2r}[C_0(\|U_r^{(k)}\|_1^{4/3}+1)-\nu M^2]\mathrm{d}r},
\end{align*}
where $C_0$ is the positive constant from Lemma \ref{lemma2.4}, and we have used \eqref{4.12} with $m=\nu M^2$ in the fourth inequality.

Next, we turn to prove \eqref{4.6}. One easily gets that
\begin{align}\label{4.11}
&\frac{\mathrm{d}}{\mathrm{d}t}(\zeta^*\|P_{M}w_{t}^{(k)}-P_{M}w_{t}^{(0)}\|^{2}
+\|P_{M}\theta_{t}^{(k)}-P_{M}\theta_{t}^{(0)}\|^{2}) \notag\\
=&(-2\zeta^*\nu_1\|P_{M}w_{t}^{(k)}-P_{M}w_{t}^{(0)}\|_{1}^{2}
-2\nu_2\|P_{M}\theta_{t}^{(k)}-P_{M}\theta_{t}^{(0)}\|_{1}^{2}) \notag\\
&-2\langle P_{M}\theta_{t}^{(k)}-P_{M}\theta_{t}^{(0)},B_2( w_{t}^{(k)},\theta_{t}^{(k)})-B_2(w_{t}^{(0)},\theta_{t}^{(0)})\rangle
\\
&-2\zeta^*\langle P_{M}w_{t}^{(k)}-P_{M}w_{t}^{(0)},B_1(w_{t}^{(k)},w_{t}^{(k)})
-B_1(w_{t}^{(0)},w_{t}^{(0)})\rangle
\notag\\
&+2\zeta^*\langle P_{M}w_{t}^{(k)}-P_{M}w_{t}^{(0)},
g\partial_x(\theta_{t}^{(k)}-\theta_{t}^{(0)})\rangle \notag\\
:=& -2\zeta^*\nu_1\|P_{M}w_{t}^{(k)}-P_{M}w_{t}^{(0)}\|_{1}^{2}
-2\nu_2\|P_{M}\theta_{t}^{(k)}-P_{M}\theta_{t}^{(0)}\|_{1}^{2}+I_1+I_2+I_3.  \notag
\end{align}
For $I_{1}$, we have
\begin{equation*}
\begin{aligned}
I_1=&-2\langle P_{M}\theta_{t}^{(k)}-P_{M}\theta_{t}^{(0)},B( w_{t}^{(k)}-w_{t}^{(0)},\theta_{t}^{(k)})\rangle \\
&-2\langle P_{M}\theta_{t}^{(k)}-P_{M}\theta_{t}^{(0)},B( w_{t}^{(0)},\theta_{t}^{(k)}-\theta_{t}^{(0)})\rangle
:= I_{11}+I_{12}.
\end{aligned}
\end{equation*}
By \eqref{2.2}, we get
\begin{equation*}
\begin{aligned}
I_{11}\leq& C\|P_{M}\theta_{t}^{(k)}-P_{M}\theta_{t}^{(0)}\|_1
\|P_{M}w_{t}^{(k)}-P_{M}w_{t}^{(0)}\|\|\theta_{t}^{(k)}\|_1\\
&+C\|P_{M}\theta_{t}^{(k)}-P_{M}\theta_{t}^{(0)}\|_1
\|Q_{M}w_{t}^{(k)}-Q_{M}w_{t}^{(0)}\|\|\theta_{t}^{(k)}\|_1
\\
\leq& \frac{\nu_2}{3}\|P_{M}\theta_{t}^{(k)}-P_{M}\theta_{t}^{(0)}\|_1^2\\
&+C(\|P_{M}w_{t}^{(k)}-P_{M}w_{t}^{(0)}\|^2
+\|Q_{M}w_{t}^{(k)}-Q_{M}w_{t}^{(0)}\|^2)\|\theta_{t}^{(k)}\|_1^2,
\end{aligned}
\end{equation*}
similarly, we have
\begin{equation*}
\begin{aligned}
I_{12}\leq& \frac{\nu_2}{3}\|P_{M}\theta_{t}^{(k)}-P_{M}\theta_{t}^{(0)}\|_1^2
+C\zeta^*\|Q_{M}\theta_{t}^{(k)}-Q_{M}\theta_{t}^{(0)}\|_1^2\|w_t^{(0)}\|^2.
\end{aligned}
\end{equation*}
Similar to $I_1$, we have
\begin{equation*}
\begin{aligned}
I_2\leq&\nu_1\zeta^*\|P_{M}w_{t}^{(k)}-P_{M}w_{t}^{(0)}\|_1^2
+C\zeta^*(\|P_{M}w_{t}^{(k)}-P_{M}w_{t}^{(0)}\|^2
+\|Q_{M}w_{t}^{(k)}-Q_{M}w_{t}^{(0)}\|^2)\|w_{t}^{(k)}\|_1^2\\
&+C\zeta^*\|Q_{M}w_{t}^{(k)}-Q_{M}w_{t}^{(0)}\|_1^2\|w_t^{(0)}\|^2,
\end{aligned}
\end{equation*}
and
\begin{equation*}
\begin{aligned}
I_3\leq& C\| P_{M}w_{t}^{(k)}-P_{M}w_{t}^{(0)}\|\|\theta_{t}^{(k)}-\theta_{t}^{(0)}\|_1\\
\leq& C\| P_{M}w_{t}^{(k)}-P_{M}w_{t}^{(0)}\|^2+\frac{\nu_2}{3}
\|P_M(\theta_{t}^{(k)}-\theta_{t}^{(0)})\|_1^2
+\frac{\nu_2}{3}
\|Q_M(\theta_{t}^{(k)}-\theta_{t}^{(0)})\|_1^2.
\end{aligned}
\end{equation*}
Combining the estimates of $I_{j},(j=1,2,3)$ with \eqref{4.11}, \eqref{4.5}, and taking into account the fact that $\|Q_Mw_0^{(k)}\|+\|Q_M \theta_0^{(k)}\|\leq\mathfrak{R}$, we arrive at
\begin{align*}
&\zeta^*\|P_{M}w_{t}^{(k)}-P_{M}w_{t}^{(0)}\|^{2}+\|P_{M}\theta_{t}^{(k)}-P_{M}\theta_{t}^{(0)}\|^{2}\\
&\leq\left(\zeta^*\|P_{M}w_{0}^{(k)}-P_{M}w_{0}^{(0)}\|^{2}
+\|P_{M}\theta_{0}^{(k)}-P_{M}\theta_{0}^{(0)}\|^{2}\right)
e^{C\int_{0}^{t}(\zeta^*\|w_{s}^{(k)}\|_{1}^{2}+\|\theta_{s}^{(k)}\|_{1}^{2})\mathrm{d}s} \\
&~~~+Ce^{C\int_{0}^{t}(\zeta^*\|w_{s}^{(k)}\|_{1}^{2}+\|\theta_{s}^{(k)}\|_{1}^{2})\mathrm{d}s}
\left(\int_{0}^{t}\|Q_{M}w_{s}^{(k)}-Q_{M}w_{s}^{(0)}\|^{2}
(\zeta^*\|w_{s}^{(0)}\|^{2}+\zeta^*\|w_{s}^{(k)}\|_1^{2}+\|\theta_{s}^{(k)}\|_1^{2})\mathrm{d}s\right.\\
&~~~~~~\left.+\int_{0}^{t}\|Q_{M}\theta_{s}^{(k)}-Q_{M}\theta_{s}^{(0)}\|^{2}
(\zeta^*\|w_{s}^{(0)}\|^{2}+\nu_2)\mathrm{d}s\right) \\
&\leq\|P_{M}U_{0}^{(k)}-P_{M}U_{0}^{(0)}\|^{2}
e^{C\int_{0}^{t}\|U_{s}^{(k)}\|_{1}^{2}\mathrm{d}s} \\
&~~~+Ce^{C\int_{0}^{t}\|U_{s}^{(k)}\|_{1}^{2}\mathrm{d}s}
\left[\left(\int_0^t\|Q_Mw_s^{(k)}-Q_Mw_s^{(0)}\|^2\mathrm{d}s\right)
\left(\sup_{s\in[0,t]}\|U_{s}^{(0)}\|^{2}+\int_0^t\|U_{s}^{(k)}\|_1^{2}\mathrm{d}s\right) \right. \\
&~~~~~~\left.+\left(\int_0^t\|Q_M\theta_s^{(k)}-Q_M\theta_s^{(0)}\|^2\mathrm{d}s\right)
\sup_{s\in[0,t]}(\|U_s^{(0)}\|^2+1)\right]
 \\
&\leq\|P_{M}U_{0}^{(k)}-P_{M}U_{0}^{(0)}\|^{2}
e^{C\int_{0}^{t}\|U_{s}^{(k)}\|_{1}^{2}\mathrm{d}s}
+C_{\mathfrak{R}}e^{C\int_{0}^{t}\|U_{s}^{(k)}\|_{1}^{2}\mathrm{d}s}
\times\left(\int_0^t\|U_{s}^{(k)}\|_1^{2}\mathrm{d}s
+\sup_{s\in[0,t]}(\|U_s^{(0)}\|^2+1)\right)\\
&~~~\times \left[e^{-\nu M^2 t}\|Q_{M}U_{0}^{(k)}\|^{2}
+C\|\xi\|^2e^{\int_s^te^{-\nu M^2r}[C_0(\|U_r^{(k)}\|_1^{4/3}+1)-\nu M^2]\mathrm{d}r} \cdot \sup\limits_{s\in[0,t]}\|U_{s}^{(k)}\|^2\right.\\
&~~~~~~\left.+C\|\xi\|^2e^{\int_s^te^{-\nu M^2r}[C_0(\|U_r^{(k)}\|_1^{4/3}+1)-\nu M^2]\mathrm{d}r}\right],
\end{align*}
which implies \eqref{4.6}. This completes the proof.
\end{proof}

\begin{lemma}\label{lemma4.4} For any $0\leq s\leq t, k\in\mathbb{N}$ and $\xi\in H$ with $\|\xi\|=1$, one has
\begin{equation*}
\begin{aligned}
\|J_{s,t}^{(k)}\xi-J_{s,t}^{(0)}\xi\|^2
\leq C&\sup_{r\in[s,t]}\|U_r^{(k)}-U_r^{(0)}\|^2\cdot e^{C\int_s^t(\|U_r^{(k)}\|_1^{4/3}+1)
\mathrm{d}r},
\end{aligned}
\end{equation*}
where $C$ is a positive constant depending on $\nu=\min\{\nu_1,\nu_2\},\nu_S,d,g$.
\end{lemma}
\begin{proof}
From the system \eqref{B.2}, we can deduce that
\begin{equation}\label{B.7}
\begin{aligned}
&\frac{\mathrm{d}}{\mathrm{d}t}\|J_{s,t}^{(k)}\xi-J_{s,t}^{(0)}\xi\|^{2}
\leq-2\nu\|J_{s,t}^{(k)}\xi-J_{s,t}^{(0)}\xi\|_{1}^{2} \\
&~~~-2\langle J_{s,t}^{(k)}\xi-J_{s,t}^{(0)}\xi,
B(U_{t}^{(k)},J_{s,t}^{(k)}\xi)
-B(U_{t}^{(0)},J_{s,t}^{(0)}\xi)\rangle  \\
&~~~-2\langle J_{s,t}^{(k)}\xi-J_{s,t}^{(0)}\xi,
B(J_{s,t}^{(k)}\xi,U_{t}^{(k)})
-B(J_{s,t}^{(0)}\xi,U_{t}^{(0)})\rangle \\
&~~~+2\langle G(J_{s,t}^{(k)}\xi-J_{s,t}^{(0)}\xi),
J_{s,t}^{(k)}\xi-J_{s,t}^{(0)}\xi\rangle
\\
&:=-2\nu\|J_{s,t}^{(k)}\xi-J_{s,t}^{(0)}\xi\|_{1}^{2} +2(P_1+P_2+P_3).
\end{aligned}
\end{equation}
By using \eqref{2.1}, we can get
\begin{align*}
P_{1}& \leq|\langle J_{s,t}^{(k)}\xi-J_{s,t}^{(0)}\xi,
B(U_{t}^{(k)}-U_{t}^{(0)},J_{s,t}^{(k)}\xi)\rangle|  \\
&\leq C\|U_{t}^{(k)}-U_{t}^{(0)}\|
\|J_{s,t}^{(k)}\xi-J_{s,t}^{(0)}\xi\|_{1}
\|J_{s,t}^{(k)}\xi\|_{1/2} \\
&\leq C\|U_{t}^{(k)}-U_{t}^{(0)}\|^{2}
\|J_{s,t}^{(k)}\xi\|_{1/2}^{2}
+\frac{\nu}{6}\|J_{s,t}^{(k)}\xi-J_{s,t}^{(0)}\xi\|_{1}^{2},
\end{align*}
and
\begin{align*}
 P_{2} &\leq|\langle J_{s,t}^{(k)}\xi-J_{s,t}^{(0)}\xi,
B(J_{s,t}^{(k)}\xi-J_{s,t}^{(0)}\xi,
U_{t}^{(k)})+B(J_{s,t}^{(0)}\xi,U_{t}^{(k)}-U_{t}^{(0)})
\rangle|  \\
&\leq C\|U_{t}^{(k)}\|_{1}
\|J_{s,t}^{(k)}\xi-J_{s,t}^{(0)}\xi\|^{3/2}
\|J_{s,t}^{(k)}\xi-J_{s,t}^{(0)}\xi\|_{1}^{1/2} \\
&+C\|U_t^{(k)}-U_t^{(0)}\|
\|J_{s,t}^{(k)}\xi-J_{s,t}^{(0)}\xi\|_1
\|J_{s,t}^{(0)}\xi\|_{1/2} \\
&\leq\frac{\nu}{6}\|J_{s,t}^{(k)}\xi-J_{s,t}^{(0)}\xi\|_{1}^{2}
+C\|J_{s,t}^{(k)}\xi-J_{s,t}^{(0)}\xi\|^{2}\|U_{t}^{(k)}\|_{1}^{4/3} \\
&+C\|U_t^{(k)}-U_t^{(0)}\|^2\|J_{s,t}^{(0)}\xi\|_{1/2}^2.
\end{align*}
By Poincaré inequality, one has
\begin{equation*}
\begin{aligned}
P_3\leq g\|J_{s,t}^{(k)}\xi-J_{s,t}^{(0)}\xi\|_1
\|J_{s,t}^{(k)}\xi-J_{s,t}^{(0)}\xi\|
\leq \frac{\nu}{6}\|J_{s,t}^{(k)}\xi-J_{s,t}^{(0)}\xi\|_1^2
+C\|J_{s,t}^{(k)}\xi-J_{s,t}^{(0)}\xi\|^2.
\end{aligned}
\end{equation*}
Combining \eqref{B.7} with the estimates of $P_1,P_2,P_3$, we obtain
\begin{equation*}
\begin{aligned}
\|J_{s,t}^{(k)}\xi-J_{s,t}^{(0)}\xi\|^{2}
\leq C \sup_{r\in[s,t]}\|U_{r}^{(k)}-U_{r}^{(0)}\|^{2}\cdot e^{C\int_{s}^{t}(\|U_{r}^{(k)}\|_{1}^{4/3}+1)\mathrm{d}r} \int_{s}^{t}\left[\|J_{s,r}^{(k)}\xi\|_{1/2}^{2}
+\|J_{s,r}^{(0)}\xi\|_{1/2}^{2}\right]\mathrm{d}r.
\end{aligned}
\end{equation*}
Note that the fact that $\|u\|^2_{1/2}\leq\|u\|\|u\|_1$, then using Lemma \ref{lemma2.4} and H\"{o}lder inequality, the desired estimate now
follows.
\end{proof}

With the help of Lemma \ref{lemma 4.4} and Lemma \ref{lemma4.4}, then combining with the arguments in \cite[Appendix B]{Peng-2024}, we can get the following result.

\begin{proposition}\label{proper3.5}
For $\alpha\in(0,1],\mathfrak{R}>0$ and $N\in \mathbb{N}$, we have
\begin{equation*}
\lim_{\varepsilon\rightarrow 0}r(\varepsilon,\alpha,\mathfrak{R},N)
=0.
\end{equation*}
\end{proposition}

\section{Weak irreducibility and ergodicity}

The ergodicity of invariant measures can usually be obtained by proving the irreducibility and e-property of Markov process, or strong Feller property, or asymptotic strong Feller property, refer to \cite{prato-1996,MH-2006,Kapica-2012,Komorowski-2010}. In this section, we will establish e-property and irreducibility to obtain the ergodicity for the system \eqref{2.1}.

Let us first do some preparatory work. Assume that $\kappa_0=\kappa_0(\nu,\{\alpha_k^m\}_{k\in\mathcal{Z},m\in\{0,1\}},\nu_{S},d)$ is the constant adapted to Lemma \ref{moment}. Recall that for any $\kappa\in(0,\kappa_0]$, the stopping times $\eta_k$ are defined in \eqref{2.7} and \eqref{2.8}. For any $\kappa\in(0,\kappa_0]$ and $n\in \mathbb{N}$, we define the following random variables on $\mathbb{S}$:
\begin{equation}
X_{n}=\int_{\eta_{n}}^{\eta_{n+1}}e^{2\nu(\eta_{n+1}-s)-16\mathfrak{B}_{0}
\kappa(\ell_{\eta_{n+1}}-\ell_{s})}\mathrm{d}s,\quad Y_{n}=\ell_{\eta_{n+1}}-\ell_{\eta_{n}}.
\end{equation}
By the strong law of large numbers, Lemma \ref{moment} and the definitions of $\eta_k$, we have
\begin{equation}
\lim_{n\to\infty}\frac{\sum_{i=0}^{n-1}X_{i}}{n}<\infty,
\quad\mathrm{a.s.},
\end{equation}
and
\begin{equation}
\lim_{n\to\infty}\frac{\sum_{i=0}^{n-1}Y_{i}}{n}
\leq\frac{\nu}{8\mathfrak{B}_{0}\kappa}\lim_{n\to\infty}
\frac{\sum_{i=0}^{n-1}(\eta_{i+1}-\eta_{i})}{n}<\infty,\quad\mathrm{a.s.}.
\end{equation}
Therefore, with probability one, we have
\begin{equation}\label{4.3}
\Theta:=\sup_{n\geq1}\frac{\sum_{i=0}^{n-1}X_i}n+\sup_{n\geq1}
\frac{\sum_{i=0}^{n-1}Y_i}n<\infty.
\end{equation}

For any $\Upsilon,M>0$, $\Phi\in C_b^1(H)$ and $U_0,U_0'\in B_H(\Upsilon):=\{U\in H,\|U\|\leq \Upsilon\}$, it holds
\begin{equation}\label{4.1}
\begin{aligned}
&|P_{t}\Phi(U_{0})-P_{t}\Phi(U_{0}^{\prime})|
=|\mathbb{E}[\Phi(U_{t}(U_{0}))]-\mathbb{E}[\Phi(U_{t}(U_{0}'))]| \\
&\leq|\mathbb{E}[\Phi(U_{t}(U_{0}))I_{\{\Theta\leq M\}}]-\mathbb{E}[\Phi(U_{t}(U_{0}^{\prime}))I_{\{\Theta\leq M\}}]|+2\|\Phi\|_{\infty}\mathbb{P}(\Theta\geq M) \\
&:=\mathcal{Q}_1+\mathcal{Q}_2.
\end{aligned}
\end{equation}

Denote $P_t^M \Phi(U_0)=\mathbb{E}\Phi(U_{t}(U_{0}))I_{\{\Theta\leq M\}}$. We infer that for any process $v\in L^{2}(\Omega\times[0,\infty);\mathbb{R}^{d})
=L^{2}(\mathbb{W}\times\mathbb{S}\times[0,\infty);\mathbb{R}^{d})$,
\begin{equation}\label{5.1}
\begin{aligned}
&\left|\nabla_{\xi}P_{t}^{M}\Phi(U_{0})\right|
=\left|\mathbb{E}[\nabla\Phi(U_{t})J_{0,t}\xi I_{\{\Theta\leq M\}}]\right| \\
&=\left|\mathbb{E}[\nabla\Phi(U_{t})(\mathcal{A}_{0,t}v+J_{0,t}\xi -\mathcal{A}_{0,t}v)I_{\{\Theta\leq M\}}]\right| \\
&=\left|\mathbb{E}[\nabla\Phi(U_{t})\mathcal{D}^v U_t I_{\{\Theta\leq M\}}]
+\mathbb{E}[\nabla\Phi(U_{t})\rho_t I_{\{\Theta\leq M\}} ]\right|\\
&=\mathbb{E}\left[\Phi(U_{t})\int_{0}^{\ell_t}v(s)\cdot \mathrm{d}W(s) I_{\{\Theta\leq M\}}\right]
+\mathbb{E}\left[\nabla\Phi(U_{t})\rho_t I_{\{\Theta\leq M\}}\right] \:=\mathcal{Q}_{11}+\mathcal{Q}_{12}.
\end{aligned}
\end{equation}
From the above result, we note that,
for any fixed $\ell\in \mathbb{S}$, the process $v=v^{\ell}$ in the above will be chosen such that $v^\ell\in L^{2}(\mathbb{W}\times[0,\ell_t];\mathbb{R}^d)$ and that $v=v^{\ell}$ is Skorokhod integrable with respect to the
Brownian motion $W$. Moreover, for any fixed $\ell\in \mathbb{S}$, the integral $\int_{0}^{\ell_t}v(s) \mathrm{d}W(s) $ is interpreted as the Skohorod integral. Altogether, through some gradient estimates of $\nabla_{\xi}P_{t}^{M}\Phi(U_{0})$, we can obtain the estimate of $\mathcal{Q}_1$.

We need three steps to achieve this goal. Firstly, we need to choose suitable direction $v$, that is, \textbf{the choice of $v$}.

We now build the control of $v$ and derive the associated $\rho$ in \eqref{2.12}, in this section, we always assume that $\|\xi\|=1$. In \cite{Foldes-2015}, the authors chose the control $v$ restricted to the time interval $[s,t]$ and let $\rho:=\rho(n)$. Unlike them, we utilize the stopping time $\eta_k$ defined in \eqref{2.7}-\eqref{2.8}. For any $\ell\in \mathbb{S}$ and $\kappa>0$, we define the  perturbation $v$ to be 0 on all intervals of the type $[\ell_{\eta_{n+1}},\ell_{\eta_{n+2}}],n\in 2\mathbb{N}$, and by several $v_{\eta_{n},\eta_{n+1}}\in L^2([\ell_{\eta_{n}},\ell_{\eta_{n+1}}],H),n\in 2\mathbb{N}$ on the remaining intervals. For fixed $\ell\in \mathbb{S}$ and non-negative integer $n\in 2\mathbb{N}$, define the infinitesimal variation:
\begin{equation}\label{4.7}
\begin{aligned}
&v_{\eta_{n},\eta_{n+1}}(r)=\mathcal{A}_{\eta_{n},\eta_{n+1}}^{*}
(\mathcal{M}_{\eta_{n},\eta_{n+1}}
+\beta\mathbb{I})^{-1}J_{\eta_{n},\eta_{n+1}}\rho_{\eta_{n}},~~ r\in[\ell_{\eta_{n}},\ell_{\eta_{n+1}}],\\
&v_{\eta_{n+1},\eta_{n+2}}(r)=0,~~r\in[\ell_{\eta_{n+1}},\ell_{\eta_{n+2}}].
\end{aligned}
\end{equation}
where $\rho_{\eta_{n}}$ is the residual of the infinitesimal displacement at time $\eta_n$. And set
\begin{equation}\label{4.8}
v(r)=\begin{cases}
v_{\eta_n,\eta_{n+1}}(r),
&r\in[\ell_{\eta_n},\ell_{\eta_{n+1}}]~~~~\text{and}~~n\in2\mathbb{N},
\\
v_{\eta_{n+1},\eta_{n+2}}(r),
&r\in[\ell_{\eta_{n+1}},\ell_{\eta_{n+2}}]~~\text{and}~~n\in2\mathbb{N}.
\end{cases}\end{equation}
According to \cite[Lemma 4.1]{Peng-2024}, if we define the direction $v$ as in \eqref{4.8}, then for any $\beta>0$, we have
\begin{equation}\label{4.9}
\rho_{\eta_{n+2}}= J_{\eta_{n+1},\eta_{n+2}}\beta(\mathcal{M}_{\eta_{n},\eta_{n+1}}
+\beta\mathbb{I})^{-1}J_{\eta_{n},\eta_{n+1}}\rho_{\eta_{n}},
\quad\forall~n\in2\mathbb{N}.
\end{equation}

Secondly, we will perform some moment estimates for $\rho_t$, that is,
\textbf{the control of $\rho_{\eta_n}$}.

Having defined the control $v$ and the associated error $\rho$, in \eqref{4.8} and \eqref{4.9} respectively, we can, by employing similar arguments as those in \cite{Peng-2024}, derive the following key decay estimate for $\|\rho\|_{\eta_n}$, this requires the result from Lemma \ref{lemma 2.7}.
\begin{lemma}{\cite[Lemma 4.5]{Peng-2024}}\label{lemma4.5}
  For any $\kappa\in(0,\kappa_0]$, $M>0$, $\gamma_0>0$, there exists a positive constant $\beta=\beta(\kappa,M,\gamma_{0},\nu,\{\alpha_k^m\}_{k\in\mathcal{Z},m\in\{0,1\}},\nu_{S},d)$ such that if we define the direction $v$ according to \eqref{4.7}, then the following estimate holds
\begin{equation}
\begin{aligned}
&\mathbb{E}_{U_{0}}\left[(1+\eta_{2n+2}-\eta_{2n})^{8}
\exp\left\{8 C_{0}\int_{\eta_{2n}}^{\eta_{2n+2}}
\|U_{s}\|_{1}^{4/3}\mathrm{d}s\right\}\|\rho_{\eta_{2n}}\|^{4}
I_{\{\Theta\leq M\}}\right]\\
&\leq C_{\kappa,M,\gamma_{0}}\exp\{4\kappa a\|U_{0}\|^{2}-n\gamma_{0}\},
\end{aligned}
\end{equation}
for every $n\in \mathbb{N}$ and $U_0\in H$, where $a=\frac{1}{1-e^{-1}}$, $C_{\kappa,M,\gamma_0}$ is a positive constant depending
on $\kappa,M,\gamma_{0}$ and $\nu,\{\alpha_{k}^m\}_{k\in\mathcal{Z},m\in\{0,1\}},\nu_{S},d$, $C_0$ is the constant introduced in Lemma \ref{lemma2.4}.
\end{lemma}

Thirdly, we determine
\textbf{the control of $\int_0^{\ell_t}v(s)\mathrm{d}W(s)$}.

For any $M, t > 0$ and $n \in\mathbb{N}$, the aim of this step is to give an estimate for the moment of the stochastic integral:
\begin{equation*}
\mathbb{E}\left[\Big|\int_{\ell_{\eta_{2n}}}^{(\ell_{\eta_{2n+1}}
	\wedge\ell_{t})\vee\ell_{\eta_{2n}}} v(s)\mathrm{d}W(s)\Big|^{2}I_{\{\Theta\leq M\}}\right].
\end{equation*}

We start with an estimate on the moments of $\rho_t$. Similar to the arguments in \cite{Peng-2024}, by Lemma \ref{lemma2.1}, Lemma \ref{lemma4.5} and Lemma \ref{lemma2.10}, we get the following result.

\begin{lemma}\label{lemma4.6}
For any $\kappa \in ( 0, \kappa _{0}] , M> 0, \gamma _{0}> 0$, let $\beta$ and $C_{\kappa,M,\gamma_{0}}$ be the constants chosen according to Lemma \ref{lemma4.5}.  Then,  for any $U_0\in H$, $n\in \mathbb{N}$ and $t\geq 0$,  one has
$$
\mathbb{E}_{U_0}\left[\|\rho_t\|^4I_{\{\Theta\leq M\}}I_{\{t\in[\eta_{2n},\eta_{2n+2})\}}\right]\leq C_{\kappa,M,\gamma_0}\exp\left\{4\kappa a\|U_0\|^2-n\gamma_{0}\right\}.
$$
%where $C_{\kappa,M,\gamma_{0}}$ is a positive constant depending on $\kappa, M, \gamma_{0}$ and $\{\alpha_{k}^{m}\}_{k\in \mathcal{Z},m=\{0,1\}}, \nu_{S}, d$.
\end{lemma}

\begin{lemma}\label{lemma4.7}
For any $M>0$, $\gamma_0>0$, $\kappa\in(0,\kappa_0]$, let $\beta$ and $C_{\kappa,M,\gamma_0}$ be the constants chosen according to Lemma \ref{lemma4.5}. Then, we have
\begin{equation}
\mathbb{E}_{U_{0}}\left[\left|\int_{\ell_{\eta_{2n}}}
^{\ell_{\eta_{2n+1}}}v(s)\mathrm{d}W(s)\right|^{2}I_{\{\Theta\leq M\}}\right]
\leq C_{\kappa,M,\gamma_{0}}\exp\{2\kappa a\|U_{0}\|^{2}-\gamma_{0}n/2\}, \end{equation}
and
\begin{equation}
\mathbb{E}_{U_0}\left[\left|\int_{\ell_{\eta_{2n}}}
^{(\ell_{\eta_{2n+1}}\wedge\ell_t)\vee\ell_{\eta_{2n}}}
v(s)\mathrm{d}W(s)\right|^2I_{\{\Theta\leq M\}}\right]
\leq C_{\kappa,M,\gamma_{0}}\exp\{2\kappa a\|U_{0}\|^{2}-\gamma_{0}n/2\}
\end{equation}
for $n\geq0,t\geq0$ and $U_0\in H$.
%here $C_{\kappa,M,\gamma_0}$ is a positive constant depending on $\kappa,M,\gamma_0$ and $\nu,\{\alpha_{k}^{l}\}_{k\in \mathcal{Z},l=\{0,1\}},\nu_{S},d$.
\end{lemma}

Through the above preparation work, we provide the following e-property.

\begin{proposition}[e-property]\label{propo1.4}  Under Condition \ref{Condition-2.1} and Condition \ref{Condition-2.2}, the Markov semigroup $\left \{ P_{t}\right \} _{t\geq 0}$ has the e-property, i.e., for any bounded and Lipschitz continuous function $\Phi$, $U_0\in H$ and $\varepsilon > 0$, there exists $\delta > 0$ such that
$|P_t\Phi( U_0^{\prime })-P_t\Phi( U_0)|<\varepsilon$, $\forall t\geq 0$ and $U_0^{\prime }$ with $\| U_0^{\prime }- U_0\| < \delta$.
\end{proposition}
\begin{proof}
For the term $\mathcal{Q}_{11}$ in \eqref{5.1}, by Lemma \ref{lemma4.7}, we have
\begin{equation*}
\begin{aligned}
\mathcal{Q}_{11}&\leq\|\Phi\|_{\infty}\sum_{n=0}^{\infty}
\mathbb{E}\left[\Big|\int_{\ell_{\eta_{2n}}}^{(\ell_{\eta_{2n+1}}\wedge\ell_{t})\vee\ell_{\eta_{2n}}}
v(s)\mathrm{d}W_{s}I_{\{\Theta\leq M\}}\Big|\right]
\\
&\leq\|\Phi\|_{\infty}\sum_{n=0}^{\infty}\left(\mathbb{E}
\left[\left|\int_{\ell_{\eta_{2n}}}^{(\ell_{\eta_{2n+1}}
\wedge\ell_{t})\vee\ell_{\eta_{2n}}} v(s)\mathrm{d}W_{s}\right|^{2}I_{\{\Theta\leq M\}}\right]\right)^{1/2}
\\
&\leq\|\Phi\|_{\infty}\sum_{n=0}^{\infty}C_{\kappa,M,\gamma_{0}}\exp\{\kappa a\|U_{0}\|^{2}-\gamma_{0}n/4\}.
\end{aligned}
\end{equation*}
For the term $\mathcal{Q}_{12}$, by Lemma \ref{lemma4.6}, we have
\begin{equation*}
\begin{aligned}
\mathcal{Q}_{12}&\leq\sum_{n=0}^{\infty}\|\nabla \Phi\|_{\infty}\mathbb{E}\big[\|\rho_{t}\|I_{\{\Theta\leq M\}}I_{\{t\in[\eta_{2n},\eta_{2n+2})\}}\big]
\\
&\leq\sum_{n=0}^{\infty}\|\nabla \Phi\|_{\infty}C_{\kappa,M,\gamma_{0}}\exp
\left\{a\kappa\|U_{0}\|^{2}-\gamma_{0}n/4\right\}.
\end{aligned}
\end{equation*}
According to the estimates of $\mathcal{Q}_{11}$ and $\mathcal{Q}_{12}$, for any $\xi$ with $\|\xi\|=1$, one sees that
\begin{equation*}
|\nabla_{\xi}P_{t}^{M}\Phi(U_{0})|\leq C_{\kappa,M,\gamma_{0}}\big(\|\Phi\|_{\infty}+\|\nabla \Phi\|_{\infty}\exp\big\{a\kappa\|U_{0}\|^{2}\big\}\big).
\end{equation*}
Let $\gamma(s)=sU_0+(1-s)U_0'$, by the above inequality, we have
\begin{equation}\label{5.3}
\begin{aligned}
&\mathbb{E}\left[\Phi(U_{t}(U_{0}))I_{\{\Theta\leq M\}}\right]-\mathbb{E}[\Phi(U_{t}(U_{0}^{\prime}))I_{\{\Theta\leq M\}}]\\
&=\int_{0}^{1}\langle\nabla P_{t}^{M}\Phi(\gamma(s)),U_{0}-U_{0}^{\prime}\rangle\mathrm{d}s
\\
&\leq C_{\kappa,M,\gamma_{0}}\|U_{0}-U_{0}^{\prime}\|\big(\|\Phi\|_{\infty}+\|\nabla \Phi\|_{\infty}\sup_{s\in[0,1]}\exp\{a\kappa\|\gamma(s)\|^{2}\}\big).
\end{aligned}
\end{equation}

By the arguments in \cite{Komorowski-2010}, for any bounded and Lipschitz continuous function $\Phi$ on $H$, there exists a sequence $(\Phi_n)$ satisfying $(\Phi_n)\subseteq C_b^1(H)$ and $\lim\limits_{n\rightarrow\infty}\Phi_n(x)=\Phi(x)$ pointwise. Note that $\|\Phi_n\|_{\infty}\leq\|\Phi\|_{\infty}$ and $\|\nabla\Phi_n\|_{\infty}\leq \text{Lip}_\Phi$, where $\text{Lip}_\Phi=\sup\limits_{x\neq y}\frac{|\Phi(x)-\Phi(y)|}{\|x-y\|}$. Therefore, for any $\kappa\in(0,\kappa_0]$, $U_0,U_0' \in B_H(\Upsilon), \Phi\in C_b^1(H)$ and $M\geq 1, t>0$, one arrives that
\begin{equation}\label{5.4}
\begin{aligned}
&|P_{t}\Phi (U_{0})-P_{t}\Phi(U_{0}')|=\lim\limits_{n\to\infty}
|P_{t}\Phi_{n}(U_{0})-P_{t}\Phi_{n}(U_{0}')| \\
\leq&\lim_{n\to\infty}\left[C_{\kappa,M,\gamma_{0}}
\|U_{0}-U_{0}^{\prime}\|\big(\|\Phi_{n}\|_{\infty}+\|\nabla \Phi_{n}\|_{\infty}\exp\{a\kappa\Upsilon^{2}\}\big)+2\|\Phi_{n}\|_{\infty}
\mathbb{P}(\Theta\geq M)\right] \\
\leq& C_{\kappa,M,\gamma_{0}}\|U_{0}-U_{0}^{\prime}\|
(\|\Phi\|_{\infty}+\text{Lip}_\Phi\exp
\left\{a\kappa\Upsilon^{2}\right\})
+2\|\Phi\|_{\infty}\mathbb{P}(\Theta\geq M) \\
:=&J_{1}+J_{2}.
\end{aligned}
\end{equation}
For any $\varepsilon>0$, by \eqref{4.3}, we can find a $M>0$ such that $J_2<\frac{\varepsilon}{2}$. Besides, there exists a $\delta>0$ such that for any $U_0,U_0'\in B_H(\Upsilon)$ and $\|U_0-U_0'\|<\delta$, $J_1<\frac{\varepsilon}{2}$ holds. Then we can conclude that for any positive constants $\varepsilon,\Upsilon$, there exists a $\delta>0$, such that
\begin{equation}
|P_{t}\Phi(U_{0})-P_{t}\Phi(U_{0}^{\prime})|<\varepsilon, ~~\forall t\geq0\mathrm{~and~}U_{0},U_{0}^{\prime}\in B_{H}(\Upsilon)\mathrm{~with~}\|U_{0}-U_{0}^{\prime}\|<\delta.
\end{equation}
We complete the proof of Proposition \ref{propo1.4}.
\end{proof}

\begin{proposition}[Weak  Irreducibility]\label{irre}  For any $\mathcal{C}, \gamma > 0$,  there exists a $T=T(\mathcal{C},\gamma)>0$ such that
\begin{equation*}
\inf\limits_{\|U_0\|\leq\mathcal{C}}P_T(U_0,\mathcal{B}_\gamma)>0,
\end{equation*}
where $\mathcal{B} _\gamma = \{ U\in H:\| U\| \leq \gamma \}$.
\end{proposition}
\begin{proof}
Define $V_t=U_t-\zeta_t$, where $\zeta_t=\sum_{k\in \mathcal{Z},m\in\{0,1\}}\alpha_k^m \sigma_k^m W^{k,m}_{S_t} $, then $V_t$ satisfies
\begin{equation*}
\frac{\partial V_t}{\partial t}
=- A(V_t+\zeta_t)-B(U_t,U_t)+G(U_t)
=- A(V_t+\zeta_t)-B(U_t,V_t+\zeta_t)+G(V_t+\zeta_t).
\end{equation*}
Taking the inner product of this equation with $V_t$ to derive
\begin{equation*}
\begin{aligned}
&\frac {\mathrm{d}}{\mathrm{d}t}\|V_t\|^2 =-\nu\| V_{t}\|_1^{2}+\langle A\zeta_{t},V_{t}\rangle
+\langle B(U_{t},\zeta_{t}),V_{t}\rangle+\langle G(V_t+\zeta_t),V_t\rangle  \\
&\leq -\nu\| V_t\|_1^2+C_1\|V_t\|\| A\zeta_t\|+C_1\|V_t\|\|\zeta_t\|_2\|U_t\|_1+C_2(\| V_t\|_1+\|\zeta_t\|_1)\|V_t\| \\
&=-\nu\| V_t\|_1^2+C_1\|V_t\|\|\zeta_t\|_2+C_1\|V_t\|\|\zeta_t\|_2\|V_t+\zeta_t\|
+C_2(\| V_t\|_1+\|\zeta_t\|_1)\|V_t\|\\
&\leq-\nu\| V_t\|_1^2+C_1\|V_t\|\|\zeta_t\|_2+C_1\|V_t\|\|\zeta_t\|_2\|V_t\|
+C_1\|V_t\|\|\zeta_t\|_2\|\zeta_t\|
+C_2\| V_t\|_1\|V_t\|+C_2\|\zeta_t\|_1\|V_t\|\\
&\leq-\frac{\nu}{2}\| V_{t}\|_1^{2}+\frac{8C_{1}^{2}}{\nu}\|\zeta_{t}\|_2^{2}
+C_{1}\|V_{t}\|^{2}\|\zeta_{t}\|_2+\frac{8C_{1}^{2}}{\nu}\|\zeta_{t}\|_2^{2}
\|\zeta_{t}\|^{2}+\frac{8C_{2}^{2}}{\nu}\|V_t\|^2
+\frac{8C_{2}^{2}}{\nu}\|\zeta_t\|_1^2,
\end{aligned}
\end{equation*}
where $C_1, C_2>0$ are constants that depending on $\nu$ and $g$, respectively.

For any $T,\delta>0$, we define
\begin{equation*}
\begin{array}{rcl}
\Omega'(\delta,T)&=&\left\{f=(f_{s})_{s\in[0,T]}\in D([0,T];H):\sup\limits_{s\in[0,T]}\|\nabla f_{s}\|\leq\min\{\delta,\frac{\nu}{4C_{1}}-\frac{8C_2^2}{ \nu C_1}\}\right\}.
\end{array}
\end{equation*}
If $\zeta\in \Omega'$, then there exists a constant $C_3$ depending only on the domain so that
\begin{equation*}
\|V(t)\|^2\leq\|V(0)\|^2e^{-\frac{\nu}{2}t}
+\frac{C_3}{\nu}\left[\min\left(\delta,\frac{\nu}{4C_{1}}-\frac{8C_2^2}{ \nu C_1}\right)^4+\min\left(\delta,\frac{\nu}{4C_{1}}-\frac{8C_2^2}{ \nu C_1}\right)^2\right].
\end{equation*}
Let $\mathcal{C}$ and $\gamma$ be given as in the statement of Proposition \ref{irre}. As $\|U_0\|\leq \mathcal{C}$, there exists a $T>0$ and a $\delta>0$ such that
$$
\|V_T\|\leq\frac{\gamma}{2}~\text{and}~\delta\leq\frac{\gamma}{2}.
$$
To summarize, one concludes that
\begin{equation}\label{5.5}
\|U_0\|\leq \mathcal{C}\mathrm{~and~}\zeta\in\Omega^{\prime}(\delta,T)
\Rightarrow\|U_T\|\leq\|V_T\|+\|\zeta_T\|\leq\gamma.
\end{equation}
Notice that for any $T>0$, $\varepsilon>0$ and non-zero real numbers $\alpha_k^m,k\in\mathcal{Z},m\in\{0,1\}$, there holds the following result
\begin{equation*}
\mathbb{P}(\sup_{t\in[0,T]}\|\sum_{k\in\mathcal{Z},m\in\{0,1\}}\alpha_k^m \sigma_k^m W_{S_t}^{k,m} \|<\varepsilon)
\geq p_0>0,
\end{equation*}
where $p_0=p_0(T,\varepsilon,\{\alpha_k^m\}_{k\in\mathcal{Z},m\in\{0,1\}})$.
Combining the above inequality with \eqref{5.5}, we complete the proof.
\end{proof}

We finally give the proof of our main result, which concludes that statistically invariant states of \eqref{2.1} are unique.

\begin{proof}[Proof of Theorem \ref{ergodicity}]
We first prove the existence. By Lemma \ref{lemma2.1}, it holds that
\begin{equation*}
\nu \mathbb{E}\left[\int_{0}^{t}\|U_{s}\|_{1}^{2}\mathrm{d}s\right]
\leq\|U_{0}\|^{2}+C_{2}t,\quad\forall t\geq0.
\end{equation*}
Building upon the arguments presented in the proof by \cite{Zheng-2013} and utilizing the Krylov-Bogoliubov criterion, we can obtain the existence of an invariant measure.

Now we prove the uniqueness. Assume that there are two distinct invariant
probability measures $\mu_1$ and $\mu_2$ for $\{P_t\}_{t\geq0}$. By Proposition \ref{propo1.4} and \cite{Kapica-2012}, we can get
\begin{equation}\label{supp}
\mathrm{Supp~}\mu_1\cap\mathrm{Supp~}\mu_2=\emptyset.
\end{equation}

On the other hand, by Lemma \ref{lemma2.1}, for every invariant measure $\mu$, there holds
\begin{equation*}
\int_H\|U\|^2\mu(\mathrm{d}U)\leq C.
\end{equation*}
Following the arguments in the proof of \cite[Corollary 4.2]{MH-2006} and using Proposition \ref{irre}, for every invariant measure $\mu$, we have
$0\in \mathrm{Supp~}\mu$. This contradicts \eqref{supp}. We complete the proof of uniqueness, which implies the ergodicity.
\end{proof}

%%%%%%%%%%%%%%%%%%%%%%%%%%%%%%%%%%%%%%%%%%%%%%%%%%%%%

\subsection*{Acknowledgments}
The authors thank the editor and the anonymous reviewers for their helpful comments and suggestions. 

\subsection*{Funding}
Jianhua Huang is supported in part
by NNSFC (No.12371198, 12031020), Scientific Research Program Foundation of National University of Defense Technology (No.22-ZZCX-016), Hunan Provincial Natural Science Foundation of China (No.2024JJ5400).
Xuhui Peng is supported in part by NNSFC (No. 12471128, 12071123), China Scholarship Council (No.202406720025), the science and technology innovation Program of 
Hunan Province (No.2022RC1189).
%
%\subsection*{Conflict of interest}
%This work does not have any conflicts of interest.
%
%\subsection*{Availability of date and materials}
%Not applicable.

\newcounter{cankan}

\end{document}